\documentclass{amsart}
\usepackage[utf8]{inputenc}

\usepackage[margin=1in]{geometry}
\usepackage[expansion=false]{microtype}

\usepackage{amsmath, amsfonts, amssymb, amsthm, amsopn, amscd}
\usepackage{eucal}

\usepackage{booktabs,enumitem}
\setlist[enumerate]{label=(\roman*),font=\normalfont}

\usepackage{tikz}
\usetikzlibrary{matrix}
\usepackage{tikz-cd}

\usepackage[pdfusetitle,unicode,bookmarksopen]{hyperref}

\makeatletter
\def\new@part{\@startsection{part}{0}%
  \z@{\linespacing\@plus\linespacing}{.5\linespacing}%
  {\normalfont\large\bfseries\raggedright}}
\def\part#1{%
{\def\@secnumfont{\bfseries}%
\new@part{#1}%
}}
\makeatletter

\numberwithin{equation}{section}

\theoremstyle{plain}
\newtheorem*{theorem*}{Theorem}
\newtheorem{theorem}[equation]{Theorem}
\newtheorem{proposition}[equation]{Proposition}
\newtheorem{lemma}[equation]{Lemma}
\newtheorem{corollary}[equation]{Corollary}
\newtheorem{conjecture}[equation]{Conjecture}

\theoremstyle{definition}
\newtheorem{definition}[equation]{Definition}
\newtheorem{construction}[equation]{Construction}
\newtheorem{example}[equation]{Example}

\newtheorem{observation}[equation]{Observation}

\newtheorem{digression}[equation]{Digression}

\newtheorem{remark}[equation]{Remark}
\newtheorem{warning}[equation]{Warning}

\let\scr=\mathcal
\let\bb=\mathbb
\let\rm=\mathrm
\def\N{\bb N}
\def\Z{\bb Z}

\def\Q{\bb Q}
\def\C{\bb C}
\def\F{\bb F}
\def\A{\bb A}
\def\P{\bb P}
\def\V{\bb V}
\def\1{\mathbf 1}
\def\h{\mathrm h}

\def\G{\mathbb G}
\def\ph{\mathord-}
\makeatletter
\def\pt{{\mathpalette\pt@{.75}}}
\def\pt@#1#2{\mathord{\scalebox{#2}{$\m@th#1\bullet$}}}
\makeatother

\def\L{\mathrm{L}{}}
\def\LB{\mathrm{LB}{}}

\let\into=\hookrightarrow
\let\onto=\twoheadrightarrow
\def\simto{\xrightarrow{\sim}}

\let\emptyset=\varnothing

\DeclareMathOperator{\Sym}{Sym}

\def\id{\mathrm{id}}

\def\Hom{\mathrm{Hom}}

\def\End{\mathrm{End}}
\def\Aut{\mathrm{Aut}}
\def\Map{\mathrm{Map}}

\DeclareMathOperator{\Spec}{Spec}
\DeclareMathOperator{\Spf}{Spf}
\DeclareMathOperator{\Proj}{Proj}
\def\Th{\mathrm{Th}}
\def\Pic{\mathrm{Pic}}

\def\HH{\mathrm{H}}

\def\E{\mathbb{E}}

\def\THH{\mathrm{THH}}
\def\SSeq{\mathrm{SSeq}}
\def\hatotimes{\mathbin{\widehat\otimes}}
\def\adams#1{{\smash[t]{(#1)}}}

\DeclareMathOperator{\cofib}{cofib}

\DeclareMathOperator{\rk}{rk}
\def\Nis{\mathrm{Nis}}
\def\Zar{\mathrm{Zar}}

\def\et{\mathrm{\acute et}}

\DeclareMathOperator{\Tr}{Tr}

\def\QCoh{\mathrm{QCoh}{}}

\def\Bl{\mathrm{Bl}}

\let\cat=\mathrm

\def\Gr{\mathrm{Gr}{}}

\def\MU{\mathrm{MU}}

\def\MGL{\mathrm{MGL}}
\def\MGr{\mathrm{MGr}}
\def\PMGL{\mathrm{PMGL}}
\def\KGL{\mathrm{KGL}}
\def\KU{\mathrm{KU}}
\def\K{\mathrm{K}{}}

\def\Mod{\cat{M}\mathrm{od}{}}
\def\Sch{\cat{S}\mathrm{ch}{}}

\def\Sm{{\cat{S}\mathrm{m}}}

\def\Fin{\cat F\mathrm{in}}
\def\op{\mathrm{op}}

\def\gr{\mathrm{gr}}
\def\Fil{\mathrm{Fil}}
\def\Vect{\cat{V}\mathrm{ect}{}}

\def\Pair{\mathrm{Pair}}

\def\MS{\mathrm{MS}}

\def\Sp{\mathrm{Sp}}

\def\CAlg{\mathrm{CAlg}{}}

\def\ev{\mathrm{ev}}

\def\setminus{-}

\def\rig{\mathrm{rig}}

\def\Cat{\mathrm{C}\mathrm{at}{}}

\def\Ab{\mathrm{A}\mathrm{b}}

\def\Tw{\mathrm{Tw}{}}
\def\fp{\mathrm{fp}}
\def\Fun{\mathrm{Fun}}

\def\all{\mathrm{all}}

\def\Span{\mathrm{Span}}

\def\Perf{\mathrm{Perf}{}}

\def\epi{\mathrm{epi}}

\def\sbu{\mathrm{sbu}}
\def\ebu{\mathrm{ebu}}

\def\sncd{\mathrm{sncd}}
\def\gys{\mathrm{gys}{}}

\newcommand{\SigmaP}{\Sigma_{\mathbb{P}^1}}

\newcommand{\CC}{\mathcal{C}}
\newcommand{\DD}{\mathcal{D}}
\newcommand{\map}{\operatorname{map}}
\newcommand{\QSyn}{\mathrm{QSyn}}
\newcommand{\TC}{\mathrm{TC}}
\newcommand{\TP}{\mathrm{TP}}
\newcommand{\HC}{\mathrm{HC}}
\newcommand{\HP}{\mathrm{HP}}
\newcommand{\crys}{\mathrm{crys}}
\newcommand{\kgl}{\mathrm{kgl}}
\newcommand{\syn}{\mathrm{syn}}
\newcommand{\tc}{\mathrm{tc}}
\newcommand{\tp}{\mathrm{tp}}
\newcommand{\thh}{\mathrm{thh}}

\newcommand{\lax}{\mathrm{lax}}

\usepackage{relsize}
\usepackage[bbgreekl]{mathbbol}
\DeclareSymbolFontAlphabet{\mathbb}{AMSb} 
\DeclareSymbolFontAlphabet{\mathbbl}{bbold}
\newcommand{\prism}{{\mathbbl{\Delta}}}

\def\w{\bar{\mathrm{h}}}
\def\fg{\mathrm{fg}}

\def\lisse{\mathrm{lisse}}
\def\dual{\mathrm{dual}}
\let\heart=\heartsuit

\let\lim=\relax
\DeclareMathOperator*{\lim}{lim}

\DeclareMathOperator*{\colim}{colim}

\let\phi=\varphi
\let\epsilon=\varepsilon

\title{Atiyah duality for motivic spectra}
\date{\today}

\author{Toni Annala}
\address{Department of Computer and Mathematical Sciences\\
University of Toronto Scarborough\\
1265 Military Trail\\
Toronto ON M1C 1A4\\
Canada}
\email{\href{mailto:toni.annala@utoronto.ca}{toni.annala@utoronto.ca}}
\urladdr{\url{https://tannala.com}}
\thanks{This article was written when T.A. was in residence at the Institute for Advanced Study in Princeton}

\author{Marc Hoyois}
\address{Fakultät für Mathematik\\
Universität Regensburg\\
Universitätsstr. 31\\
93040 Regensburg\\
Germany}
\email{\href{mailto:marc.hoyois@ur.de}{marc.hoyois@ur.de}}
\urladdr{\url{https://hoyois.app.uni-regensburg.de}}
\thanks{M.H.\ was partially supported by the Collaborative Research Center SFB 1085 \emph{Higher Invariants} funded by the DFG}

\author{Ryomei Iwasa}
\address{Department of Mathematical Sciences\\
University of Copenhagen\\
Universitetsparken 5\\
DK-2100 Copenhagen Ø\\
Denmark}
\email{\href{mailto:ryomei@math.ku.dk}{ryomei@math.ku.dk}}
\urladdr{\url{https://ryomei.com}}
\thanks{R.I.\ was supported by the European Union's Horizon 2020 research and innovation programme under the European Research Council (ERC) grant agreement No. 101001474}

\begin{document}
	
\maketitle

\begin{abstract}
	We prove that Atiyah duality holds in the $\infty$-category of non-$\A^1$-invariant motivic spectra over arbitrary derived schemes: every smooth projective scheme is dualizable with dual given by the Thom spectrum of its negative tangent bundle. The Gysin maps recently constructed by L.~Tang are a key ingredient in the proof. We then present several applications. First, we study $\A^1$-colocalization, which transforms any module over the $\A^1$-invariant sphere into an $\A^1$-invariant motivic spectrum without changing its values on smooth projective schemes. This can be applied to all known $p$-adic cohomology theories and gives a new elementary approach to ``logarithmic'' or ``tame'' cohomology theories; it recovers for instance the logarithmic crystalline cohomology of strict normal crossings compactifications over perfect fields and shows that the latter is independent of the choice of compactification. Second, we prove a motivic Landweber exact functor theorem, associating a motivic spectrum to any graded formal group law classified by a flat map to the moduli stack of formal groups. Using this theorem, we compute the ring of $\P^1$-stable cohomology operations on the algebraic K-theory of qcqs derived schemes, and we prove that rational motivic cohomology is an idempotent motivic spectrum.
\end{abstract}

\tableofcontents

\section{Introduction}

This paper is a sequel to \cite{AHI}, where we introduced the stable $\infty$-category $\MS_S$ of motivic spectra over a derived scheme $S$.
Our goal is to prove the following theorem and discuss some applications:

\begin{theorem}\label{thm:main-intro}
	Let $f\colon X\to S$ be a smooth projective morphism between derived schemes.
	\begin{enumerate}
		\item \textnormal{(Ambidexterity, Theorem~\ref{thm:ambidexterity})} There is a canonical isomorphism of functors
		\[
		f_\sharp\Sigma^{-\Omega_f} \simeq f_*\colon \MS_X\to \MS_S.
		\]
		\item \textnormal{(Atiyah duality, Corollary~\ref{cor:atiyah})} For every $\xi\in\K(X)$, the Thom spectrum $\Th_X(\xi)\in\MS_S$ is dualizable with dual $\Th_X(-\xi-\Omega_f)$.
	\end{enumerate}
\end{theorem}

In $\A^1$-homotopy theory, this ambidexterity theorem was announced by Voevodsky \cite{Voevodsky-6FF} and proved independently by Ayoub \cite{Ayoub} and Röndigs \cite{RondigsSH}. This theorem is one of two key ingredients in the construction of the six-functor formalism for $\A^1$-invariant motivic spectra, the other one being the localization theorem of Morel and Voevodsky. While the latter does not hold as is without $\A^1$-invariance, Theorem~\ref{thm:main-intro} is hopefully a first major step towards a six-functor formalism for some theory of non-$\A^1$-invariant motivic spectra.

It follows from Atiyah duality that any cohomology theory satisfying the Künneth formula and representable in $\MS_S$ satisfies Poincaré duality.
This unifies essentially all known instances of Poincaré duality for cohomology theories of schemes, including non-$\A^1$-invariant cases such as de Rham cohomology by Clausen \cite{Cla21} and relative prismatic cohomology by Tang \cite{Tan22}.

\textbf{Part I} is dedicated to the proof of Theorem~\ref{thm:main-intro}.
We crucially make use of the \emph{Gysin maps} in $\MS_S$ constructed by Longke Tang \cite{Tang}: for a closed immersion $Z\into X$ between smooth $S$-schemes, the associated Gysin map is a canonical map
\[
\gys\colon X_+ \to \Th_Z(\scr N_{Z/X})
\]
in $\MS_S$, whose $\A^1$-localization recovers the purity isomorphism $\L_{\A^1}(X/(X-Z))\simeq \L_{\A^1}\Th_Z(\scr N_{Z/X})$ of Morel and Voevodsky.\footnote{Note that the purity isomorphism does not hold prior to $\A^1$-localization. In fact, a motivic spectrum in $\MS_S$ satisfies purity if and only if it is $\A^1$-invariant; see Remark~\ref{rmk:purity}.} Using the Gysin map for the diagonal embedding $X\into X\times_S X$ of a smooth separated $S$-scheme $X$, we can define in a standard way a canonical pairing
\[
\ev_{X,\xi}\colon \Th_X(\xi)\otimes \Th_X(-\xi-\Omega_{X/S}) \to \1_S
\]
in $\MS_S$, which we investigate in Section~\ref{sec:eval}. A more precise statement of Atiyah duality is that the pairing $\ev_{X,\xi}$ is perfect when $X$ is smooth and projective.

Similarly to the proofs of Ayoub and Röndigs in $\A^1$-homotopy theory, we reduce the ambidexterity theorem to the case of projective space $\P^n_S\to S$, where we then argue by induction on $n$. 
The heart of the proof is the induction step, which is carried out in Section~\ref{sec:Pdual}. In more details, the structure of the proof of Theorem~\ref{thm:main-intro} is as follows:
\begin{itemize}
	\item We prove Atiyah duality for $X=\P^n_S$ and $\xi=\scr O(-1)^m$ by induction on $n$ (Theorem~\ref{thm:dualsofP}).
	\item Atiyah duality for $X=\P^n_S$ and $\xi=0$ implies ambidexterity for $f\colon \P^n_S\to S$ (Proposition~\ref{prop:dual->ambi}).
	\item Ambidexterity for $\P^n$ implies ambidexterity for all smooth projective morphisms (Theorem~\ref{thm:ambidexterity}).
	\item Finally, ambidexterity for $f\colon X\to S$ implies Atiyah duality for any $\xi\in\K(X)$ (Corollary~\ref{cor:atiyah}).
\end{itemize}

\begin{remark}
	In \cite[Section 5.3]{hoyois-sixops}, a different proof of ambidexterity for $f\colon \P^n_S\to S$ is given, which does not proceed by induction on $n$ but instead uses a direct geometric construction of the unit map $\eta_f\colon \id\to f_\sharp\Sigma^{-\Omega_f}f^*$ (called the Pontryagin–Thom collapse map in \emph{loc.\ cit.}).
	This geometric construction strongly relies on $\A^1$-homotopy invariance, so it is not clear how to adapt it to $\MS_S$.
\end{remark}

\begin{remark}
	In $\A^1$-homotopy theory, the ambidexterity theorem was generalized to smooth \emph{proper} morphisms by Cisinski and Déglise in \cite{CD}.
	To obtain such a generalization in our setting requires further developments towards the six-functor formalism. If $k$ is a field of characteristic $0$, one can combine Theorem~\ref{thm:main-intro}(ii), Chow's lemma, resolution of singularities, and weak factorization to prove that $\Th_X(\xi)$ is dualizable in $\MS_k$ for any smooth proper $k$-scheme $X$ and any $\xi\in\K(X)$.
\end{remark}

In \textbf{Part II}, we discuss some applications of Atiyah duality. 
As one of our main applications, we study \emph{$\A^1$-colocalization} in Section~\ref{sec:logarithmic}, by which we mean the \emph{right} adjoint to the inclusion of the subcategory $\MS_S^{\smash[t]{\A^1}}$ of $\A^1$-invariant motivic spectra (which is the usual stable motivic homotopy $\infty$-category of Morel and Voevodsky).
More specifically, we consider the functor
\[
\Mod_{\1_{\A^1}}(\MS_S) \to \MS_S^{\A^1}, \quad E\mapsto E^\dagger,
\]
which is right adjoint to the inclusion, where $\1_{\A^1}=\L_{\A^1}(\1)$ is the $\A^1$-invariant motivic sphere.
Atiyah duality implies that the $\A^1$-colocalization $E^\dagger$ of a $\1_{\A^1}$-module $E$ has the same values as $E$ on smooth projective schemes. Moreover, its values on more general smooth schemes can sometimes be computed by means of a suitable compactification:

\begin{theorem}[Computing the $\A^1$-colocalization]
	\label{thm:intro-A1-coloc}
	Let $S$ be a derived scheme and let $E\in \Mod_{\1_{\A^1}}(\MS_S)$.
	\begin{enumerate}
		\item \textnormal{(Proposition \ref{prop:coloc})} For any smooth projective $S$-scheme $X$ and any $\xi\in \K(X)$, the counit map
		\[
		E^\dagger(\Th_X(\xi)) \to E(\Th_X(\xi))
		\]
		is an isomorphism.
		\item \textnormal{(Proposition \ref{prop:log})} Let $U$ be a smooth $S$-scheme admitting a smooth projective compactification $U\into X$ with a relative strict normal crossings boundary $\partial X=X- U$. Let $\partial_1X,\dotsc,\partial_nX$ be the smooth components of $\partial X$, and for any subset $I\subset \{1,\dotsc,n\}$, let $\partial_IX=\bigcap_{i\in I}\partial_iX$ and let $i_I\colon \partial_IX\into X$ be the inclusion.
		 Then, for any $\xi\in\K(X)$, one may compute $E^\dagger(\Th_U(\xi))$ as the total cofiber of an $n$-cube $I\mapsto E(\Th_{\partial_IX}(\xi+\scr N_{i_I}))$, whose edges are Gysin maps.
	\end{enumerate}
\end{theorem}

For this theorem to be useful, we need to know some examples of $\1_{\A^1}$-modules. 
A large supply of $\1_{\A^1}$-modules comes from the fact that algebraic K-theory is $\A^1$-invariant on regular noetherian schemes.
As recently proved by Bachmann \cite{Bac22}, Voevodsky's slice filtration of the motivic spectrum $\KGL$ over a Dedekind domain $D$ induces a multiplicative filtration of the algebraic K-theory of smooth $D$-schemes with associated graded given by the Bloch–Levine motivic complexes; this is the \emph{motivic filtration} $F^*\K$ of algebraic K-theory. This filtration is represented by the motivic spectrum $\kgl\in \MS_D$ in the sense that $F^n\K=\Omega^{\infty-n}_{\P^1}\kgl$ for all $n\in \Z$. Moreover, the structure maps $F^{n}\K\to F^{n-1}\K$ are induced by the multiplication by the Bott element $\beta\colon\P^1\to \kgl$, and we have
\[
\kgl[\beta^{-1}]=\KGL\quad\text{and}\quad \kgl/\beta=\HH\Z.
\]
In a similar way, the motivic filtration of $p$-complete topological cyclic homology $\TC_p$ defined by Bhatt, Morrow, and Scholze in \cite{BMS} is represented by a motivic spectrum $\tc_p\in \MS_S$ over any derived scheme $S$, with a Bott element $\beta\colon \P^1\to\tc_p$ such that
\[
\tc_p[\beta^{-1}]=\TC_p\quad\text{and}\quad \tc_p/\beta=\HH\Z_p^\syn.
\]
Over a Dedekind domain $D$, we show that the cyclotomic trace $\K\to \TC_p$ is compatible with these motivic filtrations (Proposition~\ref{prop:motfil}), which gives rise to a morphism of motivic $\E_\infty$-ring spectra $\kgl\to \tc_p$ in $\MS_D$.

\begin{theorem}[Examples of $\1_{\A^1}$-modules]
	\label{thm:intro-A1-modules}
	\leavevmode
	\begin{enumerate}
		\item \textnormal{(Localizing invariants, Example~\ref{ex:locinv})} Let $E$ be a localizing invariant of $\Z$-linear $\infty$-categories and let $E_S\in \MS_S$ be the associated motivic spectrum over a qcqs derived scheme $S$. If $S$ is regular noetherian or if $E$ is truncating, then $E_S$ is a $\1_{\A^1}$-module.
		\item \textnormal{(Rational orientable ring spectra, Example~\ref{ex:Q-orientable})} If $S$ is regular noetherian, then any rational orientable commutative ring in $\h\MS_S$ is a $\1_{\A^1}$-module.
		\item \textnormal{($p$-adic étale cohomology theories, Corollary~\ref{cor:tc}, Examples \ref{ex:syn} and \ref{ex:prism})} Let $S$ be $p$-completely smooth over a Dedekind domain and let $(A,I)$ be a prism with a map $S^\wedge_p\to\Spf(A/I)$. Then the following motivic spectra are $\E_\infty$-algebras over $\kgl$ and hence over $\1_{\A^1}$ in $\MS_S$:
		\[
		\tc_p,\,\tp_p,\,\tc_p^-,\,\thh_p,\,
		\TC_p,\,\TP_p,\,\TC_p^-,\,\THH_p,\,
		\HH\Z_p^\syn,\,\HH\Z_p^\prism,\, \HH A^\prism.
		\]
		Here, the motivic spectra $\tc_p$, $\tp_p$, $\tc_p^-$, and $\thh_p$ represent the Bhatt–Morrow–Scholze motivic filtrations of $\TC_p$, $\TP_p$, $\TC_p^-$, and $\THH_p$, respectively, $\HH\Z_p^\syn=\tc_p/\beta$ represents syntomic cohomology (of $p$-adic formal schemes), $\HH\Z_p^\prism=\tp_p/\beta$ represents absolute prismatic cohomology, and $\HH A^\prism$ represents prismatic cohomology relative to $(A,I)$.
		
		In particular, if $k$ is a perfect field of characteristic $p$, then the crystalline cohomology motivic spectrum $\HH W(k)^\crys=\HH W(k)^\prism$ is an $\E_\infty$-algebra over $\1_{\A^1}$ in $\MS_k$.
	\end{enumerate}
\end{theorem}

\begin{remark}\label{rmk:1=1A1}
	Given that almost all examples of cohomology theories are provably or conjecturally $\1_{\A^1}$-modules over regular noetherian schemes, one may reasonably ask if the motivic sphere itself is $\A^1$-invariant over such bases.
	While an answer to this question is probably out of reach in this generality, the case of a base field seems approachable, especially in characteristic zero.
	One can also enforce a positive answer by replacing $\MS_S$ with $\Mod_{\1_{\A^1}}(\MS_S)$ without any obvious undesirable side effects (at least when $S$ is a field or a Dedekind domain).
\end{remark}

Theorem~\ref{thm:intro-A1-coloc}(ii) suggests that $E^\dagger(U)$ may be interpreted as the ``logarithmic $E$-cohomology'' of the pair $(X,\partial X)$, but it is defined without reference to a choice of compactification (which may not even exist in general). We make this interpretation more precise for crystalline cohomology by comparing the $\A^1$-colocalization of $\HH W(k)^\crys$ with the logarithmic crystalline cohomology defined by Kato~\cite{Kat}:

\begin{theorem}[Comparison with rigid and logarithmic crystalline cohomology, Proposition~\ref{prop:crys}]
	\label{thm:intro-crystalline}
	 Let $k$ be a perfect field of characteristic $p>0$.
	 \begin{enumerate}
	 \item The motivic spectrum $\HH W(k)^{\crys,\dagger}[1/p]$ represents Berthelot's rigid cohomology.
	 \item For a smooth $k$-scheme $U$, suppose that there is a smooth projective compactification $X$ such that $\partial X=X\setminus U$ is a strict normal crossings divisor.
	 Then there is a $W(k)$-linear isomorphism
	 \[
	 	\HH W(k)^{\crys,\dagger}(U) \simeq \mathrm{R}\Gamma_\crys((X,\partial X)/W(k)) \rlap,
	 \]
	 where the right hand side is logarithmic crystalline cohomology; in particular, the latter does not depend on the choice of compactification.
	 \end{enumerate}
\end{theorem}

The question of independence of the choice of compactification has been open since the introduction of logarithmic crystalline cohomology.
Mokrane \cite{Mok} proved it under resolution of singularities.

In Section~\ref{sec:lisse}, we show that the algebraic cobordism spectrum $\MGL$ can be built in a familiar way using Grassmannians (which is less obvious than in $\A^1$-homotopy theory). Atiyah duality then implies that $\MGL$ is \emph{lisse}, i.e., a colimit of a dualizable objects in $\MS_S$. This in turn implies a homological version of the Conner–Floyd isomorphism of \cite{AHI}:

\begin{theorem}[Homological Conner–Floyd isomorphism, Corollary~\ref{cor:conner-floyd}]
	\label{thm:intro-conner-floyd}
	Let $S$ be a qcqs derived scheme. There is an isomorphism of bigraded multiplicative homology theories
	\[
	\MGL_{**}(\ph)\otimes_\L\Z[\beta^{\pm 1}]\simeq \KGL_{**}(\ph)\colon \MS_S\to \Ab^{\Z\times\Z}.
	\]
\end{theorem}

The fact that $\MGL$ is lisse further allows us to adapt the proof of the motivic Landweber exact functor theorem of Naumann–Spitzweck–Østvær \cite{Naumann:2009} to our non-$\A^1$-invariant setting, which is the content of Section~\ref{sec:LEFT}. This theorem associates to certain graded modules $M$ over the Lazard ring $\L$ a motivic spectrum $\Phi(M)\in\MS_S$ with an isomorphism of homology theories $\Phi(M)_{**}(\ph)\simeq \MGL_{**}(\ph)\otimes_\L M$.
For example, Theorem~\ref{thm:intro-conner-floyd} says that the motivic K-theory spectrum $\KGL$ is obtained in this way from the multiplicative formal group law $x+y-\beta xy$ on $\Z[\beta^{\pm 1}]$. It turns out that this construction $\Phi$ has good functorial and multiplicative properties, at least modulo phantom maps, which is enough for some applications.
To succintly formulate these properties, we consider on the one hand a certain symmetric monoidal $1$-category $\Mod_{\smash[b]{\fg}}^{\flat,+}$ of flat modules over the moduli stack $\scr M_{\smash[b]{\fg}}$ of formal groups, and on the other hand the symmetric monoidal $1$-category $\w\MS_S^\lisse$ whose objects are lisse motivic spectra and whose morphisms are those of $\h\MS_S^\lisse$ modulo phantom maps.

\begin{theorem}[Motivic Landweber exact functor theorem, Theorem~\ref{thm:LEFT}]
	\label{thm:intro-LEFT}
	There is for any qcqs derived scheme $S$ a symmetric monoidal functor
	\[
	\Phi\colon \Mod_\fg^{\flat,+}\to \w\MS_S^\lisse,
	\]
	natural in $S$, such that for any $\Z$-graded $\L$-module $M$ that is flat over $\scr M_\fg$, we have an isomorphism
	\[
	\Phi(M)_{**}(\ph)\simeq \MGL_{**}(\ph)\otimes_\L M\colon \MS_S\to \Ab^{\Z\times\Z}.
	\]
\end{theorem}

In Section~\ref{sec:HQ}, we use Theorems \ref{thm:intro-conner-floyd} and~\ref{thm:intro-LEFT} to study $\P^1$-stable operations in algebraic K-theory. We explicitly compute the bigraded endomorphism ring $\KGL^{**}_\Lambda\KGL$ for any subring $\Lambda\subset\Q$ in terms of the algebraic K-theory of the base scheme and the automorphisms of the multiplicative formal group (Proposition~\ref{prop:End(KGL)}). When $\Lambda=\Q$, we recover the idempotents defined by Riou in $\A^1$-homotopy theory \cite[Section 5]{RiouK}, which induce a decomposition of $\KGL_\Q$ into eigenspaces of the Adams operations. Along the way, we generalize to derived schemes a theorem of Soulé about the finiteness of the Adams decomposition (Proposition~\ref{prop:Adams-dec}(ii)). 
We then define the motivic spectrum $\HH\Q$ over any derived scheme as the $0$th Adams eigenspace of $\KGL_\Q$; it is stable under arbitrary base change and represents Beilinson's rational motivic cohomology \cite[Section 2.2]{BeilinsonHQ}.
Over a regular noetherian scheme, $\HH\Q$ is $\A^1$-invariant and coincides with the motivic spectrum introduced by Riou and further investigated by Cisinski and Déglise in \cite[Section 14]{CD}.
Finally, as another application of Theorem~\ref{thm:intro-LEFT}, we prove:

\begin{theorem}
	\label{thm:intro-HQ}
	Let $S$ be a derived scheme.
	\begin{enumerate}
		\item \textnormal{(Idempotence of $\HH\Q$, Theorem~\ref{thm:HQ})} The rational motivic cohomology spectrum $\HH\Q\in \MS_S$ is an idempotent $\E_\infty$-algebra.
		\item \textnormal{(Characterization of $\HH\Q$-modules, Proposition~\ref{prop:HQ-modules})} A $\Q$-linear motivic spectrum lies in the image of the fully faithful embedding $\Mod_{\HH\Q}(\MS_S)\into \MS_S$ if and only if it admits a structure of $\MGL$-module in the homotopy category $\h\MS_S$.
	\end{enumerate}
\end{theorem}

Note that even though $\HH\Q$ is the usual $\A^1$-invariant rational motivic cohomology spectrum when $S$ is regular noetherian, its idempotence in $\MS_S$ is a stronger statement than its idempotence as an $\A^1$-invariant motivic spectrum. 
A natural question that we do not pursue in this paper is whether Morel's isomorphism $\HH\Q\simeq \1_\Q^+$ holds in $\MS_S$. Here, $\1_\Q=\1_\Q^+\times\1_\Q^-$ is the decomposition of the rational motivic sphere into the eigenspaces of the swap automorphism of $\P^1\otimes\P^1$.
We expect this to be true in general and record it as a conjecture:

\begin{conjecture}
	For any derived scheme $S$, we have $\HH\Q\simeq \1_\Q^+$ in $\MS_S$.
\end{conjecture}

By Theorem~\ref{thm:intro-HQ}(ii), this conjecture is equivalent to the orientability of the motivic ring spectrum $\1_\Q^+$.
Establishing this conjecture would be a first step towards answering the question of the $\A^1$-invariance of the motivic sphere raised in Remark~\ref{rmk:1=1A1}.

\subsection*{Conventions and notation}
We generally use the same notation as in \cite{AHI}, except that we use the Nisnevich-local version of the $\infty$-category of motivic spectra:
\[
\MS_S=\Sp_{\P^1}(\scr P_{\Nis,\ebu}(\Sm_S,\Sp)).
\]
Thus, the Morel–Voevodsky stable motivic homotopy $\infty$-category is the full subcategory of $\MS_S$ spanned by the $\A^1$-invariant objects.
Moreover, by \cite[Proposition 2.2]{AHI}, any motivic spectrum satisfies smooth blowup excision.\footnote{In fact, we will never directly use Nisnevich descent, and all results in this paper remain valid if we define $\MS_S$ using $\scr P_{\Zar,\sbu}$; the proof of ambidexterity for projective spaces even works using $\scr P_{\Zar,\ebu}$. To simplify the exposition, and because they do not seem to be of any practical relevance, we chose not to keep track of such refinements.}

We write $\Map(X,Y)$ for the mapping anima in an $\infty$-category, $\map(X,Y)$ for the mapping spectrum in a stable $\infty$-category, and $\Hom(X,Y)$ for the internal Hom object in a monoidal $\infty$-category.

Given motivic spectra $E,X\in \MS_S$, we write 
\begin{align*}
	E_{p,q}(X)&=\pi_0\Map(\Sigma^{p-2q}\SigmaP^q\1, X\otimes E),\\
	E^{p,q}(X)&=\pi_0\Map(X, \Sigma^{p-2q}\SigmaP^qE),
\end{align*}
and we let $E_n(X)=E_{2n,n}(X)$ and $E^n(X)=E^{2n,n}(X)$. We also write $E(X)$ for the mapping spectrum $\map(X,E)$.

A scheme is a derived scheme by default. Note that we often use hooked arrows $\into$ for immersions of derived schemes, even though these are not monomorphisms.
We write $\Sch_X$ for the $\infty$-category of $X$-schemes and $\Sm_X\subset \Sch_X$ for the full subcategory of smooth $X$-schemes. The superscript ``$\fp$'' means ``of finite presentation''.

We write $\Vect(X)$ for the anima of finite locally free sheaves over a scheme $X$, and $\Pic(X)=\Vect_1(X)$ for the subanima of invertible sheaves.
For a sheaf $\scr E\in\Vect(X)$, we denote by $\V(\scr E)=\Spec(\Sym\scr E)$ and $\P(\scr E)=\Proj(\Sym\scr E)$ the associated vector and projective bundles.
We denote by $\scr N_i=\scr L_i[-1]$ the conormal sheaf of a quasi-smooth immersion $i\colon Z\into X$, so that $\V(\scr N_i)\to Z$ is the normal bundle of $i$.

\subsection*{Acknowledgments}
We are especially grateful to Longke Tang, who shared with us his construction of Gysin maps, to Jacob Lurie for suggesting the link between $\A^1$-colocalization and logarithmic cohomology theories, and to Dustin Clausen and Akhil Mathew for communicating with us about the filtered cyclotomic trace.

\part{Atiyah duality}

\section{Background on Gysin maps}
\label{sec:gysin}

Let $S$ be a derived scheme.
In \cite[Section 7]{AHI}, we constructed a symmetric monoidal functor
\[
\K(S)\to \Pic(\MS_S),\quad \xi\mapsto \Sigma^{\xi}\1_S,
\]
called the \emph{J-homomorphism}, where $\K(S)$ is the K-theory anima of $S$, which is moreover natural in $S$. It is induced by a symmetric monoidal structure on the functor
\[
\Vect^\epi(S)\to \scr P_\ebu(\Sm_S)_*,\quad \scr E\mapsto \P(\scr E\oplus\scr O)/\P(\scr E),
\]
where $\Vect^\epi(S)$ is the $\infty$-category of finite locally free sheaves on $S$ and epimorphisms.
For a smooth morphism $f\colon X\to S$ and an element $\xi\in \K(X)$, we define the \emph{Thom spectrum} $\Th_X(\xi)\in\MS_S$ to be the motivic spectrum $f_\sharp\Sigma^\xi\1_X$. 
Note that any fiber sequence $\scr F\to\scr E\to\scr G$ in $\Perf(X)$ induces an isomorphism $\scr E\simeq \scr F\oplus\scr G$ in $\K(X)$, hence an isomorphism of Thom spectra
\[
\Th_X(\scr E)\simeq \Th_X(\scr F\oplus\scr G).
\]

Let $\Pair_S$ be the subcategory of the arrow category $\Sm_S^{\Delta^1}$ whose objects are closed embeddings and whose morphisms are excess intersection squares
\[
\begin{tikzcd}
Z' \arrow[hook]{r} \arrow[d] & X' \arrow[d]\\
Z \arrow[hook]{r} & X
\end{tikzcd}
\]
in the sense of \cite[Introduction]{KhanExcess}, i.e., topologically cartesian squares such that the induced map $\scr N_{Z/X} \vert_{Z'} \to \scr N_{Z'/X'}$ is surjective (these are exactly the conditions under which there is an induced map $\Bl_{Z'}(X')\to\Bl_Z(X)$). 
If $\int_{\Sm_S}\Vect^\epi\to\Sm_S$ denotes the cartesian fibration classified by $\Vect^\epi$, there is an obvious symmetric monoidal functor
\[
\Pair_S \to \int_{\Sm_S}\Vect^{\mathrm{epi}},\quad (Z\into X)\mapsto (Z,\scr N_{Z/X}).
\]
Applying the J-homomomorphism yields a symmetric monoidal functor
\begin{equation}\label{eqn:Th-Pair}
\Pair_S\to \MS_S,\quad (Z\into X)\mapsto \Th_Z(\scr N_{Z/X}).
\end{equation}
Furthermore, let $\mathrm{Triple}_S$ be the subcategory of $\Sm_S^{\Delta^2}$ whose objects are composable pairs of closed embeddings and whose morphisms are pairs of excess intersection squares. 
If $S_2\Vect^\epi$ denotes the $\infty$-category of short exact sequences of finite locally free sheaves (with epimorphisms as morphisms), there is a symmetric monoidal functor
\[
\mathrm{Triple}_S \to \int_{\Sm_S}S_2\Vect^\epi,\quad (W\into Z\into X) \mapsto (W,\scr N_{Z/X}|_W\into \scr N_{W/X}\onto\scr N_{W/Z}).
\]
Applying the J-homomorphism yields a symmetric monoidal functor
\begin{equation}\label{eqn:Th-Triple}
\mathrm{Triple}_S \to \MS_S^{\Delta^1},\quad (W\into Z\into X) \mapsto (\Th_W(\scr N_{W/X})\simeq \Th_W(\scr N_{Z/X}|_W\oplus \scr N_{W/Z})).
\end{equation}

Let $Z \into X$ be a closed embedding between smooth $S$-schemes. In \cite{Tang}, Longke Tang constructs a \emph{Gysin map}
\begin{equation*}
\gys\colon X_+ \to \Th_{Z}(\scr N_{Z/X})
\end{equation*}
in $\MS_S$, where $\scr N_{Z/X}$ is the conormal sheaf of the embedding. The following theorem records the minimal properties of Tang's construction that we will need.

\begin{theorem}[Tang]
\label{thm:tang}
\leavevmode
\begin{enumerate}
\item \textnormal{(Functoriality in pairs and linearity)} The Gysin maps assemble into an $\Sm_S$-linear functor
\[
\Pair_S \to \MS_S^{\Delta^1},\quad (Z\into X)\mapsto (\gys\colon X_+\to\Th_Z(\scr N_{Z/X})),
\]
natural in $S$, whose boundary $\Pair_S\to \MS_S^{\partial\Delta^1}$ is given by $(Z\into X)\mapsto \Sigma^\infty_{\P^1}X_+$ and~\eqref{eqn:Th-Pair}.
In particular, we have for any $\scr E\in\Vect(X)$ an $\scr E$-twisted Gysin map
\[
\Th_X(\scr E) =\frac{\P(\scr E\oplus\scr O_X)}{\P(\scr E)}\xrightarrow{\gys}\frac{\Th_{\P(\scr E_Z\oplus\scr O_Z)}(\scr N_{Z/X})}{\Th_{\P(\scr E_Z)}(\scr N_{Z/X})}\simeq \Th_Z(\scr E_Z\oplus\scr N_{Z/X}).
\]

\item \textnormal{(Composition of closed immersions)} 
The Gysin map $X_+\to\Th_X(\scr N_{\id})=X_+$ is the identity, naturally in $X$ and in $S$. If $Z \into X$ and $W \into Z$ are closed embeddings in $\Sm_S$, then there is a commutative square
\[
\begin{tikzcd}
	X_+ \ar{r}{\mathrm{gys}} \ar{d}[swap]{\mathrm{gys}} & \Th_{Z}(\scr N_{Z/X}) \ar{d}{\mathrm{gys}} \\
	 \Th_{W}(\scr N_{W/X}) \ar{r}{\sim} & \Th_{W}(\scr N_{Z/X}\vert_W \oplus \scr N_{W/Z} )
\end{tikzcd}
\]
in $\MS_S$, where the lower isomorphism comes from the fiber sequence $\scr N_{Z/X}\vert_W\to \scr N_{W/X}\to\scr N_{W/Z}$ in $\Perf(W)$.
Moreover, these squares assemble into an $\Sm_S$-linear functor
\[
\mathrm{Triple}_S\to \MS_S^{\Delta^1\times\Delta^1},
\]
natural in $S$,
whose boundary $\mathrm{Triple}_S\to \MS_S^{\partial(\Delta^1\times\Delta^1)}$ is given by three instances of \textnormal{(i)} and one instance of~\eqref{eqn:Th-Triple}.

\item \textnormal{(Normalization)} For any $\scr E\in\Vect(S)$, the null-sequence
\[
\P(\scr E)_+\to \P(\scr E \oplus \scr O)_+ \xrightarrow{\gys} \Th_S(\scr E)
\]
in $\MS_S$ obtained by applying \textnormal{(i)} to the excess intersection square
\[
\begin{tikzcd}
	\emptyset \ar[hook]{r}\ar{d} & \P(\scr E) \ar{d} \\
	S\ar[hook]{r} & \P(\scr E \oplus\scr O)
\end{tikzcd}
\]
is a cofiber sequence.
\end{enumerate}
\end{theorem}

\begin{proof}
	The construction and functoriality of the Gysin maps in (i) is \cite[Construction~3.1 and Remark~3.6]{Tang}.
	It is in fact a symmetric monoidal functor \cite[Theorem 3.9]{Tang} and in particular $\Sm_S$-linear.
	The commutative square in (ii) with a different lower horizontal isomorphism is \cite[Theorem~3.17]{Tang}, and the identification with the isomorphism from \eqref{eqn:Th-Triple} is \cite[Proposition~3.14]{Tang}. Finally, (iii) is \cite[Theorem~3.10]{Tang}.
\end{proof}

\begin{remark}
Theorem~\ref{thm:tang}(i) is a coherent version of the following statement: if
\[
\begin{tikzcd}
Z' \arrow[r,hookrightarrow]{} \arrow[d] & X' \arrow[d]\\
Z \arrow[r,hookrightarrow]{} & X
\end{tikzcd}
\]
is a topologically cartesian square in $\Sm_S$ such that the induced map $\scr N_{Z/X} \vert_{Z'} \to \scr N_{Z'/X'}$ is surjective, then there is an induced commutative square
\[
\begin{tikzcd}
X'_+ \arrow[r]{}{\gys} \arrow[d] & \Th_{Z'}(\scr N_{Z'/X'}) \arrow[d] \\
X_+ \arrow[r]{}{\gys} & \Th_Z(\scr N_{Z/X})
\end{tikzcd}
\]
in $\MS_S$. Moreover, the $\Sm_S$-linearity implies that we have for every $Y\in\Sm_S$ a commutative square
\[
\begin{tikzcd}
	X_+\otimes Y_+ \ar{r}{\gys\otimes \id} \ar{d}[swap,sloped]{\sim} & \Th_Z(\scr N_{Z/X}) \otimes Y_+ \ar{d}[sloped]{\sim} \\
	(X\times Y)_+ \ar{r}{\gys} & \Th_{Z\times Y}(\scr N_{Z\times Y/X\times Y})\rlap.
\end{tikzcd}
\]
\end{remark}

\begin{remark}\label{rmk:purity}
	Let $Z\into X$ be a closed embedding in $\Sm_S$ and let $U=X-Z$ be the open complement.
	Applying the functoriality of Gysin maps to the cartesian square
	\[
	\begin{tikzcd}
		\emptyset \ar[hook]{r} \ar{d} & U \ar{d} \\
		Z \ar[hook]{r} & X\rlap,
	\end{tikzcd}
	\]
	we obtain a null-homotopy of the composition
	\[
	U_+\to X_+\xrightarrow{\gys} \Th_Z(\scr N_{Z/X})
	\]
	in $\MS_S$. Unlike in $\A^1$-homotopy theory, this is almost never a cofiber sequence in $\MS_S$.
	For example, by Theorem~\ref{thm:tang}(iii), the fiber of the Gysin map $\P^1_+\to\Th_\infty(\scr N_\infty)$ is $\{0\}_+$ and not $\A^1_+$.
\end{remark}

\begin{construction}[Gysin transformation]
	\label{ctr:gysin}
	Let $f\colon X\to S$ be a smooth morphism and $i\colon Z\into X$ a closed embedding such that $fi$ is smooth.
	Then there is an $\Sm_S$-linear functor $\Sm_X \to \Pair_S$ that sends $Y \in \Sm_X$ to the pair $Y_Z \into Y$. Composing it with the functor of Theorem~\ref{thm:tang}(i), we get an $\Sm_S$-linear natural transformation
\begin{equation}\label{eq:funcyc}
f_\sharp\Sigma^\infty_{\P^1}(\ph)_+ \to (fi)_\sharp \Sigma^{\scr N_i} i^*\Sigma^\infty_{\P^1}(\ph)_+\colon \Sm_X \to \MS_S.
\end{equation}
By the universal property of $\MS_X$ as a $\scr P_{\Nis, \ebu}(\Sm_S,\Sp)$-linear $\infty$-category \cite[Proposition 1.2.2]{AnnalaIwasa2}, composition with $\Sigma^\infty_{\P^1}(\ph)_+\colon \Sm_X\to \MS_X$ induces an isomorphism of $\infty$-categories
\begin{equation}\label{eqn:univ-prop-MS}
\Fun_{\MS_S}^\mathrm{L}(\MS_X,\MS_S) \simto \Fun_{\Sm_S}^{\Nis,\ebu}(\Sm_X,\MS_S),
\end{equation}
where $\Fun_{\scr C}$ means $\scr C$-linear functors, $\Fun^\mathrm{L}$ means colimit-preserving functors, and $\Fun^{\Nis,\ebu}$ means functors that send Nisnevich sieves and elementary blowup squares to colimit diagrams.
We then define the \emph{Gysin transformation} 
\begin{equation*}
\mathrm{gys}(f,i)\colon f_\sharp \to g_\sharp \Sigma^{\scr N_i} i^*\colon \MS_X\to\MS_S
\end{equation*}
to be the $\MS_S$-linear natural transformation obtained from~\eqref{eq:funcyc} using the isomorphism~\eqref{eqn:univ-prop-MS}.
\end{construction}

\begin{proposition}[Properties of Gysin transformations]
	\label{prop:gysin-trans}
	\leavevmode
	\begin{enumerate}[beginpenalty=10000]
		\item \textnormal{(Base change)} Given a cartesian diagram
		\[
		\begin{tikzcd}
			W \ar[hook]{r}{k} \ar{d}{c} & Y \ar{r}{g} \ar{d}{b} & T \ar{d}{a} \\
			Z \ar[hook]{r}{i} & X \ar{r}{f} & S\rlap,
		\end{tikzcd}
		\]
		where $i$ is a closed immersion and $f$ and $fi$ are smooth, the following square commutes:
		\[
		\begin{tikzcd}[column sep=4em]
			g_\sharp b^* \ar{r}{\mathrm{gys}(g,k)b^*} \ar{d}[sloped]{\sim}[swap]{\mathrm{BC}} & (gk)_\sharp\Sigma^{\scr N_k}c^*i^* \ar{d}{\mathrm{BC}}[swap,sloped]{\sim} \\
			a^*f_\sharp \ar{r}{a^*\mathrm{gys}(f,i)} & a^*(fi)_\sharp\Sigma^{\scr N_i}i^*\rlap.
		\end{tikzcd}
		\]
		\item \textnormal{(Base independence)} Given morphisms
		\[
		\begin{tikzcd}
			Z \ar[hook]{r}{i} & X \ar{r}{f} & S \ar{r}{a} & T\rlap,
		\end{tikzcd}
		\]
		where $i$ is a closed immersion and $f$, $fi$, and $a$ are smooth, the following square commutes:
		\[
		\begin{tikzcd}[column sep=4em]
			a_\sharp f_\sharp \ar{r}{a_\sharp\mathrm{gys}(f,i)} \ar{d}[swap,sloped]{\sim} & a_\sharp(fi)_\sharp\Sigma^{\scr N_i}i^* \ar{d}[sloped]{\sim} \\
			(af)_\sharp \ar{r}{\mathrm{gys}(af,i)} & (afi)_\sharp\Sigma^{\scr N_i}i^*\rlap.
		\end{tikzcd}
		\]
		\item \textnormal{(Linearity)} Given morphisms
		\[
		\begin{tikzcd}
			Z \ar[hook]{r}{i} & X \ar{r}{f} & S\rlap,
		\end{tikzcd}
		\]
		where $i$ is a closed immersion and $f$ and $fi$ are smooth, and given $A\in\MS_X$ and $B\in \MS_S$, the following square commutes:
		\[
		\begin{tikzcd}[column sep=4em]
			f_\sharp(A\otimes f^*(B)) \ar{d}[sloped]{\sim}[swap]{\mathrm{PF}} \ar{r}{\mathrm{gys}(f,i)} & (fi)_\sharp\Sigma^{\scr N_i}(i^*(A)\otimes (fi)^*(B)) \ar{d}{\mathrm{PF}}[swap,sloped]{\sim} \\
			f_\sharp(A) \otimes B \ar{r}{\mathrm{gys}(f,i)\otimes\id} & (fi)_\sharp \Sigma^{\scr N_i}i^*(A)\otimes B\rlap.
		\end{tikzcd}
		\]
		\item \textnormal{(Composition of closed immersions)}
		The map $\mathrm{gys}(f,\id)$ is the identity, and 
		given morphisms
		\[
		\begin{tikzcd}
			W \ar[hook]{r}{k} & Z \ar[hook]{r}{i} & X \ar{r}{f} & S\rlap,
		\end{tikzcd}
		\]
		where $i$ and $k$ are closed immersions and $f$, $fi$, and $fik$ are smooth, the following square commutes:
		\[
		\begin{tikzcd}[column sep=4em]
			f_\sharp \ar{r}{\mathrm{gys}(f,i)} \ar{d}[swap]{\mathrm{gys}(f,ik)} & (fi)_\sharp\Sigma^{\scr N_i}i^* \ar{d}{\mathrm{gys}(fi,k)} \\
			 (fik)_\sharp\Sigma^{\scr N_{ik}}k^*i^* \ar{r}{\sim} & (fik)_\sharp\Sigma^{\scr N_k+k^*\scr N_i}k^*i^* \rlap.
		\end{tikzcd}
		\]
		Here, the isomorphism is induced by the fiber sequence $k^*(\scr N_i)\to \scr N_{ik}\to\scr N_k$ in $\Perf(W)$.
	\end{enumerate}
\end{proposition}

\begin{proof}
	(i) It suffices to build a square in $\Fun_{\Sm_S}(\Sm_X,\MS_T)$. By Theorem~\ref{thm:tang}(i), there is a commutative diagram of $\Sm_S$-modules
	\[
	\begin{tikzcd}
		\Sm_X \ar{r} \ar{d} & \Pair_S \ar{d} \ar{r} & \MS_S^{\Delta^1} \ar{d} \\
		\Sm_Y \ar{r} & \Pair_T \ar{r} & \MS_T^{\Delta^1}\rlap.
	\end{tikzcd}
	\]
	The boundary is an $\Sm_S$-linear transformation between two functors $\Sm_X\to\MS_T^{\Delta^1}$, i.e., a functor $\Delta^1\times\Delta^1\to\Fun_{\Sm_S}(\Sm_X,\MS_T)$. This is the desired square.
	
	(ii) We consider the cartesian diagram
	\[
	\begin{tikzcd}
		Z\times_TS \ar[hook]{r}{\bar\imath} \ar{d}{\pi_Z} & X\times_TS \ar{r}{\bar f} \ar{d}{\pi_X} & S\times_TS \ar{r}{\bar a} \ar{d}{\pi_S} & S \ar{d}{a} \\
		Z \ar[hook]{r}{i} & X \ar{r}{f} & S \ar{r}{a} & T\rlap.
	\end{tikzcd}
	\]
	There is an $\Sm_S$-linear functor $\Sm_X\to\Pair_S^{\Delta^1}$ sending $U$ to the excess intersection square
	\[
	\begin{tikzcd}
		U_Z \ar[hook]{r} \ar{d} & U \ar{d} \\
		U_Z\times_TS \ar[hook]{r} & U\times_TS\rlap.
	\end{tikzcd}
	\]
	Composing with the Gysin map functor $\Pair_S\to \MS_S^{\Delta^1}$ and applying the isomorphism~\eqref{eqn:univ-prop-MS}, we obtain a commutative square
	\[
	\begin{tikzcd}[column sep=4em]
		f_\sharp \ar{r}{\gys(f,i)} \ar{d} & (fi)_\sharp\Sigma^{\scr N_i} i^* \ar{d} \\
		(\bar a\bar f)_\sharp \pi_X^* \ar{r}{\gys(\bar a\bar f,\bar\imath)} & (\bar a\bar f\bar\imath)_\sharp\Sigma^{\scr N_{\bar\imath}}\pi_Z^*i^*\rlap.
	\end{tikzcd}
	\]
	On the other hand, (i) gives a commutative square
	\[
	\begin{tikzcd}[column sep=4em]
		(\bar a\bar f)_\sharp \pi_X^* \ar{r}{\mathrm{gys}(\bar a\bar f,\bar\imath)} \ar{d}[sloped]{\sim}[swap]{\mathrm{BC}} & (\bar a\bar f\bar\imath)_\sharp\Sigma^{\scr N_{\bar\imath}}\pi_Z^*i^* \ar{d}{\mathrm{BC}}[swap,sloped]{\sim} \\
		a^*(af)_\sharp \ar{r}{a^*\mathrm{gys}(af,i)} & a^*(afi)_\sharp\Sigma^{\scr N_i}i^*\rlap.
	\end{tikzcd}
	\]
	Combining these two squares and using the adjunction $a_\sharp\dashv a^*$ gives the desired commutative square.
	
	(iii) The natural transformation $\gys(f,i)$ is $\MS_S$-linear by construction.
	
	(iv) This square is obtained from Theorem~\ref{thm:tang}(ii) in the same way that the Gysin transformation was obtained from Theorem~\ref{thm:tang}(i), using the functor $\Sm_X\to\mathrm{Triple}_S$ sending $Y$ to $Y_W\into Y_Z\into Y$.
\end{proof}

\begin{remark}[Functoriality of Gysin transformations]
	\label{rmk:func-gysin}
	Let $\int_{\Pair_S}\MS\to \Pair_S$ be the cartesian fibration classified by 
	\[
	\Pair_S^\op\to \Cat_\infty,\quad (Z\into X)\mapsto \MS_X.
	\]
\begin{enumerate}
	\item There is an $\MS_S$-linear functor
	\begin{equation*}\label{eqn:Gys-natural}
	\int_{\Pair_S}\MS \to \MS_S^{\Delta^1},
	\end{equation*}
	whose restriction to the fiber over $Z\into X$ is the associated Gysin transformation $\MS_X\to \MS_S^{\Delta^1}$.
	To see this, consider the functor
	\[
	\int_{\Pair_S}\Sm \to \Pair_S,\quad (Z\into X,Y\in \Sm_X)\mapsto (Y_Z\into Y),
	\]
	which is fiberwise $\Sm_S$-linear. Composing it with the $\Sm_S$-linear functor from Theorem~\ref{thm:tang}(i) and applying the isomorphism~\eqref{eqn:univ-prop-MS} fiberwise yields the claimed $\MS_S$-linear functor.
	
	In the same way, there is an $\MS_S$-linear functor
	\[
	\int_{\mathrm{Triple}_S}\MS \to \MS_S^{\Delta^1\times\Delta^1}
	\]
	encoding the commutative squares of Proposition~\ref{prop:gysin-trans}(iv).
	\item The $\Sm_S$-linear functor from Theorem~\ref{thm:tang}(i) can also be extended to an $\Sm_S$-linear functor
	\[
	\int_{\Pair_S}\MS \to \MS_S^{\Delta^1},
	\]
	where the action of $\Sm_S$ on the source is given by
	\[
	Y \otimes (Z\into X, E) = (Y\times Z\into Y\times X, \pi_X^*E).
	\]
	Note that this differs from the $\Sm_S$-linear functor of (i), which used the fiberwise $\Sm_S$-action on the source.
	However, the obvious maps
	\[
	(Y\times Z\into Y\times X, \pi_X^*E) \to (Z\into X, \pi_{X\sharp}\pi_X^*E)
	\]
	are part of an oplax $\Sm_S$-linear structure on the identity of $\int_{\Pair_S}\MS$, viewed as a functor from the above $\Sm_S$-module structure to the fiberwise one, which by smooth base change becomes strict after composing with the functor to $\MS_S^{\Delta^1}$.
\end{enumerate}
\end{remark}

\begin{remark}[Gysin map for Thom spectra]
	\label{rmk:gysin-thom}
	Let $i\colon Z\into X$ be a closed immersion in $\Sm_S$ and let $\xi\in\K(X)$. Applying the Gysin transformation to $\Sigma^{\xi}\1_X$ yields a $\xi$-twisted version of the Gysin map
	\[
	\Th_X(\xi) \to \Th_Z(\scr N_i+i^*\xi).
	\]
	When $\xi$ comes from a finite locally free sheaf $\scr E\in\Vect(X)$, this is by construction the $\scr E$-twisted Gysin map described in Theorem~\ref{thm:tang}(i).
	
	Composing the functor of Remark~\ref{rmk:func-gysin}(ii) with the J-homomorphism, we obtain an extension of the functor of Theorem~\ref{thm:tang}(i) to an $\Sm_S$-linear functor
	\[
	\int_{\Pair_S}\K\to \MS_S^{\Delta^1},
	\]
	sending $(Z\into X,\xi)$ to the $\xi$-twisted Gysin map above.
	Similarly, we can extend the functor of Theorem~\ref{thm:tang}(ii) to an $\Sm_S$-linear functor
	\[
	\int_{\mathrm{Triple}_S}\K\to \MS_S^{\Delta^1\times\Delta^1}.
	\]
\end{remark}

\section{The geometric evaluation map}
\label{sec:eval}

Here, we construct the geometric dual and the geometric evaluation map of Thom spectra over smooth separated $S$-schemes, which play an instrumental role in our proofs of Atiyah duality and ambidexterity. The geometric evaluation map induces a comparison map between the geometric dual and the categorical dual, which will ultimately be shown to be an isomorphism for Thom spectra over smooth projective $S$-schemes (see Corollary~\ref{cor:atiyah}). The naturality properties of this comparison map are also investigated.

\begin{construction}[Geometric evaluation map and comparison map]
\label{ctr:ev}
Let $X$ be a smooth separated $S$-scheme and let $\xi\in\K(X)$.
The \emph{geometric evaluation map}
\[
\ev_{X,\xi}\colon \Th_X(\xi)\otimes \Th_X(-\xi-\Omega_X) \to\1_S
\] 
is defined as the composition
\[
\Th_X(\xi)\otimes \Th_X(-\xi-\Omega_X) \simeq \Th_{X\times X}(\xi\boxplus (-\xi-\Omega_X)) \xrightarrow{\gys} \Th_X(\scr N_{\delta}-\Omega_X) \simeq X_+ \to\1_S,
\]
where $\delta\colon X\into X\times X$ is the diagonal embedding, whose conormal sheaf $\scr N_\delta$ is identified with $\Omega_X$ via the canonical isomorphisms \[\scr N_\delta\simeq \delta^*\Omega_{\pi_2}\simeq \delta^*\pi_1^*(\Omega_X)\simeq\Omega_X.\]
The Thom spectrum $\Th_X(-\xi - \Omega_X)$ is called the \emph{geometric dual} of $\Th_X(\xi)$ in $\MS_S$.

By adjunction, the geometric evaluation map induces the \emph{comparison map}
\begin{equation*}
\rm{comp}_{X,\xi} \colon \Th_X(- \xi - \Omega_X) \to \Th_X(\xi)^\vee
\end{equation*}
in $\MS_S$, where $E^\vee=\Hom(E,\1_S)$ denotes the categorical dual of a motivic spectrum $E$.
\end{construction}

\begin{remark}
	The choice of isomorphism $\scr N_\delta\simeq\Omega_X$ used in the definition of $\ev_{X,\xi}$ is arbitrary, but it will be important to remember this choice. The other choice yields a pairing that differs from $\ev_{X,\xi}$ by the automorphism $\langle -1\rangle^{\rk\Omega_X}$ of $\1_X$, which can be nontrivial when $X$ is odd-dimensional (e.g., if it admits a formally real residue field).
\end{remark}

\begin{construction}[Gysin null-sequences]
	\label{ctr:null-sequences}
	Let $X$ be a smooth $S$-scheme and let $Y$ and $Z$ be smooth closed subschemes in $X$ with $Y\cap Z=\emptyset$.
	The cartesian square
	\[
	\begin{tikzcd}
		\emptyset \ar[hook]{r} \ar{d} & Y \ar{d} \\
		Z \ar[hook]{r} & X
	\end{tikzcd}
	\]
	is then a morphism from $\emptyset\into Y$ to $Z\into X$ in the category $\mathrm{Pair}_S$, inducing a null-sequence
	\[
	\Th_Y(\xi|_Y) \xrightarrow{\mathrm{inc}} \Th_X(\xi) \xrightarrow{\gys} \Th_Z(\scr N_Z+\xi|_Z)
	\]
	in $\MS_S$ for every $\xi\in\K(X)$. 
\end{construction}

\begin{proposition}[Bivariant naturality of the comparison map]
	\label{prop:null-sequences}
	Let $X$ be a smooth separated $S$-scheme and let $Y$ and $Z$ be smooth closed subschemes in $X$ with $Y\cap Z=\emptyset$.
	For any $\xi\in \K(X)$, the comparison maps $\mathrm{comp}_{Z,\xi+\scr N_Z}$, $\mathrm{comp}_{X,\xi}$, and $\mathrm{comp}_{Y,\xi}$ participate in a morphism of null-sequences
	\begin{equation}\label{eqn:null-sequences}
	\begin{tikzcd}
		\Th_Z(-\xi-\Omega_X) \ar{r}{\mathrm{inc}} \ar{d} & \Th_X(-\xi-\Omega_X) \ar{r}{\mathrm{gys}} \ar{d} & \Th_Y(-\xi-\Omega_Y) \ar{d} \\
		\Th_Z(\xi+\scr N_Z)^\vee \ar{r}{\mathrm{gys}^\vee} & \Th_X(\xi)^\vee \ar{r}{\mathrm{inc}^\vee} & \Th_Y(\xi)^\vee\rlap,
	\end{tikzcd}
	\end{equation}
	where both rows are instances of Construction~\ref{ctr:null-sequences}.
\end{proposition}

\begin{proof}
	Unpacking the duals, we are led to consider the following diagram in $\MS_S$ (where we omit $\xi$ for simplicity):
	\begin{equation}\label{eqn:big-3d-diagram}
	\begin{tikzcd}[row sep={25,between origins}, column sep={42,between origins}]
		& & 0 \ar[hook]{rrrr} && && Y_+\otimes \Th_Y(-\Omega_Y) \ar{dddr} \\
	 & Y_+\otimes \Th_Z(-\Omega_X) \ar[hook]{rrrr} \ar[hook]{dddd} \ar{ur} \ar{dl} && && Y_+\otimes \Th_X(-\Omega_X) \ar[hook]{dddd} \ar{ur} && \\
		0 \ar[hook]{dddd} \ar[crossing over]{drrr} \\
		& && 0 \ar[from=uuul,crossing over] \ar[crossing over,hook]{rrrr} \ar[from=uull,shorten >=2pt] && && Y_+ \ar[hook]{dddd} \ar[from=uull,shorten >=2pt] \\
		\\
		& X_+\otimes \Th_Z(-\Omega_X) \ar[hook]{rrrr} \ar{ddrr} \ar{dl} && &&  X_+\otimes \Th_X(-\Omega_X) \ar{ddrr} && \\
		\Th_Z(\scr N_Z)\otimes \Th_Z(-\Omega_X) \ar{drrr} \\
		& && Z_+ \ar[hook]{rrrr} \ar[from=uuuu,crossing over,hook] && && X_+\rlap{$\to \1_S$.}
	\end{tikzcd}
	\end{equation}
	Here, the horizontal and vertical arrows are induced by closed immersions in $\Sm_S$, while the other arrows are the obvious Gysin maps. 
	The left prism is induced by the excess intersection squares
	\[
	\begin{tikzcd}
		\emptyset \ar[hook]{r} \ar{d} & \emptyset \ar[hook]{r} \ar{d} & Y\times Z \ar{d} \\
		Z \ar[hook]{r} & Z\times Z \ar[hook]{r} & X\times Z\rlap.
	\end{tikzcd}
	\]
	The prism at the top is similarly induced by the excess intersection squares
	\[
	\begin{tikzcd}
		\emptyset \ar[hook]{r} \ar{d} & \emptyset \ar[hook]{r} \ar{d} & Y\times Z \ar{d} \\
		Y \ar[hook]{r} & Y\times Y \ar[hook]{r} & Y\times X\rlap.
	\end{tikzcd}
	\]
	Finally, the cube is induced by the following commutative square in $\mathrm{Pair}_S$:
	\[
	\begin{tikzcd}
		(\emptyset\into Y\times Z) \ar{r} \ar{d} & (Y\into Y\times X) \ar{d} \\
		(Z\into X\times Z) \ar{r} & (X\into X\times X)\rlap.
	\end{tikzcd}
	\]
	The bottom (resp.\ right) face of~\eqref{eqn:big-3d-diagram} corresponds to the left (resp.\ right) square in~\eqref{eqn:null-sequences}.
	The uppermost (resp.\ leftmost) $0$ in~\eqref{eqn:big-3d-diagram} corresponds to the null-homotopy of the first (resp.\ second) row in~\eqref{eqn:null-sequences} (here, we use the $\Sm_S$-linearity of Gysin maps from Remark~\ref{rmk:gysin-thom}).
	Finally, the factorization through the middle $0$ in~\eqref{eqn:big-3d-diagram} identifies the two resulting null-homotopies of the composite map $\Th_Z(-\xi-\Omega_X)\to \Th_Y(\xi)^\vee$ in~\eqref{eqn:null-sequences}.
\end{proof}

\begin{lemma}\label{lem:PF-BC}
	Consider a cartesian square
	\[
	\begin{tikzcd}
	Z \ar{d}[swap]{\bar g} \ar{r}{\bar f} \ar{dr}[description]{h} & Y \ar{d}{g} \\
	X \ar{r}[swap]{f} & S\rlap,
	\end{tikzcd}
	\]
	where $f$ and $g$ are smooth.
	Then the following diagram commutes for all $A\in \MS_X$ and $B\in\MS_Y$:
	\[
	\begin{tikzcd}[column sep=2.5em]
	f_\sharp (A\otimes f^*g_\sharp (B)) \ar{d}[sloped]{\sim}[swap]{\mathrm{PF}} \ar[<-]{r}{\mathrm{BC}}[swap]{\sim} & f_\sharp (A\otimes \bar g_\sharp \bar f^*(B))  \\
	f_\sharp (A)\otimes g_\sharp (B) & h_\sharp (\bar g^*(A)\otimes \bar f^*(B)) \ar{d}{\mathrm{PF}}[swap,sloped]{\sim} \ar{u}[sloped]{\sim}[swap]{\mathrm{PF}}  \\
	g_\sharp (g^*f_\sharp (A)\otimes B) \ar{u}{\mathrm{PF}}[swap,sloped]{\sim} \ar[<-]{r}{\mathrm{BC}}[swap]{\sim} & g_\sharp (\bar f_\sharp \bar g^*(A)\otimes B)\rlap.
	\end{tikzcd}
	\]
\end{lemma}

\begin{proof}
This follows from the lax monoidal structure of 
\[
\h\MS^*_{\sharp}\colon \Span(\Sch,\all,\mathrm{smooth})\to\Cat_1, \quad (X\xleftarrow{f} Z\xrightarrow{g} Y)\mapsto g_{\sharp} f^*.\qedhere
\]
\end{proof}

\begin{proposition}\label{prop:skew-symmetry}
	Let $X$ be a smooth separated $S$-scheme and let $\xi\in\K(X)$. Then the pairing
	\[
	\ev_{X,\xi}\colon \Th_X(\xi) \otimes \Th_X(-\xi-\Omega_X) \to \1_S
	\]
	of Construction~\ref{ctr:ev} is skew-symmetric in the following sense:
	\[
	\ev_{X,\xi}\circ\sigma = \ev_{X,-\xi-\Omega_X} \langle -1\rangle^{\rk\Omega_X},
	\]
	where $\sigma$ is the swap map and $\langle -1\rangle$ is the image of the nontrivial loop by the J-homomorphism $\Omega\K(\Z)\to\Aut(\1_\Z)$.
\end{proposition}

\begin{proof}
	Let $f\colon X\to S$ be the structure map.
	The pairing $\ev_{X,\xi}$ is by definition
	\begin{align*}
	f_\sharp\Sigma^\xi\1_X\otimes f_\sharp\Sigma^{-\xi-\Omega_f}\1_X \overset{\mathrm{PF}}{\simeq} & f_\sharp \Sigma^\xi f^*f_\sharp\Sigma^{-\xi-\Omega_f}\1_X
	\overset{\mathrm{BC}}{\simeq} f_\sharp\Sigma^\xi\pi_{2\sharp}\pi_1^*\Sigma^{-\xi-\Omega_f}\1_X \\
	 \xrightarrow{\mathrm{gys}(\pi_2,\delta)} & f_\sharp\Sigma^{\xi}\Sigma^{\scr N_\delta}\Sigma^{-\xi-\Omega_f}\1_X
	 \simeq f_\sharp\1_X \xrightarrow{\epsilon} \1_S.
	\end{align*}
	The following commutative diagram (with $A=\Sigma^\xi\1_X$ and $B=\Sigma^{-\xi-\Omega_f}\1_X$) shows that the pairing with values in $f_\sharp\1_X$ is already skew-symmetric in the desired sense:
	\[
	\begin{tikzcd}[column sep=3em]
	f_\sharp(A\otimes f^*f_\sharp(B)) \ar{d}[sloped]{\sim}[swap]{\mathrm{PF}} \ar[<-]{r}{\mathrm{BC}}[swap]{\sim} & f_\sharp(A\otimes \pi_{2\sharp}\pi_1^*(B)) \ar{r}{\mathrm{gys}(\pi_2,\delta)} & f_\sharp(A\otimes \Sigma^{\scr N_\delta}B) \ar{r}{\Sigma^{\alpha_2}}[swap]{\sim} & f_\sharp(A\otimes \Sigma^{\Omega_f}B) \ar{dd}{\Sigma^{-\id}}[sloped,swap]{\sim}  \\
	f_\sharp(A)\otimes f_\sharp(B) & h_\sharp(\pi_2^*(A)\otimes \pi_1^*(B)) \ar{r}{\mathrm{gys}(h,\delta)} \ar{d}{\mathrm{PF}}[swap,sloped]{\sim} \ar{u}[sloped]{\sim}[swap]{\mathrm{PF}} & f_\sharp\Sigma^{\scr N_\delta}(A\otimes B) \ar{d}[sloped,swap]{\sim} \ar{u}[sloped]{\sim} \\
	f_\sharp(f^*f_\sharp(A)\otimes B) \ar{u}{\mathrm{PF}}[swap,sloped]{\sim} \ar[<-]{r}{\mathrm{BC}}[swap]{\sim} \ar[equal]{d} & f_\sharp(\pi_{1\sharp}\pi_2^*(A)\otimes B) \ar{r}{\mathrm{gys}(\pi_1,\delta)} \ar{d}{\mathrm{BC}}[swap,sloped]{\sim} & f_\sharp(\Sigma^{\scr N_\delta}A\otimes B) \ar{d}{\Sigma^{-\id}}[sloped,swap]{\sim} \ar{r}{\Sigma^{\alpha_1}}[swap]{\sim} & f_\sharp(\Sigma^{\Omega_f}A\otimes B) \ar[equal]{d} \\
	f_\sharp(f^*f_\sharp(A)\otimes B) \ar[<-]{r}{\mathrm{BC}}[swap]{\sim} & f_\sharp(\pi_{2\sharp}\pi_1^*(A)\otimes B) \ar{r}{\mathrm{gys}(\pi_2,\delta)} & f_\sharp(\Sigma^{\scr N_\delta}A\otimes B) \ar{r}{\Sigma^{\alpha_2}}[swap]{\sim} & f_\sharp(\Sigma^{\Omega_f}A\otimes B)\rlap.
	\end{tikzcd}
	\]
	Here:
	\begin{itemize}
		\item the big left rectangle commutes by Lemma~\ref{lem:PF-BC};
		\item the two central upper rectangles commute by the $\MS$-linearity and base independence of Gysin transformations;
		\item the rectangle of base change isomorphisms is induced by the cartesian squares
		\[
		\begin{tikzcd}
			X\times_SX \ar{d}[swap]{\sigma} \ar{r}{\pi_1} \ar[bend right=50]{dd}[swap]{\pi_2} & X \ar{d}{\id} \ar[bend left=50]{dd}{f} \\
			X\times_SX \ar{d}[swap]{\pi_1} \ar{r}{\pi_2} & X \ar{d}{f} \\
			X \ar{r}[swap]{f} & S\rlap,
		\end{tikzcd}
		\]
		where $\sigma$ is the swap map;
		\item the central bottom rectangle commutes by the compatibility of Gysin transformations with base change, applied to the following cartesian diagram:
		\[
		\begin{tikzcd}
			X \ar[hook]{r}{\delta} \ar{d}[swap]{\id} & X\times_SX \ar{d}[swap]{\sigma} \ar{r}{\pi_1} & X \ar{d}{\id} \\
			X \ar[hook]{r}{\delta} & X\times_SX \ar{r}{\pi_2} & X
		\end{tikzcd}
		\]
		(note that the automorphism of $\scr N_\delta$ induced by $\sigma$ is $-\id$);
		\item the isomorphism $\alpha_2$ is $\scr N_\delta\simeq \delta^*\Omega_{\pi_2}\simeq \delta^*\pi_1^*\Omega_f\simeq\Omega_f$ and similarly for $\alpha_1$, so that $\alpha_1\alpha_2^{-1}=-\id$; therefore, the upper and lower right rectangles commute.\qedhere
	\end{itemize}
\end{proof}

\section{Duals of projective spaces}
\label{sec:Pdual}

Here, we identify the duals of projective spaces in $\MS_S$ with their geometric duals. More precisely, we show that the comparison map
\begin{equation*}
\rm{comp}_{\P^n,\scr O(-1)^m} \colon \Th_{\P^n} (- \scr O(-1)^m - \Omega_{\P^n})\to \Th_{\P^n}(\scr O(-1)^m)^\vee
\end{equation*}
is an isomorphism for all $n,m \geq 0$. We will argue by induction on $n$, using Proposition~\ref{prop:null-sequences} for the disjoint closed subschemes $\P^0,\P^{n-1}\subset \P^n$. The nontrivial part is to show that the null-sequences in \emph{loc.\ cit.} are in fact cofiber sequences.

\begin{lemma}\label{lem:projquot}
Let $\scr E$ and $\scr F$ be finite locally free sheaves on a derived scheme. The commutative square
\[
\begin{tikzcd}
	\P_{\P(\scr E)}(\scr F) \ar{d} \ar[hook]{r} & \P_{\P(\scr E)}(\scr O (1) \oplus \scr F) \ar{d} \\
	\P(\scr F) \ar[hook]{r} & \P(\scr E \oplus \scr F)
\end{tikzcd}
\]
is a universal strict virtual Cartier divisor over $\P(\scr F) \into \P(\scr E \oplus \scr F)$. In particular, there is an isomorphism
\begin{equation*}
\Bl_{\P(\scr F)}(\P(\scr E \oplus \scr F)) \simeq \P_{\P(\scr E)}(\scr O (1) \oplus \scr F)
\end{equation*}
identifying the exceptional divisor with $\P_{\P(\scr E)}(\scr F)$. 
\end{lemma}


\begin{proof}
The embedding $\P(\scr F) \into \P(\scr E \oplus \scr F)$ is the vanishing locus of the tautological map $\scr E \to \scr O(1)$ on $\P(\scr E \oplus \scr F)$. Thus, by the universal property of derived blowup of a vector bundle section \cite[Proposition 2.7]{annala-chern}, the points of the blowup are commutative squares
\begin{equation*}\label{eq:qcohsq}
\begin{tikzcd}
\scr E \arrow[r,twoheadrightarrow] \arrow[d] & \scr L \arrow[d] \\
\scr E \oplus \scr F \arrow[r,twoheadrightarrow] & \scr L'
\end{tikzcd}
\end{equation*}
of quasi-coherent sheaves on the blowup, where $\scr L$ and $\scr L'$ are invertible (see also \cite[Lemma~3.2]{AHI}).

On the other hand, the points of $\P_{\P(\scr E)}(\scr O(1) \oplus \scr F)$ are pairs of surjections $\scr E \onto \scr L$ and $\scr L \oplus \scr F \onto \scr L'$. This data is clearly equivalent to that of the commutative square above. The exceptional divisor is the vanishing locus of $\scr L \to \scr L'$, which coincides with $\P_{\P(\scr E)}(\scr F)$.
\end{proof}

\begin{construction}[Thom spaces of twists by divisors]
	\label{ctr:thomtwist}
Let $X \in \Sm_S$, let $D\subset X$ be a smooth divisor, and let $\scr E$ be a finite locally free sheaf on $X$.
We construct a zigzag in $\Sm_S^\sncd$ \cite[Definition 2.4]{AHI}
\begin{equation*}
\P_X(\scr E \oplus \scr O) \leftarrow B \rightarrow \P_X(\scr E(D) \oplus \scr O)
\end{equation*}
inducing isomorphisms
\begin{equation*}
\Th_X(\scr E) / \Th_D(\scr E) \xleftarrow{\sim} B/\partial B \xrightarrow{\sim} \Th_X(\scr E(D)) / \Th_D(\scr E(D))
\end{equation*}
in $\scr P_{\sbu}(\Sm_S)_*$ (or in $\scr P_{\ebu}(\Sm_S)_*$ if $D\into X$ is elementary).
Let 
\[
b_1\colon B_1 \to \P_X(\scr E \oplus \scr O) \quad\text{and} \quad  b_2\colon B_2 \to \P_X(\scr E(D) \oplus \scr O)
\]
be the blowups of $\P_D(\scr O) \into \P_X(\scr E \oplus \scr O)$ and $\P_D(\scr E(D)) \into \P_X(\scr E(D) \oplus \scr O)$, respectively. The functor of points of these blowups may be described using the universal property of blowups of vanishing loci (see \cite[Theorem 122]{AnnalaThesis} or \cite[Lemma 3.2]{AHI}):
\begin{itemize}
\item An $X$-morphism to $B_1$ corresponds to a surjection
\[
\scr E \oplus \scr O \onto \scr L_1
\]
of locally free sheaves, where $\scr L_1$ is invertible, and a factorization 
\[
\begin{tikzcd}
\scr E(D) \oplus \scr L_1 \arrow[r] \arrow[d,twoheadrightarrow,dashed] & \scr L_1(D) \\
\scr M_1 \ar[dashed]{ur}
\end{tikzcd}
\]
through a surjection to an invertible sheaf $\scr M_1$. The vanishing locus of $\scr M_1\to\scr L_1(D)$ is the exceptional divisor $E_1$ of the blowup.

\item An $X$-morphism to $B_2$ corresponds to a surjection
\[
\scr E(D) \oplus \scr O \onto \scr L_2
\]
of locally free sheaves, where $\scr L_2$ is invertible, and a factorization 
\[
\begin{tikzcd}
\scr O(D) \oplus \scr L_2 \arrow[r] \arrow[d,twoheadrightarrow,dashed] & \scr L_2(D) \\
\scr M_2 \ar[dashed]{ur}
\end{tikzcd}
\]
through a surjection to an invertible sheaf $\scr M_2$. The vanishing locus of $\scr M_2\to\scr L_2(D)$ is the exceptional divisor $E_2$ of the blowup.
\end{itemize}
We note that in both cases, since the composite $\scr L_i\to\scr M_i\to\scr L_i(D)$ is the identity tensored with the tautological map $\scr O\to\scr O(D)$, so is the reverse composite $\scr M_i\to\scr L_i(D)\to\scr M_i(D)$.

We claim that $B_1$ and $B_2$ can be identified in such a way that $\scr L_2=\scr M_1$ and $\scr M_2=\scr L_1(D)$:
\begin{itemize}
\item Over $B_1$, the composition
\begin{equation*}\label{eq:thomtwistsurj1}
\scr E (D) \oplus \scr O \to \scr E (D) \oplus \scr L_1 \onto \scr M_1
\end{equation*}
is a surjection, which defines a map $B_1\to\P_X(\scr E(D) \oplus \scr O)$. Indeed, the vanishing loci of $\scr E (D) \to \scr M_1$ and $\scr O \to \scr M_1$ are strict transforms of $\P_X(\scr O)$ and $\P_D(\scr E \oplus \scr O) \cup \P_X(\scr E )$, respectively, and these are disjoint.

Furthermore, we have a factorization
\[
\begin{tikzcd}
\scr O (D) \oplus \scr M_1 \arrow[r] \arrow[d]& \scr M_1(D)\rlap. \\
\scr L_1(D) \arrow[ur]&
\end{tikzcd}
\]
The vanishing loci of the maps $\scr O (D) \to \scr L_1(D)$ and $\scr M_1 \to \scr L_1(D)$ are $\P_X(\scr E)$ and the exceptional divisor, respectively. Since they are disjoint, the vertical map is surjective. This defines a map $B_1\to B_2$.

\item Over $B_2$, the composition
\begin{equation*}\label{eq:thomtwistsurj2}
\scr E(D) \oplus \scr O(D) \to \scr L_2 \oplus \scr O(D) \onto \scr M_2
\end{equation*}
is a surjection, which defines a map $B_2 \to \P_X(\scr E \oplus \scr O)$.
Indeed, the vanishing loci of $\scr E(D) \to \scr M_2$ and $\scr O(D) \to \scr M_2$ are strict transforms of $\P_X(\scr O)\cup\P_D(\scr E(D) \oplus \scr O)$ and $\P_X(\scr E(D))$, respectively, and these are disjoint.

Furthermore, we have a factorization
\[
\begin{tikzcd}
\scr E (D) \oplus \scr M_2(-D) \arrow[r] \arrow[d]& \scr M_2\rlap. \\
\scr L_2 \arrow[ur]&
\end{tikzcd}
\]
The vanishing loci of the maps $\scr E (D) \to \scr L_2$ and $\scr M_2(-D)\to\scr L_2$ are $\P_X(\scr O)$ and the exceptional divisor, respectively. Since they are disjoint, the vertical map is a surjection. This defines a map $B_2 \to B_1$. 
\end{itemize}
It is clear from the construction that the maps $B_1\to B_2$ and $B_2\to B_1$ are inverse to one another. 

Finally, we define relative strict normal crossings divisors $\partial B_i$ on $B_i$ as follows:
\begin{itemize}
\item The smooth components of $\partial B_1$ are the strict transforms $D_1$ and $P_1$ of $\P_D(\scr E \oplus \scr O)$ and $\P_X(\scr E)$, as well as the exceptional divisor $E_1$.
\item The smooth components of $\partial B_2$ are the strict transforms $D_2$ and $P_2$ of $\P_D(\scr E(D) \oplus \scr O)$ and $\P_X(\scr E(D))$, as well as the exceptional divisor $E_2$.
\end{itemize}
By smooth blowup excision \cite[Proposition 2.7]{AHI}, the maps $b_i$ induce isomorphisms
\begin{align*}
B_1/\partial B_1 &\simto \P_X(\scr E \oplus \scr O) / (\P_D(\scr E \oplus \scr O) \cup \P_X(\scr E)) \simeq \Th_X(\scr E) / \Th_D(\scr E), \\
B_2 / \partial B_2 & \simto \P_X(\scr E(D) \oplus \scr O) / (\P_D(\scr E(D) \oplus \scr O) \cup \P_X(\scr E(D))) \simeq \Th_X(\scr E(D)) / \Th_D(\scr E(D))
\end{align*}
in $\scr P_{\sbu}(\Sm_S)_*$.
As the isomorphism $B_1 \simeq B_2$ identifies $D_1$ with $E_2$, $P_1$ with $P_2$, and $E_1$ with $D_2$, we obtain the desired zigzag in $\Sm_S^\sncd$.
\end{construction}

\begin{definition}
	Let $X$ be a smooth $S$-scheme, let $\xi\in\K(X)$, and let $Y,Z\subset X$ be smooth closed subschemes.
	We say that $(X,Y,Z,\xi)$ is a \emph{Gysin quadruple} if $Y\cap Z =\emptyset$ and if the null-sequence
	\[
	\Th_{Y}(\xi) \xrightarrow{\mathrm{inc}} \Th_{X}(\xi) \xrightarrow{\gys}\Th_{Z}(\scr N_Z+\xi)
	\]
	of Construction~\ref{ctr:null-sequences} is a cofiber sequence in $\MS_S$.
\end{definition}

\begin{proposition}\label{prop:gys-cofib}
Let $X$ be a smooth $S$-scheme, let $\xi\in \K(X)$, and let $Y,Z\subset X$ be smooth closed subschemes with $Y\cap Z=\emptyset$.
\begin{enumerate}
	\item For any $\scr E\in\Vect(S)$, $(\P(\scr E\oplus\scr O),\P(\scr E),S,0)$ is a Gysin quadruple.
	\item Let $\zeta\in\K(S)$. Then $(X,Y,Z,\xi)$ is a Gysin quadruple if and only if $(X,Y,Z,\xi+\zeta)$ is.
	\item Let $B$ be the blowup of $X$ in $Y$ and let $E$ be the exceptional divisor. Then $(X,Y,Z,\xi)$ is a Gysin quadruple if and only if $(B,E,Z,\xi)$ is.
	\item Let $D$ be a smooth divisor on $X$ with $D\cap Z=\emptyset$ and let $\scr E\in\Vect(X)$. Then $(X,D,Z,\scr E)$ is a Gysin quadruple if and only if $(X,D,Z,\scr E(D))$ is.
\end{enumerate}
\end{proposition}

\begin{proof}
	(i) This is exactly Theorem~\ref{thm:tang}(iii).
	
	(ii) This follows from the $\MS_S$-linearity of Gysin maps.
	
	(iii) This follows from smooth blowup excision.
	
	(iv) Consider the zigzag of isomorphisms
	\[
	\Th_X(\scr E)/\Th_D(\scr E) \xleftarrow{\sim} B/\partial B\xrightarrow{\sim} \Th_X(\scr E(D))/\Th_D(\scr E(D))
	\]
	from Construction~\ref{ctr:thomtwist}, which is obtained by taking the total cofibers of the following cubes:
	\begin{gather*}
	\begin{tikzcd}[row sep={25,between origins}, column sep={40,between origins}, ampersand replacement=\&]
		\emptyset \ar{rr} \ar{dr} \ar{dd} \&\& \emptyset \ar{dr} \ar{dd} \\
		\& \P_D(\scr E) \ar[crossing over]{rr} \& \& \P_D(\scr E(D)\oplus \scr O) \ar{dd}  \\
		\P_D(\scr E) \ar{rr} \ar{dr} \&\& \P_X(\scr E) \ar{dr} \\
		\& \P_{\P_D(\scr E)}(\scr O(1)\oplus\scr O) \ar{rr} \ar[from=uu,crossing over] \& \& B
	\end{tikzcd}
	\qquad\qquad
	\\
	\swarrow\qquad\qquad\qquad\qquad\qquad\searrow\qquad\qquad\qquad
	\\
	\begin{tikzcd}[row sep={25,between origins}, column sep={40,between origins}, ampersand replacement=\&]
		\emptyset \ar{rr} \ar{dr} \ar{dd} \&\& \emptyset \ar{dr} \ar{dd} \\
		\& \P_D(\scr O) \ar[crossing over]{rr} \& \& \P_D(\scr O) \ar{dd}  \\
		\P_D(\scr E) \ar{rr} \ar{dr} \&\& \P_X(\scr E) \ar{dr} \\
		\& \P_D(\scr E\oplus\scr O) \ar{rr} \ar[from=uu,crossing over] \& \& \P_X(\scr E\oplus\scr O)
	\end{tikzcd}
	\qquad\qquad
	\begin{tikzcd}[row sep={25,between origins}, column sep={40,between origins}, ampersand replacement=\&]
		\P_D(\scr E) \ar{rr} \ar{dr} \ar{dd} \&\& \P_D(\scr E) \ar{dr} \ar{dd} \\
		\& \P_D(\scr E) \ar[crossing over]{rr} \& \& \P_D(\scr E(D)\oplus \scr O) \ar{dd}  \\
		\P_D(\scr E) \ar{rr} \ar{dr} \&\& \P_X(\scr E) \ar{dr} \\
		\& \P_D(\scr E) \ar{rr} \ar[from=uu,crossing over] \& \& \P_X(\scr E(D)\oplus\scr O)\rlap.
	\end{tikzcd}
	\end{gather*}
	Since $D\cap Z=\emptyset$, the cube 
	\[
	\begin{tikzcd}[row sep={25,between origins}, column sep={40,between origins}]
		\emptyset \ar{rr} \ar{dr} \ar{dd} && \emptyset \ar{dr} \ar{dd} \\
		& \emptyset \ar[crossing over]{rr} & & \emptyset \ar{dd}  \\
		\emptyset \ar{rr} \ar{dr} && \P_Z(\scr E) \ar{dr} \\
		& \emptyset \ar{rr} \ar[from=uu,crossing over] & & \P_Z(\scr E\oplus\scr O)
	\end{tikzcd}
	\]
	maps compatibly via closed immersions to each of the cubes above, defining a lift of the above diagram $\Tw(\Delta^1)\times(\Delta^1)^3\to \Sm_S$ to $\Pair_S$.	
	Applying the functor $\Pair_S\to\MS_S^{\Delta^1}$ from Theorem~\ref{thm:tang}(i) to this diagram and taking the total cofibers of the resulting six cubes, we obtain a commutative diagram
	\[
	\begin{tikzcd}
		\Th_X(\scr E)/\Th_D(\scr E) \ar{d}{\gys} & B/\partial B \ar{l}[swap]{\sim} \ar{r}{\sim} \ar{d}{\gys} & \Th_X(\scr E(D))/\Th_D(\scr E(D)) \ar{d}{\gys} \\
		\Th_Z(\scr N_Z\oplus\scr E) & \Th_Z(\scr N_Z\oplus\scr E) \ar[equal]{l} \ar[equal]{r} & \Th_Z(\scr N_Z\oplus\scr E)
	\end{tikzcd}
	\]
	in $\MS_S$, where the first and last vertical maps are induced by the given null-sequence for $\xi=\scr E$ and $\xi=\scr E(D)$, respectively. Thus, the first vertical map is an isomorphism if and only if the last one is.
\end{proof}

\begin{corollary}\label{cor:gys-cofib}
	Let $\scr E$ and $\scr F$ be finite locally free sheaves on $S$ and let $\xi\in\K(\P(\scr E \oplus\scr F))$. Then the null-sequence 
	\[
	\Th_{\P(\scr F)}(\xi) \xrightarrow{\mathrm{inc}} \Th_{\P(\scr E\oplus\scr F)}(\xi) \xrightarrow{\gys} \Th_{\P(\scr E)}(\scr F(-1)+\xi)
	\]
	from Construction~\ref{ctr:null-sequences} is a cofiber sequence in the following cases:
	\begin{enumerate}
		\item $\xi=\scr H(n)$ for some $\scr H\in \Vect(S)$ and $n\in\Z$;
		\item $\xi=-\Omega_{\P(\scr E\oplus\scr F)}-\scr H(-1)$ for some $\scr H\in\Vect(S)$.
	\end{enumerate}
\end{corollary}


\begin{proof}
	In each case, we show that we can reach $\xi$ using the moves from Proposition~\ref{prop:gys-cofib}. Let $b\colon B=\P_{\P(\scr E)}(\scr F(-1)\oplus\scr O)\to\P(\scr E\oplus\scr F)$ be the blowup with exceptional divisor $E=\P_{\P(\scr E)}(\scr F)$ as in Lemma~\ref{lem:projquot}. By Proposition~\ref{prop:gys-cofib}(iii), it suffices to show that $(B,E,\P(\scr E),b^*(\xi))$ is a Gysin quadruple (which now lives over $\P(\scr E)$). By Proposition~\ref{prop:gys-cofib}(i), we know that $(B,E,\P(\scr E),0)$ is a Gysin quadruple.
		Note that \[b^*(\scr O(1))=p^*(\scr O(1))(E),\] 
		where $p\colon B\to \P(\scr E)$ is the structure map.
	
	(i) We have $b^*(\scr H(n))=p^*(\scr H(n))(nE)$, so we can conclude using Proposition~\ref{prop:gys-cofib}(ii,iv).
	
	(ii) Let $\Omega=\Omega_{\P(\scr E\oplus\scr F)/S}$. We consider the Euler fiber sequence over $\P(\scr E\oplus\scr F)$ and its dual:
	\begin{gather}
		\label{eqn:euler1}\Omega\to (\scr E\oplus\scr F)(-1) \to\scr O,\\
		\label{eqn:euler2}\scr O(-1) \to (\scr E\oplus\scr F)^\vee \to \Omega^\vee(-1).
	\end{gather}
	These imply the equations
	\begin{align*}
	-\Omega &= \scr O-(\scr E\oplus\scr F)(-1),\\
	 -\scr O(-1) &= \Omega^\vee(-1)-(\scr E\oplus\scr F)^\vee,
	\end{align*}
	whence
	\[
	-\Omega-\scr H(-1) = (\scr E\oplus\scr F\oplus\scr H)\otimes \Omega^\vee(-1)+\scr O-(\scr E\oplus\scr F\oplus\scr H)\otimes(\scr E\oplus\scr F)^\vee
	\]
	in $\K(\P(\scr E \oplus\scr F))$. By Proposition~\ref{prop:gys-cofib}(ii), it will suffice to show that \[(\P(\scr E\oplus\scr F),\P(\scr F),\P(\scr E),\scr H\otimes\Omega^\vee(-1))\] is a Gysin quadruple for any $\scr H\in\Vect(S)$. By Proposition~\ref{prop:gys-cofib}(iii,iv), this is equivalent to the statement that \[(B,E,\P(\scr E),b^*(\scr H\otimes\Omega^\vee(-1))(E))\] is a Gysin quadruple.
	If we pull back the fiber sequence~\eqref{eqn:euler2} to the blowup and tensor it with $\scr H(E)$, we obtain the equation
	\[
	b^*(\scr H\otimes \Omega^\vee(-1))(E)=(\scr H\otimes (\scr E\oplus\scr F)^\vee)(E)-p^*(\scr H\otimes \scr O(-1))
	\]
	in $\K(B)$. We now conclude using Proposition~\ref{prop:gys-cofib}(ii,iv).
\end{proof}

\begin{theorem}\label{thm:dualsofP}
The comparison maps
\[
\rm{comp}_{\P^n,\scr O(-1)^m} \colon \Th_{\P^n}(- \scr O(-1)^{m} - \Omega_{\P^n}) \to \Th_{\P^n}(\scr O(-1)^{m})^\vee
\]
are isomorphisms for all $n\geq -1$ and $m\geq 0$.
\end{theorem}

\begin{proof}
We use induction on $n$, the cases $n=-1$ and $n=0$ being trivial (the case $n=0$ uses that the Gysin map of the identity embedding is the identity).
By Proposition~\ref{prop:null-sequences}, the comparison maps assemble into a morphism of null-sequences
\[
\begin{tikzcd}
	\Th_{\P^{n-1}}(- \scr O(-1)^{m}-\Omega_{\P^n}) \ar{r}{\mathrm{inc}} \ar{d} & \Th_{\P^n}(- \scr O(-1)^{m}- \Omega_{\P^n}) \ar{r}{\gys} \ar{d} & \Th_S(-\scr O^m) \ar{d} \\
	\Th_{\P^{n-1}}(\scr O(-1)^{m+1})^\vee \ar{r}{\gys^\vee} & \Th_{\P^n}(\scr O(-1)^{m})^\vee \ar{r}{\mathrm{inc}^\vee} & \Th_S(\scr O^m)^\vee\rlap,
\end{tikzcd}
\]
in which both sequences are cofiber sequences by Corollary~\ref{cor:gys-cofib}. By induction hypothesis, the left and right vertical maps are isomorphisms, hence also the middle one.
\end{proof}

\section{Ambidexterity and Atiyah duality}

Let $f\colon X\to S$ be a smooth morphism of derived schemes. We define the functor $f_!\colon \MS_X\to \MS_S$ by
\[
f_!=f_\sharp \circ \Sigma^{-\Omega_{f}}.
\]
It has a right adjoint $f^!\colon \MS_S\to \MS_X$ given by
\[
f^!=\Sigma^{\Omega_f}\circ f^*.
\]
Using the naturality and monoidality of the J-homomorphism $\xi\mapsto \Sigma^\xi$, we deduce the following properties:
\begin{itemize}
	\item (Functoriality) Let $g\colon Y\to X$ be another smooth morphism. Then there is a canonical isomorphism
	\[
	(fg)_!\simeq f_!g_!.
	\]
	This uses the fiber sequence $g^*\Omega_f\to \Omega_{fg}\to \Omega_g$ in $\Perf(Y)$.
	\item (Base change) For any cartesian square
	\[
	\begin{tikzcd}
		X' \ar{r}{h} \ar{d}[swap]{f'} & X \ar{d}{f} \\
		S' \ar{r}{g} & S\rlap,
	\end{tikzcd}
	\]
	 there is a canonical isomorphism
	\[
	g^*f_!\simeq f'_!h^*.
	\]
	\item (Projection formula) The functor $f_!$ is $\MS_S$-linear. In particular, for any $A\in \MS_X$ and $B\in \MS_S$, there is a canonical isomorphism
	\[
	f_!(A\otimes f^*(B)) \simeq f_!(A)\otimes B.
	\]
\end{itemize}
These properties are clearly coherent, in the sense that they define a lax symmetric monoidal functor
\[
\MS^*_!\colon\Span(\Sch,\all,\mathrm{smooth})\to \Cat_\infty, \quad (X\xleftarrow{f} Z\xrightarrow{g} Y)\mapsto g_!f^*.
\]
This coherence will not be needed in the sequel, however.

\begin{construction}[Gysin transformation]
\label{ctr:gysin!}
Consider a commutative triangle
\[
\begin{tikzcd}
	Z \ar[hook]{r}{i} \ar{dr}[swap]{g} & X \ar{d}{f} \\
	& S\rlap,
\end{tikzcd}
\]
where $f$ and $g$ are smooth and $i$ is a closed immersion. There is then a canonical fiber sequence
\[
\scr N_i\to i^*(\Omega_f) \to \Omega_g
\]
in $\Perf(Z)$, inducing an isomorphism $i^*(\Omega_f)\simeq \scr N_i+\Omega_g$ in $\K(Z)$, whence and isomorphism
\[
\Sigma^{-\Omega_g} \simeq \Sigma^{\scr N_i}\Sigma^{-i^*(\Omega_f)}
\]
of endofunctors of $\MS_Z$. 
Using this isomorphism, the Gysin transformation $\mathrm{gys}(f,i)\colon f_\sharp\to g_\sharp \Sigma^{\scr N_i}i^*$ of Construction~\ref{ctr:gysin} may be rewritten as
\[
\mathrm{gys}(f,i)\colon f_! \to g_!i^*.
\]
\end{construction}

\begin{remark}[Properties of Gysin transformations]
	For convenience, we rewrite the main properties of Gysin transformations (Proposition~\ref{prop:gysin-trans}) using the shriek functors:
	\begin{enumerate}
		\item (Base change) Given a cartesian diagram
		\[
		\begin{tikzcd}
			W \ar[hook]{r}{k} \ar{d}{c} & Y \ar{r}{g} \ar{d}{b} & T \ar{d}{a} \\
			Z \ar[hook]{r}{i} & X \ar{r}{f} & S\rlap,
		\end{tikzcd}
		\]
		where $i$ is a closed immersion and $f$ and $fi$ are smooth, the following square commutes:
		\[
		\begin{tikzcd}[column sep=4em]
			g_!b^* \ar{r}{\mathrm{gys}(g,k)b^*} \ar{d}[sloped]{\sim}[swap]{\mathrm{BC}} & (gk)_!c^*i^* \ar{d}{\mathrm{BC}}[swap,sloped]{\sim} \\
			a^*f_! \ar{r}{a^*\mathrm{gys}(f,i)} & a^*(fi)_!i^*\rlap.
		\end{tikzcd}
		\]
		\item (Base independence) Given morphisms
		\[
		\begin{tikzcd}
			Z \ar[hook]{r}{i} & X \ar{r}{f} & S \ar{r}{a} & T\rlap,
		\end{tikzcd}
		\]
		where $i$ is a closed immersion and $f$, $fi$, and $a$ are smooth, the following square commutes:
		\[
		\begin{tikzcd}[column sep=4em]
			a_!f_! \ar{r}{a_!\mathrm{gys}(f,i)} \ar{d}[swap,sloped]{\sim} & a_!(fi)_!i^* \ar{d}[sloped]{\sim} \\
			(af)_! \ar{r}{\mathrm{gys}(af,i)} & (afi)_!i^*\rlap.
		\end{tikzcd}
		\]
		\item (Linearity) Given morphisms
		\[
		\begin{tikzcd}
			Z \ar[hook]{r}{i} & X \ar{r}{f} & S\rlap,
		\end{tikzcd}
		\]
		where $i$ is a closed immersion and $f$ and $fi$ are smooth, and given $A\in\MS_X$ and $B\in \MS_S$, the following square commutes:
		\[
		\begin{tikzcd}[column sep=4em]
			f_!(A\otimes f^*(B)) \ar{d}[sloped]{\sim}[swap]{\mathrm{PF}} \ar{r}{\mathrm{gys}(f,i)} & (fi)_!(i^*(A)\otimes (fi)^*(B)) \ar{d}{\mathrm{PF}}[swap,sloped]{\sim} \\
			f_!(A) \otimes B \ar{r}{\mathrm{gys}(f,i)\otimes\id} & (fi)_!i^*(A)\otimes B\rlap.
		\end{tikzcd}
		\]
		\item (Composition of closed immersions)
		The map $\mathrm{gys}(f,\id)$ is the identity, and given morphisms
		\[
		\begin{tikzcd}
			W \ar[hook]{r}{k} & Z \ar[hook]{r}{i} & X \ar{r}{f} & S\rlap,
		\end{tikzcd}
		\]
		where $i$ and $k$ are closed immersions and $f$, $fi$, and $fik$ are smooth, the following triangle commutes:
		\[
		\begin{tikzcd}[column sep=4em]
			f_! \ar{r}{\mathrm{gys}(f,i)} \ar{dr}[swap]{\mathrm{gys}(f,ik)} & (fi)_!i^* \ar{d}{\mathrm{gys}(fi,k)i^*} \\
			& (fik)_!k^*i^*\rlap.
		\end{tikzcd}
		\]
		Note that the isomorphism in Proposition~\ref{prop:gysin-trans}(iv) disappears due to the following composition in $\K(W)$ being the identity:
		\[
		-\Omega_{fik} \simeq -\Omega_f+\scr N_{ik} \simeq -\Omega_f+\scr N_i+\scr N_k \simeq -\Omega_{fi} +\scr N_k \simeq -\Omega_{fik}.
		\]
		(This relation is witnessed by the 3-simplex $\Omega_{f} \to \Omega_{fi} \to \Omega_{fik}$ in the $S_\bullet$-construction.)
	\end{enumerate}
\end{remark}

\begin{construction}[Trace map]
	\label{ctr:trace}
	Let $f\colon X\to S$ be a smooth separated morphism and consider the diagram
	\[
	\begin{tikzcd}
	X \ar[hook]{dr}[description]{\delta} \ar[bend left=20]{drr}[description]{\id_X} \ar[bend right]{ddr}[description]{\id_X} & & \\
	& X\times_SX \ar{d}[swap]{\pi_1} \ar{r}{\pi_2} & X \ar{d}{f} \\
	& X \ar{r}[swap]{f} & S\rlap,
	\end{tikzcd}
	\]
	where the diagonal $\delta$ is a closed immersion.
	We define the \emph{trace}
	\[
	\epsilon_f\colon f^*f_! \to \id_{\MS_X}
	\]
	as the composition
	\[
	f^*f_!\overset{\mathrm{BC}}{\simeq} \pi_{2!}\pi_1^*
	\xrightarrow{\mathrm{gys}(\pi_2,\delta)}
	\id_{X!}\delta^*\pi_1^*\simeq \id_{\MS_X}.
	\]
	Note that $\epsilon_f$ is an $\MS_S$-linear transformation.
	By adjunction, $\epsilon_f$ is equivalent to a natural transformation
	\[
	\alpha_f\colon f_!\to f_*.
	\]
\end{construction}

\begin{definition}
	Let $f\colon X\to S$ be a smooth separated morphism. We say that $f$ is \emph{$\MS$-ambidextrous} if the trace map $\epsilon_f\colon f^*f_!\to\id$ of Construction~\ref{ctr:trace} is the counit of an adjunction $f^*\dashv f_!$, or equivalently if $\alpha_f\colon f_!\to f_*$ is an isomorphism.
\end{definition}

\begin{lemma}\label{lem:duality-pairing}
	Let $f\colon X\to S$ be a smooth separated morphism and let $\xi\in\K(X)$. Consider the pairing
	\[
	\ev_{X,\xi}\colon \Th_X(\xi) \otimes \Th_X(-\xi-\Omega_f) \to \1_S
	\]
	of Construction~\ref{ctr:ev}.
	\begin{enumerate}
		\item The induced map $\mathrm{comp}_{X,\xi}\colon \Th_X(-\xi-\Omega_f)\to\Th_X(\xi)^\vee$ coincides with $(\alpha_f)_{\Sigma^{-\xi}\1_X}$.
		\item The other induced map $\Th_X(\xi)\to \Th_X(-\xi-\Omega_f)^\vee$ coincides with $(\alpha_f)_{\Sigma^{\xi+\Omega_f}\1_X} \langle -1\rangle^{\rk\Omega_f}$.
	\end{enumerate}
\end{lemma}

\begin{proof}
	Note that there is an isomorphism $\Th_X(\xi)^\vee\simeq f_*\Sigma^{-\xi}\1_X$ induced by the following pairing:
	\begin{equation}\label{eqn:weak-dual}
	f_\sharp\Sigma^\xi\1_X \otimes f_*\Sigma^{-\xi}\1_X\overset{\mathrm{PF}}{\simeq} f_\sharp \Sigma^\xi f^*f_*\Sigma^{-\xi}\1_X \xrightarrow{\epsilon} f_\sharp\Sigma^{\xi}\Sigma^{-\xi}\1_X=f_\sharp f^*\1_S \xrightarrow{\epsilon} \1_S.
	\end{equation}
	Assertion (i) is thus equivalent to the commutativity of the triangle
	\[
	\begin{tikzcd}[row sep=1em]
		f_\sharp\Sigma^\xi\1_X \otimes f_!\Sigma^{-\xi}\1_X \ar{dr}{\ev_{X,\xi}} \ar{dd}[swap]{\id\otimes\alpha_f} & \\
		& \1_S\rlap, \\
		f_\sharp\Sigma^\xi\1_X \otimes f_*\Sigma^{-\xi}\1_X \ar{ur}[swap]{\eqref{eqn:weak-dual}} &
	\end{tikzcd}
	\]
	which follows at once from the definitions of $\ev_{X,\xi}$ and of $\alpha_f$.
	Assertion (ii) follows from (i) and Proposition~\ref{prop:skew-symmetry}.
\end{proof}

\begin{proposition}[From duality to ambidexterity]
	\label{prop:dual->ambi}
	Let $f\colon X\to S$ be a smooth separated morphism such that:
	\begin{enumerate}
		\item $\Sigma^\infty_{\P^1}X_+\in \MS_S$ is dualizable;
		\item the map $\mathrm{comp}_{X,0}\colon \Th_X(-\Omega_f)\to (\Sigma^\infty_{\P^1} X_+)^\vee$ is an isomorphism.
	\end{enumerate}
	Then $f$ is $\MS$-ambidextrous.
\end{proposition}

\begin{proof}
	We adapt an argument of Ayoub \cite[Proposition 1.7.16]{Ayoub}.
	We first show that the maps
		\begin{align*}
		\alpha_f f^*&\colon f_!f^* \to f_*f^*, \\
		\alpha_f f^! &\colon f_!f^! \to f_*f^!
		\end{align*}
	are isomorphisms. Both are $\MS_S$-linear natural transformations between lax $\MS_S$-linear functors. We claim that all four functors $f_!f^*$, $f_!f^!$, $f_*f^*$, and $f_*f^!$ are in fact strictly $\MS_S$-linear. This is clear for the first two functors. For the last two functors, we have
	\[
	f_*f^* = \Hom(\Sigma^\infty_{\P^1}X_+,\ph)\quad\text{and}\quad f_*f^! = \Hom(\Th_X(-\Omega_f), \ph),
	\]
	so their $\MS_S$-linearity follows from the given dualizability of $\Sigma^\infty_{\P^1}X_+$ and $\Th_X(-\Omega_f)$.
	Hence, it is enough to show that the maps $\alpha_f f^*$ and $\alpha_f f^!$ are isomorphisms on the unit object $\1_S$. For the first map this follows from Lemma~\ref{lem:duality-pairing}(i) and Assumption (ii). By Lemma~\ref{lem:duality-pairing}(ii) and Assumption (i), the second map is dual to the first up to the automorphism $\langle -1\rangle^{\rk\Omega_f}$ of $f_!f^!\1_S$, hence it is also an isomorphism.
	
	We now show that $\alpha_f\colon f_!\to f_*$ has an inverse on both sides. 
	On the one hand, the composite
	\[
	f_*\xrightarrow{\eta f_*} f_* f^*f_* \xrightarrow{(\alpha_f f^*)^{-1}} f_!f^*f_*\xrightarrow{f_!\epsilon} f_!
	\]
	is right inverse to $\alpha_f$ by the triangle identity for $f^*\dashv f_*$:
	\[
	\begin{tikzcd}
		f_* \ar{r}{\eta f_*} \ar{dr}[swap]{\id} & f_*f^*f_* \ar[<-]{r}{\alpha_f}[swap]{\sim} \ar{d}{f_*\epsilon} & f_!f^*f_* \ar{d}{f_!\epsilon} \\
		& f_* \ar[<-]{r}[swap]{\alpha_f} & f_!\rlap.
	\end{tikzcd}
	\]
	Similarly, the composite
	\[
	f_* \xrightarrow{f_*\eta} f_* f^!f_! \xrightarrow{(\alpha_f f^!)^{-1}} f_!f^!f_! \xrightarrow{\epsilon f_!} f_!
	\]
	is left inverse to $\alpha_f$ by the triangle identity for $f_!\dashv f^!$:
	\[
	\begin{tikzcd}[baseline=(LastCell.base)]
		f_* \ar[<-]{r}{\alpha_f} \ar{d}[swap]{f_*\eta} & f_! \ar{d}[swap]{f_!\eta} \ar{dr}{\id} \\
		f_* f^! f_! \ar[<-]{r}{\sim}[swap]{\alpha_f} & f_!f^!f_! \ar{r}[swap]{\epsilon f_!} & |[alias=LastCell]| f_!\rlap.
	\end{tikzcd}\qedhere
	\]
\end{proof}

\begin{lemma}\label{lem:PF-BC-!}
	Consider a cartesian square
	\[
	\begin{tikzcd}
	Z \ar{d}[swap]{\bar g} \ar{r}{\bar f} \ar{dr}[description]{h} & Y \ar{d}{g} \\
	X \ar{r}[swap]{f} & S\rlap,
	\end{tikzcd}
	\]
	where $f$ and $g$ are smooth.
	Then the following diagram commutes for all $A\in \MS_X$ and $B\in\MS_Y$:
	\[
	\begin{tikzcd}[column sep=2.5em]
	f_! (A\otimes f^*g_! (B)) \ar{d}[sloped]{\sim}[swap]{\mathrm{PF}} \ar[<-]{r}{\mathrm{BC}}[swap]{\sim} & f_! (A\otimes \bar g_! \bar f^*(B))  \\
	f_! (A)\otimes g_! (B) & h_! (\bar g^*(A)\otimes \bar f^*(B)) \ar{d}{\mathrm{PF}}[swap,sloped]{\sim} \ar{u}[sloped]{\sim}[swap]{\mathrm{PF}}  \\
	g_! (g^*f_! (A)\otimes B) \ar{u}{\mathrm{PF}}[swap,sloped]{\sim} \ar[<-]{r}{\mathrm{BC}}[swap]{\sim} & g_! (\bar f_! \bar g^*(A)\otimes B)\rlap.
	\end{tikzcd}
	\]
\end{lemma}

\begin{proof}
This follows from the lax monoidal structure of 
\[
\h\MS^*_{!}\colon \Span(\Sch,\all,\mathrm{smooth})\to\Cat_1, \quad (X\xleftarrow{f} Z\xrightarrow{g} Y)\mapsto g_{!} f^*.\qedhere
\]
\end{proof}

Following \cite[Définition 5.5.2]{EGA2}, we say that a morphism of derived schemes $f\colon X\to S$ is \emph{projective} if it factors as
\[
X\into \P(\scr E) \to S,
\]
where the first map is a closed immersion and $\scr E\in \QCoh(S)_{\geq 0}$ is a quasi-coherent sheaf of finite type (i.e., perfect to order $0$).
We say that $f\colon X\to S$ is \emph{locally projective} if this holds locally on $S$. Note that a locally projective morphism is proper.
Note also that if $\scr E\in\QCoh(S)_{\geq 0}$ is of finite type, there exists a surjection $\scr O^{n+1}\onto \scr E$ locally on $S$.
Hence, if $f$ is locally projective, then it factors as $X\into \P^n_S\to S$ locally on $S$.

\begin{theorem}[Ambidexterity for smooth projective morphisms]
	\label{thm:ambidexterity}
	Let $f\colon X\to S$ be smooth and locally projective. Then $f$ is $\MS$-ambidextrous.
\end{theorem}

\begin{proof}
	For $X=\P^n_S$, the claim follows from Theorem~\ref{thm:dualsofP} (with $m=0$) and Proposition~\ref{prop:dual->ambi}, noting that $\P^n_+$ is dualizable in $\MS_S$ (since $\P^n/\P^{n-1}$ is invertible).
	By Zariski descent, we may thus assume given a factorization
	\[
	\begin{tikzcd}
		X \ar[hook]{r}{i} \ar{dr}[swap]{f} & P \ar{d}{p} \\ 
		& S\rlap,
	\end{tikzcd}
	\]
	where $p$ is $\MS$-ambidextrous and $i$ is a closed immersion. 
	By definition of $\MS$-ambidexterity, the trace map $\epsilon_p\colon p^*p_!\to\id_{\MS_P}$ is the counit of an adjunction $p^*\dashv p_!$.
	Let $\eta_p\colon \id_{\MS_S}\to p_!p^*$ be a compatible unit. Since $\epsilon_p$ is $\MS_S$-linear, so is $\eta_p$.
	We then define
	\[
	\eta_f\colon \id_{\MS_S} \to f_!f^*
	\]
	as the composition
	\[
	\id_{\MS_S}\xrightarrow{\eta_p} p_!p^*\xrightarrow{\mathrm{gys}(p,i)} f_!f^*.
	\]
	As both $\eta_p$ and $\mathrm{gys}(p,i)$ are $\MS_S$-linear, $\eta_f$ is $\MS_S$-linear.
	We now show that $\eta_f$ and $\epsilon_f$ are the unit and counit of an adjunction $f^*\dashv f_!$ by verifying the triangle identities, i.e., that the following composites are equal to the identity:
	\begin{gather}
		\label{eqn:triangle-1} f^* \xrightarrow{f^*\eta_f} f^*f_!f^* \xrightarrow{\epsilon_f f^*} f^*, \\
		\label{eqn:triangle-2} f_! \xrightarrow{\eta_ff_!} f_!f^*f_! \xrightarrow{f_!\epsilon_f} f_!.
	\end{gather}
	
	We first show that~\eqref{eqn:triangle-1} is the identity.
	By assumption, the composition
	\[
	p^* \xrightarrow{p^*\eta_p} p^*p_!p^* \xrightarrow{\epsilon_pp^*} p^*
	\]
	is the identity. The lower composition in the following diagram is thus the identity, while the upper composition is~\eqref{eqn:triangle-1}:
	\[
	\begin{tikzcd}[column sep=3em,row sep=2em]
		f^* \ar{dr}[swap]{f^*\eta_p} \ar{r}{f^*\eta_f} & f^*f_!f^* \ar[<-]{r}{\mathrm{BC}}[swap]{\sim} \ar[bend left=45]{drr}{\epsilon_ff^*} & \pi_{2!}\pi_1^*f^* \ar{dr}[sloped]{\mathrm{gys}} & \\
		 & f^*p_!p^* \ar{u}{\mathrm{gys}} \ar[<-]{r}{\mathrm{BC}}[swap]{\sim} \ar[<-]{dr}[sloped]{\sim}[swap,sloped]{\mathrm{BC}} \ar[bend right=45]{drr}[swap]{i^*\epsilon_pp^*} & \bar p_!\bar f^*p^* \ar{u}{\mathrm{gys}} \ar{r}[swap]{\mathrm{gys}} \ar{d}{\mathrm{BC}}[swap,sloped]{\sim} & f^* \ar[equal]{d}{\mathrm{BC}} \\
		& & i^*\pi_{2!}\pi_1^*p^* \ar{r}[swap]{\mathrm{gys}} & i^*p^*\rlap.
	\end{tikzcd}
	\]
	This diagram commutes, showing that~\eqref{eqn:triangle-1} is the identity:
	\begin{itemize}
		\item the left triangle commutes by definition of $\eta_f$;
		\item the two squares commute by the compatibility of Gysin maps with base change, applied to the following two cartesian diagrams:
		\[
		\begin{tikzcd}[column sep=1.5em]
			X\times_SX \ar[hook]{r} \ar{d}[swap]{\pi_1} & P\times_SX \ar{r}{\bar p} \ar{d}[swap]{\bar f} & X \ar{d}{f} \\
			X \ar[hook]{r}[swap]{i} & P \ar{r}[swap]{p} & S\rlap,
		\end{tikzcd}
		\qquad\qquad
		\begin{tikzcd}[column sep=1.5em]
			X \ar[hook]{r} \ar{d}[swap]{i} & P\times_SX \ar{r}{\bar p} \ar{d}[swap]{\id\times i} & X \ar{d}{i} \\
			P \ar[hook]{r}[swap]{\delta} & P\times_SP \ar{r}[swap]{\pi_2} & P\rlap;
		\end{tikzcd}
		\]
		\item the triangle of Gysin maps is induced by the composition of closed immersions
		\[
		\begin{tikzcd}
			X \ar[hook]{r}{\delta} \ar[bend right=12]{drr}{\id} & X\times_SX \ar[hook]{r}{i\times\id} \ar{dr}{\pi_2} & P\times_S X \ar{d}{\bar p} \\
			& & X\rlap;
		\end{tikzcd}
		\]
		\item the triangle of base change isomorphisms is induced by the cartesian squares
		\[
		\begin{tikzcd}
			P\times_SX \ar{r}{\bar p} \ar{d}{\id\times i} \ar[bend right=50]{dd}[swap]{\bar f} & X \ar{d}[swap]{i} \ar[bend left=50]{dd}{f} \\
			P\times_SP \ar{r}{\pi_2} \ar{d}{\pi_1} & P \ar{d}[swap]{p} \\
			P \ar{r}{p} & S\rlap.
		\end{tikzcd}
		\]
	\end{itemize}
	
	We now show that~\eqref{eqn:triangle-2} is the identity. Instead of proving this directly, we will reduce the claim to the first triangle identity using formal properties of Gysin transformations. 
	Let $\iota$ be the transformation~\eqref{eqn:triangle-1} evaluated on $\1_S$:
	\[
	\iota\colon \1_X \xrightarrow{f^*\eta_f} f^*f_!(\1_X)\xrightarrow{\epsilon_f} \1_X.
	\]
	We claim that the endomorphism~\eqref{eqn:triangle-2} can be identified with $f_!(\iota\otimes\mathord{-})$. As we have already shown that $\iota$ is the identity, this will complete the proof.
	
	In the following diagram, the upper composition is~\eqref{eqn:triangle-2} while the lower composition is $f_!(\iota\otimes\ph)$:
	\[
	\begin{tikzcd}[column sep=4.2em]
	f_! \ar{d}[sloped]{\sim} \ar[end anchor={[shift={(-3.8em,0)}]west}]{r}{\eta_ff_!} & \llap{$f_!f^*f_!\simeq\,$} f_!(\1_X\otimes f^*f_!(\ph)) \ar{d}{\mathrm{PF}}[swap,sloped]{\sim} \ar[<-]{r}{\mathrm{BC}}[swap]{\sim} \ar[bend left=15]{rr}{f_!\epsilon_f} & f_!(\1_X\otimes \pi_{2!}\pi_1^*(\ph)) \ar{r}{\mathrm{gys}(\pi_2,\delta)} & f_!(\1_X\otimes \ph) \\
	\1_S\otimes f_!(\ph) \ar{r}{\eta_f\otimes \id} & f_!(\1_X)\otimes f_!(\ph) & h_!(\pi_2^*(\1_X)\otimes \pi_1^*(\ph)) \ar{r}{\mathrm{gys}(h,\delta)} \ar{d}{\mathrm{PF}}[swap,sloped]{\sim} \ar{u}[sloped]{\sim}[swap]{\mathrm{PF}} & f_!(\1_X\otimes \ph) \ar[equal]{u} \ar[equal]{d} \\
	f_!(f^*(\1_S)\otimes \ph) \ar{u}{\mathrm{PF}}[swap,sloped]{\sim} \ar{r}[swap]{f_!(f^*\eta_f\otimes\id)} & f_!(f^*f_!(\1_X)\otimes \ph) \ar{u}[sloped]{\sim}[swap]{\mathrm{PF}} \ar[<-]{r}{\mathrm{BC}}[swap]{\sim} \ar[equal]{d} & f_!(\pi_{1!}\pi_2^*(\1_X)\otimes \ph) \ar{r}{\mathrm{gys}(\pi_1,\delta)} \ar{d}[sloped]{\sim}[swap]{\mathrm{BC}} & f_!(\1_X\otimes \ph) \ar[equal]{d} \\
	& f_!(f^*f_!(\1_X)\otimes \ph) \ar[<-]{r}{\mathrm{BC}}[swap]{\sim} \ar[bend right=15]{rr}[swap]{f_!(\epsilon_f\otimes \id)} & f_!(\pi_{2!}\pi_1^*(\1_X)\otimes \ph) \ar{r}{\mathrm{gys}(\pi_2,\delta)} & f_!(\1_X\otimes \ph)\rlap.
	\end{tikzcd}
	\]
	This diagram commutes:
	\begin{itemize}
		\item the top left rectangle commutes by the $\MS_S$-linearity of the transformation $\eta_f$;
		\item the rectangle directly below commutes by the naturality of the projection isomorphisms;
		\item the large rectangle commutes by Lemma~\ref{lem:PF-BC-!};
		\item the two upper right rectangles commute by the $\MS$-linearity and base independence of Gysin transformations;
		\item the rectangle of base change isomorphisms is induced by the cartesian squares
		\[
		\begin{tikzcd}
			X\times_SX \ar{d}[swap]{\sigma} \ar{r}{\pi_1} \ar[bend right=50]{dd}[swap]{\pi_2} & X \ar{d}{\id} \ar[bend left=50]{dd}{f} \\
			X\times_SX \ar{d}[swap]{\pi_1} \ar{r}{\pi_2} & X \ar{d}{f} \\
			X \ar{r}[swap]{f} & S\rlap,
		\end{tikzcd}
		\]
		where $\sigma$ is the swap map;
		\item the bottom right rectangle commutes by the compatibility of Gysin transformations with base change, applied to the following cartesian diagram:
		\[
		\begin{tikzcd}
			X \ar[hook]{r}{\delta} \ar{d}[swap]{\id} & X\times_SX \ar{d}[swap]{\sigma} \ar{r}{\pi_1} & X \ar{d}{\id} \\
			X \ar[hook]{r}{\delta} & X\times_SX \ar{r}{\pi_2} & X\rlap.
		\end{tikzcd}
		\]
	\end{itemize}
	This completes the proof.
\end{proof}

\begin{lemma}
	\label{lem:smooth-proper-BC}
	Let $f\colon X\to S$ be smooth and separated.
	\begin{enumerate}
		\item For any cartesian square
		\[
		\begin{tikzcd}
			Y \ar{r}{v} \ar{d}[swap]{g} & X \ar{d}{f} \\
			T \ar{r}{u} & S\rlap,
		\end{tikzcd}
		\]
		the following diagram commutes:
		\[
		\begin{tikzcd}
			u^* f_! \ar{r}{\alpha_f} & u^* f_* \ar{d}{\mathrm{BC}} \\
			g_! v^* \ar{u}{\mathrm{BC}}[sloped,swap]{\sim} \ar{r}{\alpha_g} & g_*v^*\rlap.
		\end{tikzcd}
		\]
		\item For any $A\in\MS_X$ and $B\in\MS_S$, the following diagram commutes:
		\[
		\begin{tikzcd}
			f_!(A)\otimes B \ar{r}{\alpha_f} & f_*(A)\otimes B \ar{d}{\mathrm{PF}} \\
			f_!(A\otimes f^*(B)) \ar{u}{\mathrm{PF}}[sloped,swap]{\sim} \ar{r}{\alpha_f} & f_*(A\otimes f^*(B))\rlap.
		\end{tikzcd}
		\]
	\end{enumerate}
\end{lemma}

\begin{proof}
	By adjunction, this follows from the compatibility with base change and the $\MS_S$-linearity of the trace map $\epsilon_f\colon f^*f_!\to\id_{\MS_X}$.
\end{proof}

\begin{proposition}[Base change and projection formula for smooth projective morphisms]
	\label{prop:smooth-proper-BC}
	Let $f\colon X\to S$ be smooth and locally projective.
	\begin{enumerate}
		\item For any cartesian square
		\[
		\begin{tikzcd}
			Y \ar{r}{v} \ar{d}[swap]{g} & X \ar{d}{f} \\
			T \ar{r}{u} & S\rlap,
		\end{tikzcd}
		\]
		the base change transformation
		\[
		u^*f_*\to g_* v^*
		\]
		is an isomorphism.
		\item For every $A\in \MS_X$ and $B\in \MS_S$, the canonical map
		\[
		f_*(A)\otimes B\to f_*(A\otimes f^*(B))
		\]
		is an isomorphism.
	\end{enumerate}
\end{proposition}

\begin{proof}
	Combine Lemma~\ref{lem:smooth-proper-BC} and Theorem~\ref{thm:ambidexterity}.
\end{proof}

\begin{corollary}\label{cor:dualizable-pf}
	Let $f\colon X\to S$ be smooth and locally projective and let $A\in \MS_X$ be dualizable. Then $f_\sharp(A)\in \MS_S$ is dualizable with dual $f_*(A^\vee)$.
\end{corollary}

\begin{proof}
	Let $f\colon X\to S$ be the structure map. We have $\Hom(f_\sharp A,\ph)\simeq f_*(A^\vee\otimes f^*(\ph))$ by the smooth projection formula. On the other hand, we have $f_*(A^\vee\otimes f^*(\ph))\simeq f_*(A^\vee)\otimes (\ph)$ by Proposition~\ref{prop:smooth-proper-BC}(ii). This shows that $f_\sharp A$ is dualizable with dual $f_*(A^\vee)$.
\end{proof}

\begin{corollary}[Atiyah duality]
	\label{cor:atiyah}
	Let $S$ be a derived scheme, let $X$ be smooth and locally projective over $S$, and let $\xi\in\K(X)$.
	Then the Thom spectrum $\Th_X(\xi)$ is dualizable in $\MS_S$. Moreover, the pairing
	\[
	\ev_{X,\xi}\colon \Th_X(\xi)\otimes \Th_X(-\xi-\Omega_f) \to \1_S
	\]
	of Construction~\ref{ctr:ev} exhibits $\Th_X(-\xi-\Omega_f)$ as the dual of $\Th_X(\xi)$. 
\end{corollary}

\begin{proof}
	Dualizability is a special case of Corollary~\ref{cor:dualizable-pf}. Since $\alpha_f\colon f_!\Sigma^{-\xi}\1_X\to f_*\Sigma^{-\xi}\1_X$ is an isomorphism (Theorem~\ref{thm:ambidexterity}), it follows from Lemma~\ref{lem:duality-pairing}(i) that the comparison map $\mathrm{comp}_{X,\xi}\colon \Th_X(-\xi-\Omega_f)\to \Th_X(\xi)^\vee$ induced by the pairing $\ev_{X,\xi}$ is an isomorphism, so that the latter is a duality pairing.
\end{proof}

\begin{remark}
	If $i\colon Z\into X$ is a closed embedding between smooth and locally projective $S$-schemes, then the dual of $\Sigma^\infty_{\P^1} i_+$ can be identified under Atiyah duality with the $(-\Omega_X)$-twisted Gysin map
	\[
	\Th_X(-\Omega_X) \to \Th_Z(-\Omega_X+\scr N_{i}) \simeq \Th_Z(-\Omega_Z).
	\]
	In particular, since the formation of duals is functorial, we indirectly obtain a functoriality of Gysin maps with respect to composition of closed immersions. 
\end{remark}

\part{Applications}

\section{$\A^1$-colocalization and logarithmic cohomology theories}
\label{sec:logarithmic}

Let $\MS_S^{\A^1}$ denote the full subcategory of $\MS_S$ spanned by the $\A^1$-invariant motivic spectra, which coincides with Morel and Voevodsky's motivic stable homotopy category \cite{VoevodskyICM}. 
We will write \[\rm L_{\A^1}\colon \MS_S\to\MS_S^{\A^1}\] for the left adjoint to the inclusion, which is a symmetric monoidal functor, and $\1_{\A^1}=\rm L_{\A^1}\1$ for the unit of $\MS_S^{\A^1}$. 
The purpose of this section is to study \textit{$\A^1$-colocalization} of $\1_{\A^1}$-modules in $\MS_S$ (Definition~\ref{def:coloc}) and relate it to logarithmic cohomology theories.

\begin{proposition}\label{prop:A1-colocalization}
	Let $S$ be a derived scheme. The inclusion $\MS^{\A^1}_S\subset\MS_S$ admits a right adjoint.
\end{proposition}

\begin{proof}
	If $S$ is qcqs, then Nisnevich sheaves of spectra on $\Sm_S^\fp$ are closed under colimits in $\scr P(\Sm_S^\fp,\Sp)$, since Nisnevich descent is equivalent to Nisnevich excision. Hence, both $\MS_S$ and $\MS^{\A^1}_S$ are closed under colimits in the $\infty$-category of lax symmetric $\P^1$-spectra in $\scr P(\Sm_S^\fp,\Sp)$. This shows that a right adjoint exists when $S$ is qcqs. The existence of a right adjoint in general follows as both $\MS^{\A^1}$ and $\MS$ are Zariski sheaves.
\end{proof}

\begin{corollary}\label{cor:A1-colocalization}
The inclusion $\MS^{\A^1}_S\subset\Mod_{\1_{\A^1}}(\MS_S)$ admits a right adjoint.
\end{corollary}

\begin{proof}
Since the forgetful functor $\Mod_{\1_{\A^1}}(\MS_S)\to\MS_S$ preserves and reflects colimits, the existence of a right adjoint follows from Proposition~\ref{prop:A1-colocalization}.
\end{proof}

\begin{definition}\label{def:coloc}
We write
\[
	(\ph)^\dagger \colon \Mod_{\1_{\A^1}}(\MS_S) \to \MS_S^{\A^1}
\]
for the right adjoint to the canonical inclusion and call it \textit{$\A^1$-colocalization}.
\end{definition}

We begin by listing primary examples of $\1_{\A^1}$-modules.

\begin{example}[Localizing invariants]
	\label{ex:locinv}
Let $E$ be a spectrum-valued localizing invariant of $\Z$-linear stable $\infty$-categories (which we do not require to be finitary, as in \cite[Definition 1.2]{LT}). Let $S$ be a qcqs derived scheme and let $E_S$ denote the motivic spectrum over $S$ representing $E$, as defined in \cite[Remark 5.2.4]{AnnalaIwasa2}.
Note that $E_S$ is canonically a module over $\K_S=\KGL$ in $\MS_S$ by the universality of $\K$-theory as a localizing invariant proved in \cite{blumberg2013universal}. Then:
\begin{enumerate}
\item If $S$ is regular noetherian, then $E_S$ is an $\1_{\A^1}$-module.
\item If $E$ is truncating, then $E_S$ is an $\1_{\A^1}$-module.
\end{enumerate}
Indeed, (i) follows from the fact that $\K$-theory is $\A^1$-invariant for regular noetherian schemes \cite[Proposition~6.8]{TT}, and (ii) follows from the fact that every truncating localizing invariant satisfies cdh-descent on classical qcqs schemes \cite[Theorem~A.2]{LT} and that the cdh-sheafification of $\K$-theory on classical qcqs schemes is homotopy invariant $\K$-theory \cite[Theorem~6.3]{KST}\footnote{There is no need for the finite-dimensional noetherian assumption in \emph{loc.\ cit.} since $\K$-theory is finitary.}.
\end{example}

Next we would like to discuss examples of $\1_{\A^1}$-modules arising from ``motivic filtrations'' of localizing invariants.
We first give a general recipe for constructing motivic spectra out of filtered spectra.

\begin{construction}[Motivic filtrations and Bott elements]
	\label{cst:motivic-filtration}
	Let $S$ be a derived scheme and set $\CC=\scr{P}_\Nis(\Sm_S,\Sp)$ for notational simplicity. Write 
	\[
	\Fil(\scr C)=\Fun((\Z,\leq)^\op,\scr C) \quad\text{and}\quad \SSeq(\scr C)=\Fun(\Fin^\simeq,\scr C).
	\]
	The symmetric monoidal functors $\Fin^\simeq\to\N\to (\Z,\leq)^\op$ induce lax symmetric monoidal functors
	\[
	\Fil(\scr C)\to \scr C^\N \to \SSeq(\scr C),
	\]
	associating to any filtered object a symmetric sequence.
	Let $\P^1$ denote the object $\Sigma^\infty(\P^1,\infty)\in\scr C$ and let $F^*\P^1$ be the filtered object with $F^{> 1}\P^1=0$ and $F^{\leq 1}\P^1=\P^1$. The associated symmetric sequence is $(\P^1,\P^1,0,0,\dotsc)$.
	Note that:
	\begin{itemize}
		\item A commutative $\Sym(F^*\P^1)$-algebra in $\Fil(\scr C)$ is a commutative algebra $F^*A$ in $\Fil(\scr C)$ with an element $c\colon \Sigma^\infty\P^1\to F^1A$.
		\item A commutative $\Sym(\P^1,\P^1,0,\dotsc)$-algebra in $\SSeq(\scr C)$ is a commutative algebra $A$ in $\Sp_{\P^1}^\lax(\scr C)$ with an element $\beta\colon \Sigma^\infty\P^1\to A(\emptyset)$.
	\end{itemize}
	Let $F^*E\in\CAlg(\Fil(\scr C))$. Given a map $c\colon \Sigma^\infty\Pic\to F^1E$, we say that $F^*E$ \emph{satisfies the projective bundle formula} (with respect to $c$) if for every $n\in\Z$, $r\geq 1$, and $X\in\Sm_S$, the map
	\[
		\sum_{i=0}^r c(\scr O(1))^i \colon \bigoplus_{i=0}^r F^{n-i}E(X) \to F^{n}E(\P^r_X)
	\]
	is an isomorphism. The case $r=1$ implies that the lax $\P^1$-spectrum $e$ defined by $(F^*E,c|_{\P^1})$ is strict, so that $(e,c)$ is an oriented object in $\CAlg(\Sp_{\P^1}(\scr C))$ satisfying the projective bundle formula \cite[Definition~3.2.1]{AnnalaIwasa2}. It then follows from \cite[Lemma 3.3.5]{AnnalaIwasa2} that $e$ also satisfies elementary blowup excision, so that $e\in\CAlg(\MS_S)$. 
	Similarly, the lax symmetric monoidal functor $\Fil(\scr C)\to\SSeq(\scr C)$ sends any $F^*E$-module satisfying the projective bundle formula (with respect to $c$) to an $e$-module in $\MS_S$. In particular, this functor sends the cofiber sequence of $F^*E$-modules
	\[
	F^{*+1}E\to F^*E \to \gr_F^*E
	\]
	to the cofiber sequence of $e$-modules
	\[
	\Sigma_{\P^1}e \xrightarrow{\beta} e \to e/\beta,
	\]
	and it sends the span of filtered $\E_\infty$-rings
	\[
	F^{-\infty}E \leftarrow F^*E \to \gr_F^* E,
	\]
	where $F^{-\infty}E=\colim_{n\to-\infty} F^nE$, to a span of motivic $\E_\infty$-ring spectra
	\[
	e[\beta^{-1}] \leftarrow e \to e/\beta.
	\]
	Note that $e$ \emph{represents} $F^*E$ in the sense that $\Omega_{\P^1}^{\infty-n}e\simeq F^nE$ for all $n\in \Z$, and that the maps $F^{n+1}E\to F^nE$ correspond to multiplication by the Bott element $\beta\in e^{-1}(S)$. The motivic $\E_\infty$-ring spectra $e[\beta^{-1}]$ and $e/\beta$ represent $F^{-\infty}E$ and $\gr_F^* E$ in the same sense.
\end{construction}

\begin{remark}\label{rmk:filtered-E-infinity}
	The motivic spectrum $e\in\CAlg(\MS_S)$ alone does not contain the data of the $\Z$-graded $\E_\infty$-ring structure on $(\Omega_{\P^1}^{\infty-n}e)_{n\in \Z}\simeq (F^nE)_{n\in \Z}$; it only remembers the $\mathbb S$-graded $\E_\infty$-ring (hence the $\Z$-graded $\E_1$-ring).
	 However, we can apply the construction to the filtered $\E_\infty$-$F^*E$-algebra $F^{*+\star}E$ in $\Fil(\scr C)$, which yields an $\E_\infty$-$e$-algebra structure on the filtered motivic spectrum $(\Sigma_{\P^1}^\star e,\beta)$. 
	 In particular, we have a $\Z$-graded $\E_\infty$-$e$-algebra structure on $(\Sigma_{\P^1}^ne)_{n\in\Z}$ and, as a special case, a $\Z$-graded $\E_\infty$-$e/\beta$-algebra structure on $(\Sigma_{\P^1}^ne/\beta)_{n\in\Z}$.
\end{remark}

We will apply Construction~\ref{cst:motivic-filtration} to the motivic filtration of $p$-complete topological cyclic homology $F^*\TC_p$ defined by Bhatt, Morrow, and Scholze~\cite[Theorem 1.12]{BMS} (see \cite[Construction 5.33]{AMMN} for the extension to general $p$-complete animated rings).
We will regard $F^*\TC_p$ as an étale sheaf of $p$-complete spectra on all derived schemes by precomposing with the $p$-completion functor.
We also consider the motivic filtrations of $\TP_p$, $\TC_p^-$, and $\THH_p$, as defined in general in \cite[Section 6.2]{BhattLurie}.
To formulate the projective bundle formula in these examples, recall that $\TC_p/F^1\TC_p\simeq \Z_p$, so that there is a unique lift
\[
\begin{tikzcd}
	\Sigma^\infty\Pic \ar{r}{1-\scr L^\vee} \ar[dashed]{drr}[swap]{c} &  \K \ar{r}{\mathrm{Tr}} & \TC_p \\
	& & F^1\TC_p\rlap, \ar{u}
\end{tikzcd}
\]
where $\Tr$ is the cyclotomic trace.

\begin{proposition}\label{prop:filtered-PBF-TC}
Let $S$ be a derived scheme. 
Then $F^*\TC_p$ satisfies the projective bundle formula over $S$, i.e., the map
\[
	\sum_{i=0}^r c(\scr O(1))^{i} \colon \bigoplus_{i=0}^r F^{*-i}\TC_p(S) \to F^*\TC_p(\P^r_S)
\]
is an isomorphism for every $r\ge 0$. 
Moreover, the $\E_\infty$-$F^*\TC_p$-algebras $F^*\TP_p$, $F^*\TC_p^-$, $F^*\THH_p$ also satisfy the projective bundle formula over $S$.
\end{proposition}

\begin{proof}
As explained in \cite[Construction~5.33]{AMMN}, the filtration $F^*\TC_p$ is complete on all animated rings (in fact, $F^n\TC_p$ is $(n-1)$-connective).
By Zariski descent, it follows that the filtration $F^*\TC_p$ is complete on all derived schemes.
Hence, it suffices to show that the associated graded algebra $\gr^*_F\TC_p$ satisfies the projective bundle formula. 
The same reduction applies to $F^*\TP_p$, $F^*\TC_p^-$, and $F^*\THH_p$, as these filtrations are complete by \cite[Corollary 6.2.15]{BhattLurie}.
The associated graded algebras are respectively syntomic cohomology, Nygaard-completed prismatic cohomology, the Nygaard filtration thereon, and its associated graded. The desired projective bundle formulas are then proved in \cite[Lemma~9.1.4]{BhattLurie} with respect to the syntomic first Chern class $c_1^\syn$ defined in \emph{op.\ cit.} To conclude, we need to compare $c_1^\syn$ with the map $c\colon \Sigma^\infty\Pic\to \gr^1_F\TC_p=\Z_p^\syn(1)[2]$ induced by the cyclotomic trace.
More precisely, to deduce that the projective bundle formulas also hold with respect to $c$, it will suffice to show that $c_1^\syn$ and $c$ differ by some automorphism of $\gr^1_F\TC_p$. By \cite[Theorem~7.5.6]{BhattLurie}, the first Chern class $c_1^\syn$ exhibits $\gr^1_F\TC_p$ as the $p$-completion of $\Pic$.
On the other hand, the map $c$ factors as
\[
\begin{tikzcd}
\Sigma^\infty\Pic \ar{r}{1-\scr L^\vee} \ar{dr}[swap]{\mathrm{can}} &  \K_{\rk=0} \ar{r}{\mathrm{Tr}} \ar{d}{\det} & F^1\TC_p \ar{d} \\
 & \Pic \ar[dashed]{r} & \gr^1_F\TC_p\rlap,
\end{tikzcd}
\]
and the lower horizontal map induces an isomorphism after $p$-completion by \cite[Proposition~7.17]{BMS} (this proposition establishes the claim on quasiregular semiperfectoid rings, but it then automatically holds on $p$-quasisyntomic rings by $p$-quasisyntomic descent, and hence on all animated rings as both sides are left Kan extended from smooth $\Z$-algebras).
\end{proof}

Using Construction~\ref{cst:motivic-filtration}, we therefore obtain a motivic $\E_\infty$-ring spectrum $\tc_p$ over any derived scheme $S$ representing the filtration $F^*\TC_p$, together with a diagram of motivic $\E_\infty$-ring spectra over $S$
\[
\begin{tikzcd}
	& \tc_p \ar[rd, "(\ph)/\beta"] \ar[ld, "{(\ph)[\beta^{-1}]}" swap] & \\
	\TC_p & & \HH\Z^\syn_p \rlap.
\end{tikzcd}
\]
Here, $\HH\Z^\syn_p$ is the motivic $\E_\infty$-ring spectrum representing the syntomic cohomology of $p$-adic formal schemes.
In the same way, the motivic filtrations on $\TP_p$, $\TC_p^-$, and $\THH_p$ give morphisms of motivic $\E_\infty$-$\tc_p$-algebras in $\MS_S$
\[
\tp_p\leftarrow \tc_p^- \to \thh_p.
\]

\begin{remark}
	The motivic filtrations on $\TC_p$ and $\THH_p$ are nonnegative and hence trivially exhaustive, so that $\tc_p[\beta^{-1}]$ and $\thh_p[\beta^{-1}]$ represent $\TC_p$ and $\THH_p$, respectively. On the other hand, the filtrations on $\TC_p^-$ and $\TP_p$ are exhaustive on $\QSyn$ \cite[Theorem 1.12]{BMS} but not in general, so that $\tc_p^-[\beta^{-1}]$ and $\tp_p[\beta^{-1}]$ do not always represent $\TC_p^-$ and $\TP_p$.
\end{remark}

\begin{remark}
	It is expected that Example~\ref{ex:locinv} can be refined to motivic filtrations of localizing invariants; in particular, $\tc_p$ should be a $\1_{\A^1}$-module over any regular noetherian scheme. 
	We will prove this conjecture over Dedekind domains (Corollary~\ref{cor:tc}).
\end{remark}

Let us consider the motivic filtration of algebraic $\K$-theory $F^*\K$ as a Nisnevich sheaf of spectra on smooth schemes over Dedekind domains, obtained by the $\A^1$-local slice filtration of the motivic $\E_\infty$-ring spectrum $\KGL$:
\[
F^n\K = \Omega^\infty_{\P^1} f_n\KGL \simeq \Omega^\infty_{\P^1} \Sigma_{\P^1}^n \kgl.
\]
In particular, $F^*\K$ is canonically $\E_\infty$-multiplicative by \cite[Section 13.4]{norms}.

For a commutative ring $A$, let $\QSyn_A$ denote the category of $A$-algebras that are $p$-quasisyntomic in the absolute sense \cite[Definition C.6]{BhattLurie}.
Recall also the $p$-quasisyntomic topology on $\QSyn_A$ from \cite[Definition C.9]{BhattLurie}.

\begin{lemma}\label{lem:motfil}
Let $D$ be a Dedekind domain and let $n\in\Z$.
Then the $p$-completed left Kan extension of $F^n\K$ from smooth $D$-algebras to $\QSyn_D$ is $p$-quasisyntomic-locally $2n$-connective.
\end{lemma}

\begin{proof}
It follows from \cite[Example 1.3]{Bac22} and \cite[Theorem 7.18 and Corollary 5.22]{SpitzweckHZ} that the $n$th graded piece $\gr_F^n\K$ is isomorphic to the weight $n$ motivic complex $\Z(n)[2n]$, which is by definition the Zariski sheafification of Bloch's cycle complex (and is zero for $n<0$).
In \cite{Gei}, Geisser identified the étale sheafification of the $p$-completed motivic complex $\Z_p(n)$ on smooth $D$-schemes with Sato's $p$-adic étale Tate twist $\mathfrak T_p(n)$ \cite[(1.3.3)]{SatoTwists}. More precisely, \cite[Corollary 4.4 and Theorem 1.3]{Gei} provide a comparison map $\Z_p(n)\to\mathfrak T_p(n)$, and \cite[Theorem 1.2(4) and Theorem 1.3]{Gei} imply that it exhibits $\mathfrak T_p(n)$ as the étale sheafification of $\Z_p(n)$.
By \cite[Theorem~5.8]{BM}, $\mathfrak T_p(n)$ is in turn isomorphic to the syntomic complex $\mathrm{R}\Gamma_\syn(\ph,\Z_p(n))$ of Bhatt and Lurie \cite[Section~8.4]{BhattLurie}.
The latter presheaf on $\QSyn_D$ is $p$-completely left Kan extended from smooth $D$-algebras by \cite[Proposition~8.4.10]{BhattLurie} and it is $p$-quasisyntomic-locally connective by \cite[Theorem~14.1]{BS:2022}.
Therefore, the $p$-completed left Kan extension of $\gr_F^n\K$ from smooth $D$-algebras to $\QSyn_D$ is $p$-quasisyntomic-locally $2n$-connective.

On the other hand, we claim that $F^n\K$ is Zariski-locally $n$-connective on $\Sm_D$. By \cite[Theorem 1.1(1)]{Bac22}, the $\A^1$-invariant motivic spectrum $f_n\KGL$ is very $n$-effective and hence $n$-connective in the homotopy $t$-structure. Since the latter is complete \cite[Corollary 3.8]{SchmidtStrunk}, the motivic filtration $F^*\K$ is also complete. By \cite[Corollary 4.4]{Gei}, $\gr_F^n\K\simeq \Z(n)[2n]$ is $n$-connective. The Postnikov completeness of the Zariski $\infty$-topos now implies that $F^n\K=\lim_i F^n\K/F^{n+i}\K$ is $n$-connective.
Hence, the left Kan extension of $F^n\K$ from smooth $D$-algebras to $\QSyn_D$ is Zariski-locally, and thus $p$-quasisyntomic-locally, $n$-connective.
Now we can formally conclude that its $p$-completion is $p$-quasisyntomic-locally $2n$-connective by Lemma~\ref{lem:2n-connective} below.
\end{proof}

\begin{lemma}\label{lem:2n-connective}
Let $\CC$ be a stable $\infty$-category equipped with a $t$-structure.
Suppose that there is a fiber sequence
\[
	f^{n+1} \to f^n \to z^n
\]
for each $n\ge 0$.
Assume that $z^n$ is $2n$-connective and that $f^n$ is $\phi(n)$-connective for some function $\phi$ with $\sup_{n\geq 0}\phi(n)=\infty$.
Then $f^n$ is $2n$-connective.
\end{lemma}

\begin{proof}
The goal is to show that $(f^n)_{<2n}=0$.
Fix $n\ge 0$ and let $m\ge n$.
Since the truncation $(\ph)_{<2n}\colon\CC\to\CC_{<2n}$ is a left adjoint, the sequence
\[
	(f^{m+1})_{<2n} \to (f^m)_{<2n} \to (z^m)_{<2n}=0
\]
is a cofiber sequence in $\CC_{<2n}$.
By induction on $m$, it suffices to find $m\ge n$ such that $(f^m)_{<2n}=0$, but this holds if $\phi(m)\ge 2n$.
\end{proof}

The following result originally came out of discussions between Dustin Clausen, Akhil Mathew, and the third author, and our proof here follows the same line of thought.

\begin{proposition}[Filtered cyclotomic trace]\label{prop:motfil}
Let $D$ be a Dedekind domain.
Then the cyclotomic trace $\K\to\TC_p$ on smooth $D$-schemes lifts uniquely to a morphism of $\E_\infty$-multiplicative filtrations $F^*\K\to F^*\TC_p$.
\end{proposition}

\begin{proof}
Recall that $\TC_p$ on $\QSyn_D$ is $p$-completely left Kan extended from smooth $D$-algebras and that the motivic filtration of $\TC_p$ on $\QSyn_D$ is the double-speed $p$-quasisyntomic-local Postnikov filtration by definition and by \cite[Theorem~14.1]{BS:2022}.
Hence, we obtain the desired lift by Lemma~\ref{lem:motfil}.
\end{proof}

We say that a morphism of derived schemes $Y\to X$ is \emph{$p$-completely smooth} if the induced morphism $Y/p\to X/p$ is smooth.

\begin{corollary}\label{cor:motfil}
	Let $S$ be $p$-completely smooth over a Dedekind domain. Then the cyclotomic trace on $\Sm_S$ is refined by a morphism of motivic $\E_\infty$-ring spectra $\kgl\to\tc_p$ in $\MS_S$.
\end{corollary}

\begin{proof}
	The construction will be natural in $S$, so that we can assume $S$ affine.
	If $u\colon S^\wedge_p\to S$ is the $p$-completion of $S$, then $\tc_p\simeq u_*(\tc_p)$ by definition. We may therefore assume that $S$ is $p$-complete.
	Let $f\colon S\to \Spec(D)$ be a $p$-completely smooth morphism, where $D$ is a Dedekind domain.
	Then $S$ is the $p$-completion of a smooth affine $D$-scheme $S'$ \cite[Remark 5.13]{AMM}.
	The morphism $S\to S'$ is regular by \cite[Tags 07PX and 0AH2]{stacks}, so that $f$ is pro-smooth by Popescu's theorem \cite[Tag 07GC]{stacks}. Since the slice filtration and the inclusion $\MS^{\smash[t]{\A^1}}\into\MS$ are compatible with pro-smooth base change, we have $\kgl_S\simeq f^*(\kgl_D)$ in $\MS_S$. Applying Construction~\ref{cst:motivic-filtration} to the filtered $\E_\infty$-map of Proposition~\ref{prop:motfil}, we obtain a morphism of motivic $\E_\infty$-ring spectra $\kgl_D\to\tc_p$ in $\MS_D$.
	The desired morphism in $\MS_S$ is then the composite
	\[
	\kgl_S\simeq f^*(\kgl_D) \to f^*(\tc_p) \to \tc_p.\qedhere
	\]
\end{proof}

\begin{remark}
There is an integral refinement of the motivic filtration of topological cyclic homology due to Bouis~\cite{Bou}.
This gives an integral motivic $\E_\infty$-ring spectrum $\tc$ such that the cyclotomic trace induces a morphism of $\E_\infty$-ring spectra $\kgl\to\tc$ over Dedekind domains.
\end{remark}

\begin{corollary}\label{cor:tc}
Let $S$ be $p$-completely smooth over a Dedekind domain. Then the motivic spectrum $\tc_p$ is an $\E_\infty$-$\1_{\A^1}$-algebra in $\MS_S$.
\end{corollary}

\begin{proof}
	This follows from Corollary~\ref{cor:motfil}, as $\kgl$ is $\A^1$-invariant by definition.
\end{proof}

\begin{example}[Syntomic cohomology and étale motivic cohomology]
	\label{ex:syn}
By Corollary~\ref{cor:tc}, the syntomic cohomology spectrum $\HH\Z_p^\syn=\tc_p/\beta$ is an $\E_\infty$-$\1_{\A^1}$-algebra in $\MS_S$ if $S$ is $p$-completely smooth over a Dedekind domain.
Note that this spectrum represents the syntomic cohomology of $p$-adic formal schemes, not the syntomic cohomology of schemes.
To clarify, let $\HH\Z_p^\et$ denote the motivic $\E_\infty$-ring spectrum over a derived scheme $S$ representing the syntomic cohomology of smooth $S$-schemes (obtained by applying Construction~\ref{cst:motivic-filtration} to the $\Z$-graded $\E_\infty$-algebra $\mathrm{R}\Gamma_\syn(\ph,\Z_p(*))[2*]$ from \cite[Section 8.4]{BhattLurie}). As we recalled in the proof of Lemma~\ref{lem:motfil}, $\HH\Z_p^\et$ is stable under arbitrary base change as a $p$-complete motivic spectrum, and over Dedekind domains it is the (degreewise) étale sheafification of the $p$-adic motivic cohomology spectrum $\HH\Z_p$, which explains the notation.
In particular, $\HH\Z_p^\et$ is an $\E_\infty$-$\HH\Z_p$-algebra and hence an $\E_\infty$-$\1_{\A^1}$-algebra over Dedekind domains.
In general, the $p$-complete motivic spectrum $\HH\Z_p^\syn$ coincides with $i_*i^*\HH\Z_p^\et$, where $i$ denotes the inclusion of the $p$-adic formal scheme $S^\wedge_p$ into $S$ and $\MS_{S^\wedge_p}=\lim_n \MS_{S/p^n}$.
Later we will also define an integral étale motivic cohomology spectrum $\HH\Z^\et$ over any base, which is an $\E_\infty$-$\1_{\A^1}$-algebra over Dedekind domains (see Example~\ref{ex:HZetale}).
\end{example}

\begin{example}[Prismatic and crystalline cohomology]
	\label{ex:prism}
Syntomic cohomology is practically initial among other important $p$-adic étale cohomology theories, which are thus also $\1_{\A^1}$-modules over Dedekind domains.
Let us take prismatic cohomology as an example \cite{BS:2022,BhattLurie,AKN}.
Let $\HH\Z_p^\prism$ denote the motivic $\E_\infty$-ring spectrum representing absolute prismatic cohomology.
It is defined over any derived scheme and is an $\E_\infty$-algebra over the syntomic cohomology spectrum $\HH\Z_p^\syn$.\footnote{This $\E_\infty$-algebra structure exists before the Nygaard completion thanks to \cite[Proposition~7.12]{AKN}.}
In particular, $\HH\Z_p^\prism$ is an $\E_\infty$-$\1_{\A^1}$-algebra in $\MS_S$ if $S$ is $p$-completely smooth over a Dedekind domain.
If $(A,I)$ is a prism, relative prismatic cohomology $\mathrm{R}\Gamma_\prism(\ph/A)$ is represented by a motivic spectrum $\HH A^\prism$ over any $A/I$-scheme $S$.
It is an $\E_\infty$-algebra over $\HH\Z_p^\prism$, and thus over $\1_{\A^1}$ if $S$ is $p$-completely smooth over a Dedekind domain.
Taking the crystalline prism $(W(k),(p))$ for a perfect field $k$ of characteristic $p$, we see that the motivic spectrum $\HH W(k)^\crys=\HH W(k)^\prism$ representing crystalline cohomology is an $\E_\infty$-$\1_{\A^1}$-algebra over $k$.
\end{example}

\begin{example}[Rational orientable ring spectra]
\label{ex:Q-orientable}
Let $S$ be a regular noetherian scheme. As we will see later in Proposition~\ref{prop:HQ-modules}, every $\Q$-linear orientable object in $\CAlg(\h\MS_S)$ is canonically a $\1_{\A^1}$-module, and every $\Q$-linear orientable $\E_\infty$-algebra in $\MS_S$ is canonically an $\E_\infty$-$\1_{\A^1}$-algebra.
\end{example}

With these numerous examples of $\1_{\A^1}$-modules in mind, we now investigate the $\A^1$-colocalization of $\1_{\A^1}$-modules.

\begin{lemma}\label{lem:dual0}
Let $\CC$ be a symmetric monoidal $\infty$-category in which every object is dualizable.
Let $F\colon\DD\to\DD'$ be a $\CC$-linear functor between $\infty$-categories tensored over $\CC$.
Then a left or right adjoint of $F$ is $\CC$-linear if it exists.
\end{lemma}

\begin{proof}
Suppose that $F$ admits a right adjoint $G$.
Then $G$ is lax $\CC$-linear and thus we have a natural transformation $x\otimes G(\ph)\to G(x\otimes\ph)$ for every $x\in\CC$.
It suffices to show that the induced map
\[
	\Map(y,x\otimes G(z)) \to \Map(y,G(x\otimes z))
\]
is an isomorphism for every $y\in\DD$ and $z\in\DD'$, but this follows straightforwardly from the dualizability of $x\in\CC$ and the $\CC$-linearity of $F$.
The case of a left adjoint is proved in the same way.
\end{proof}

\begin{lemma}\label{lem:dual}
The adjoint functors
\[
\begin{tikzcd}
	\Mod_{\1_{\A^1}}(\MS_S) \ar[r, shift left=2, "\L_{\A^1}"] \ar[r, shift right=2, "(\ph)^\dagger" swap] &  \MS_S^{\A^1} \ar[l]
\end{tikzcd}
\]
are all $\MS_S^\dual$-linear, where $\MS_S^\dual$ is the full subcategory of $\MS_S$ spanned by the dualizable objects.
\end{lemma}

\begin{proof}
Apply Lemma~\ref{lem:dual0} to $\CC=\MS_S^\dual$ and note that the $\A^1$-localization $\L_{\A^1}$ is $\MS_S$-linear.
\end{proof}

\begin{proposition}\label{prop:coloc}
Let $E$ be a $\1_{\A^1}$-module in $\MS_S$ and let $A\in\MS_S^\dual$ be a dualizable motivic spectrum.
Then the counit $E^\dagger\to E$ induces an isomorphism of spectra
\[
	E^\dagger(A) \simeq E(A).
\]
This holds in particular if $A=\Th_X(\xi)$ for some smooth projective $S$-scheme $X$ and some $\xi\in\K(X)$.
\end{proposition}

\begin{proof}
It follows from the $\MS_S^\dual$-linearity of the inclusion $\MS_S^{\A^1}\into \Mod_{\1_{\A^1}}(\MS_S)$ (Lemma~\ref{lem:dual}) that
\[
	A\otimes\1_{\A^1} \simeq \L_{\A^1}A \rlap.
\]
Then we see that
\[
\begin{split}
	E^\dagger(A)  = \map(A,E^\dagger)
	&\simeq \map(\L_{\A^1}A,E^\dagger) \\
	&\simeq \map_{\1_{\A^1}}(\L_{\A^1}A,E) \\
	&\simeq \map_{\1_{\A^1}}(A\otimes\1_{\A^1},E) \\
	&\simeq \map(A,E) = E(A) \rlap,
\end{split}
\]
where $\map_{\1_{\A^1}}(\ph,\ph)$ denotes the spectrum of $\1_{\A^1}$-linear maps.
Finally, $\Th_X(\xi)$ is dualizable in $\MS_S$ by Corollary~\ref{cor:atiyah}.
\end{proof}

Informally, Proposition~\ref{prop:coloc} says that $\A^1$-colocalization is a universal machinery that converts a $\1_{\A^1}$-module into some $\A^1$-invariant cohomology theory without changing its values on smooth projective schemes.
The following is how it is calculated in practice; the conclusion is Proposition~\ref{prop:log}.

\begin{construction}\label{cst:sncd}
Let $\Box^n$ denote the poset of subsets of $\{1,\dotsc,n\}$.
Let $X$ be a smooth projective $S$-scheme and $\partial X$ a relative strict normal crossings divisor on $X$ with smooth components $\partial_1X,\dotsc,\partial_nX$. Consider the $n$-cube $\Box^{n,\op}\to\Sm_S$ sending $I$ to the intersection $\partial_IX$ of all $\partial_iX$ with $i\in I$.
This induces an $n$-cube
\[
	\Box^{n,\op}\to \Fun(\MS_X,\MS_S),\quad I \mapsto f_{I\sharp}i_I^*,
\]
where $i_I$ denotes the closed immersion $\partial_IX\into X$ and $f_I$ the structure map $\partial_IX\to S$.
Then, for each $\xi\in\K(X)$, we define the $n$-cube $\Th_{(X,\partial X)}(\xi)$ in $\MS_S$ by
\[
	\Th_{(X,\partial X)}(\xi)\colon \Box^{n}\to\MS_S,\quad I\mapsto (f_{I\sharp}i_I^*\Sigma^{-\xi-\Omega_X}\1_X)^\vee.
\]
\end{construction}

\begin{remark}
Under Atiyah duality, the cube $\Th_{(X,\partial X)}(\xi)$ sends $\emptyset\to I$ to the $\xi$-twisted Gysin map
\[
	\Th_X(\xi) \to \Th_{\partial_IX}(\xi+\mathcal{N}_{i_I}).
\]
The cube itself should be given by ``coherent Gysin maps'', but we will not need such a description.
\end{remark}

\begin{lemma}\label{lem:purity}
Let $X$ be a smooth projective $S$-scheme, $\partial X$ a relative strict normal crossings divisor on $X$, and $\xi\in\K(X)$.
Then there is a canonical morphism 
\[
	\Th_{X\setminus\partial X}(\xi) \to \Th_{(X,\partial X)}(\xi)
\]
in $\MS_S$ that exhibits $\L_{\A^1}\Th_{X\setminus\partial X}(\xi)$ as the total fiber of the cube $\L_{\A^1}\Th_{(X,\partial X)}(\xi)$.
\end{lemma}

\begin{proof}
By definition, the total fiber of $\Th_{(X,\partial X)}(\xi)$ is the fiber of
\[
	\Th_X(-\xi-\Omega_X)^\vee \to \Bigl( \colim_{I\in\Box^n\setminus\{\varnothing\}}\Th_{\partial_IX}(-\xi-\Omega_X) \Bigr)^\vee \rlap.
\]
On the other hand, by Zariski descent, we have an isomorphism
\[
	\Th_{X\setminus\partial X}(\xi) \simeq \lim_{I\in\Box^n\setminus\{\varnothing\}}\Th_{X\setminus\partial_IX}(\xi) \rlap.
\]
Hence, it suffices to construct a pairing in $\MS_S$ 
\[
	\Th_{X/(X\setminus Z)}(\xi)\otimes \Th_Z(-\xi-\Omega_X) \to \1_S
\]
that is (bivariantly) natural in the smooth closed subscheme $Z\subset X$ and becomes a perfect pairing after $\A^1$-localization.
We define it as the composition
\begin{align*}
	\Th_{X/(X\setminus Z)}(\xi)\otimes \Th_Z(-\xi-\Omega_X)
	& \simeq \Th_{X\times Z/(X\setminus Z)\times Z}(\xi\boxplus(-\xi-\Omega_X)) \\
	&\to \Th_{X\times Z/(X\times Z\setminus\delta(Z))}(\xi\boxplus(-\xi-\Omega_X)) \\
	&\to \Th_{X\times X/(X\times X\setminus\delta(X))}(\xi\boxplus(-\xi-\Omega_X)) \\
	&\to \Th_{X}(-\Omega_X+\scr N_\delta)\simeq X_+ \to \1_S,
\end{align*}
where the first map collapses the complement of $\delta(Z)$, the second map is induced by the closed immersion $X\times Z\into X\times X$, and the third map is the Gysin map with respect to the diagonal $\delta\colon X\into X\times X$ (cf.\ Construction~\ref{ctr:ev}).
This pairing is obviously natural in $Z$, as the first two maps are and the third one does not depend on $Z$.
Using Theorem~\ref{thm:tang}(ii), we can identify this pairing with
\[
\Th_{X/(X\setminus Z)}(\xi)\otimes \Th_Z(-\xi-\Omega_X) \xrightarrow{\gys\otimes\id}
\Th_{Z}(\xi+\scr N_Z)\otimes \Th_Z(-\xi-\Omega_X) \xrightarrow{\ev_{Z,\xi+\scr N_Z}} \1_S.
\]
The second map is a perfect pairing in $\MS_S$ by Atiyah duality (Corollary~\ref{cor:atiyah}), and the first map becomes an isomorphism after $\A^1$-localization, by the Morel--Voevodsky purity theorem \cite[Section~3, Theorem~2.23]{MV}.
\end{proof}

\begin{proposition}\label{prop:log}
Let $E$ be a $\1_{\A^1}$-module in $\MS_S$ and $U$ a smooth $S$-scheme.
Suppose that there is an open immersion $U\into X$ of $S$-schemes such that $X$ is smooth projective and $\partial X=X\setminus U$ is a relative strict normal crossings divisor.
Then for every $\xi\in\K(X)$, there is a canonical isomorphism
\[
	E^\dagger(\Th_U(\xi)) \simeq \operatorname{tcofib} E(\Th_{(X,\partial X)}(\xi)),
\]
where $\operatorname{tcofib}$ denotes the total cofiber. 
\end{proposition}

\begin{proof}
Under Proposition~\ref{prop:coloc}, this is the $\A^1$-equivalence of Lemma~\ref{lem:purity} applied to the $\A^1$-invariant motivic spectrum $E^\dagger$.
\end{proof}

\begin{remark}
The right hand side of the isomorphism in Proposition~\ref{prop:log} can be regarded as the ``logarithmic $E$-cohomology'' of the logarithmic pair $(X,\partial X)$.
Hence, the proposition says that $E^\dagger$ of a smooth $S$-scheme $U$ is calculated as the logarithmic $E$-cohomology of a good compactification, if it exists.
On the other hand, it implies that the logarithmic $E$-cohomology does not depend on the choice of compactification.
\end{remark}

Let us take crystalline cohomology as an example.
Consider the crystalline cohomology spectrum $\HH W(k)^\crys$ over a perfect field $k$, which is an $\E_\infty$-$\1_{\A^1}$-algebra in $\MS_k$ by Example~\ref{ex:prism}.
Then its $\A^1$-colocalization gives an integral refinement of Berthelot's rigid cohomology \cite{Ber}. 
More precisely:

\begin{proposition}[$\A^1$-colocalized crystalline cohomology]\label{prop:crys}
Let $k$ be a perfect field of characteristic $p>0$.
\begin{enumerate}
\item The motivic spectrum $\HH W(k)^{\crys,\dagger}[1/p]$ represents Berthelot's rigid cohomology.
\item For a smooth $k$-scheme $U$, suppose that there is a smooth projective compactification $X$ such that $\partial X=X\setminus U$ is a strict normal crossings divisor.
Then there is a $W(k)$-linear isomorphism
\[
	\HH W(k)^{\crys,\dagger}(U) \simeq \mathrm{R}\Gamma_\crys((X,\partial X)/W(k)) \rlap,
\]
where the right hand side is logarithmic crystalline cohomology in the sense of Kato~\cite{Kat}; in particular, the latter does not depend on the choice of compactification if one exists.
\end{enumerate}
\end{proposition}

\begin{proof}
(i) Let $\HH W(k)^\rig$ denote the motivic spectrum over $k$ representing rigid cohomology.
Then there is a morphism of motivic spectra over $k$
\begin{equation}\label{eq:rig}
	\HH W(k)^\rig \to \HH W(k)^\crys[1/p] \rlap, 
\end{equation}
that induces an isomorphism on $\Sigma^{\infty-n}_{\P^1}X_+$ for every smooth projective $k$-scheme $X$ and $n\in\Z$.
One way to construct such a morphism is to use the overconvergent de Rham--Witt complex \cite{DLZ}: it is a differential graded subalgebra $W^\dagger\Omega_{\ph/k}$ of the de Rham--Witt complex $W\Omega_{\ph/k}$, which is rationally isomorphic to rigid cohomology.
Hence, the inclusion $W^\dagger\Omega_{\ph/k}\to W\Omega_{\ph/k}$ rationally induces \eqref{eq:rig}.

By \cite[Proposition~B.1]{LYZ}, the $\infty$-category $\MS_k^{\A^1}[1/p]$ is generated under colimits and $\P^1$-desus\-pen\-sions by smooth projective $k$-schemes.
Combining this with Proposition~\ref{prop:coloc}, we see that \eqref{eq:rig} induces an isomorphism after $\A^1$-colocalization.
Since rigid cohomology is $\A^1$-invariant, $\A^1$-colocalization does not change $\HH W(k)^\rig$, and thus we obtain
\[
	\HH W(k)^\rig \simeq (\HH W(k)^\crys[1/p])^\dagger \simeq \HH W(k)^{\crys,\dagger}[1/p] \rlap.
\]
The second isomorphism holds as both sides have the same values on smooth projective $k$-schemes; see also Corollary~\ref{cor:invchar} below.

(ii) By Proposition~\ref{prop:log} (and the fact that $\HH W(k)^\crys$ is oriented), the assertion is equivalent to an isomorphism
\begin{equation}\label{eqn:log-crys}
	\mathrm{R}\Gamma_\crys((X,\partial X)/W(k)) \simeq \operatorname{tcofib}_I\HH W(k)^\crys((\Sigma_{\P^1}^{\infty-d}\partial_IX_+)^\vee),
\end{equation}
where $d=\dim(X)$.
Since crystalline cohomology satisfies the Künneth formula \cite[Remark 4.1.8]{BhattLurie}, the functor
\[
\HH W(k)^\crys(\ph)\colon (\MS_k^\omega)^\op \to \Mod_{W(k)}(\Sp)^{\wedge}_{p}
\]
is the unique symmetric monoidal exact extension of $\mathrm{R}\Gamma_\crys(\ph/W(k))\colon \Sm_k^\op\to \Mod_{W(k)}(\Sp)^{\wedge}_{p}$,
and it sends $\Sigma_{\P^1}\1$ to $W(k)[-2]$.
Hence, the right-hand side of~\eqref{eqn:log-crys} is a dualizable $W(k)$-module with dual given by the total fiber $\operatorname{tfib}_I\mathrm{R}\Gamma_\crys(\partial_IX/W(k))[2d]$. By Poincaré duality for logarithmic crystalline cohomology \cite{Tsuji}, the left-hand side of~\eqref{eqn:log-crys} is also dualizable with dual given by the logarithmic crystalline cohomology of $(X,\partial X)$ with compact supports $\mathrm{R}\Gamma_{\crys,c}((X,\partial X)/W(k))[2d]$. Thus, it suffices to produce an isomorphism
\[
	\mathrm{R}\Gamma_{\crys,c}((X,\partial X)/W(k)) \simeq \operatorname{tfib}_I\mathrm{R}\Gamma_\crys(\partial_IX/W(k)) \rlap,
\]
and this is done in \cite[(2.11.9.1)]{NS}.
\end{proof}

\begin{remark}
The quest for such an integral refinement of rigid cohomology has been a topic of interest in the recent literature.
Ertl, Shiho, and Sprang~\cite{ESS} constructed an integral $p$-adic cohomology theory under the assumption of resolution of singularities, which coincides with $\A^1$-colocalized crystalline cohomology by Proposition~\ref{prop:crys}(ii).
Merici~\cite{Mer} constructed such a theory without resolution of singularities,  and we expect that it coincides with $\A^1$-colocalized crystalline cohomology, but we do not pursue the comparison here; under resolution of singularities, it does, as he compared his theory with Ertl--Shiho--Sprang's.
\end{remark}

After inverting the characteristic exponent, we have better control of $\A^1$-colocalization.

\begin{proposition}\label{prop:invchar}
Let $k$ be a field of characteristic exponent $e$.
Then the lax symmetric monoidal inclusion
\[
	\MS_k^{\A^1}[1/e] \into \Mod_{\1_{\A^1}}(\MS_k)[1/e]
\]
is strict symmetric monoidal and preserves compact objects.
\end{proposition}

\begin{proof}
By definition, the inclusion carries the unit to the unit.
Recall that $\MS_k^{\A^1}[1/e]$ is generated under colimits and $\P^1$-desuspensions by smooth projective $k$-schemes, by~\cite[Proposition~B.1]{LYZ} (supplemented with \cite[Corollary 2.1.7]{ElmantoKhan} and \cite[Lemma 1.12]{Suslin2017} if $k$ is not perfect).
Hence, to prove that the inclusion is symmetric monoidal, it suffices to show that
\[
	(\L_{\A^1}\Sigma_{\P^1}^\infty X_+)\otimes_{\1_{\A^1}}(\L_{\A^1}\Sigma_{\P^1}^\infty Y_+)
	\simeq \L_{\A^1}\Sigma_{\P^1}^\infty (X\times Y)_+
\]
for smooth projective $k$-schemes $X$ and $Y$, but this follows from Corollary~\ref{cor:atiyah} and Lemma~\ref{lem:dual}.
Similarly, to prove that the inclusion preserves compact objects, it suffices to show that $\L_{\A^1}\Sigma_{\P^1}^\infty X_+$ is compact in $\Mod_{\1_{\A^1}}(\MS_k)$ for every smooth projective $k$-scheme $X$, which also follows from Lemma~\ref{lem:dual}.
\end{proof}

\begin{corollary}\label{cor:invchar}
Let $k$ be a field of characteristic exponent $e$.
Then the $\A^1$-colocalization functor
\[
	(\ph)^\dagger \colon \Mod_{\1_{\A^1}}(\MS_k)[1/e] \to \MS_k^{\A^1}[1/e]
\]
is lax symmetric monoidal and preserves colimits.
\end{corollary}

\begin{proof}
This is a formal consequence of Proposition~\ref{prop:invchar}.
\end{proof}

\begin{remark}
If $k$ is perfect and satisfies resolution of singularities, the inversion of the characteristic exponent is unnecessary in Proposition~\ref{prop:invchar} and Corollary~\ref{cor:invchar}, by \cite[Théorème 1.4]{Riou}.
\end{remark}

Over $\F_p$, crystalline cohomology arises as the graded pieces of the motivic filtration of $p$-complete topological periodic cyclic homology as established by Bhatt, Morrow, and Scholze~\cite[Theorem 1.12(4)]{BMS}.
Through $\A^1$-colocalization, we see that rigid cohomology arises from a motivic filtration by the same principle:

\begin{proposition}\label{prop:tp}
$\A^1$-colocalizing the motivic filtration of $\TP_p$ gives an exhaustive $\E_\infty$-multiplicative filtration $F^*\TP_p^\dagger[1/p]$ in étale sheaves of spectra over $\Sm_{\F_p}$ such that the graded piece $\gr_F^n\TP_p^\dagger[1/p]$ is isomorphic to rigid cohomology $\mathrm{R}\Gamma_\rig(\ph/\Z_p)[2n]$.
\end{proposition}

\begin{proof}
The motivic filtration of $\TP_p$ gives a diagram of motivic $\E_\infty$-ring spectra over $\mathbb{F}_p$ 
\[
\begin{tikzcd}
	& \tp_p \ar[rd, "(\ph)/\beta"] \ar[ld, "{(\ph)[\beta^{-1}]}" swap] & \\
	\TP_p & & \HH\Z^\crys_p \rlap,
\end{tikzcd}
\]
as well as an $\E_\infty$-$\tp_p$-algebra structure on the filtration $(\Sigma_{\P^1}^*\tp_p,\beta)$ (see Remark~\ref{rmk:filtered-E-infinity}).
Moreover, $\tp_p$ is an $\E_\infty$-$\1_{\A^1}$-algebra by Corollary~\ref{cor:tc}, so we can apply $\A^1$-colocalization.
If we apply $(\ph)^\dagger[1/p]$ to this diagram, then:
\begin{itemize}
\item It remains a diagram of motivic $\E_\infty$-ring spectra, as $(\ph)^\dagger[1/p]$ is lax symmetric monoidal by Corollary~\ref{cor:invchar}.
\item It becomes a diagram of étale motivic spectra by \cite[Theorem~14.3.4]{CD}.
\item $\HH\Z^\crys_p$ becomes the rigid cohomology spectrum $\HH\Q^\rig_p$ by Proposition~\ref{prop:crys}.
\item The right leg remains the quotient by $\beta$, as $(\ph)^\dagger$ commutes with $\P^1$-suspensions by Lemma~\ref{lem:dual}.
\item The left leg remains the inversion of $\beta$, as $(\ph)^\dagger[1/p]$ preserves filtered colimits by Corollary~\ref{cor:invchar}.
\end{itemize}
Therefore, the sequence
\[
	\cdots \xrightarrow{\beta} \Sigma_{\P^1}^{n+1}\tp_p^\dagger[1/p]
	\xrightarrow{\beta} \Sigma_{\P^1}^n\tp_p^\dagger[1/p] 
	\xrightarrow{\beta} \Sigma_{\P^1}^{n-1}\tp_p^\dagger[1/p] \xrightarrow{\beta} \cdots
\]
is an $\E_\infty$-multiplicative, exhaustive, and étale-local filtration of the motivic spectrum $\TP_p^\dagger[1/p]$, whose $n$th graded piece is $\Sigma_{\P^1}^n\HH\Q_p^\rig$. Taking $\Omega^\infty_{\P^1}$ gives the desired filtration.
\end{proof}

\begin{remark}\label{rmk:robalo}
	Let $k$ be a field of characteristic exponent $e$. By the proof of \cite[Corollary~4.12]{Robalo}, every $\A^1$-invariant $\KGL[1/e]$-module in $\MS_{k}$ has a canonical extension to a finitary localizing invariant of $k$-linear stable $\infty$-categories.
	In particular, $\TP_p^\dagger[1/p]$ extends to a localizing invariant of $\F_p$-linear stable $\infty$-categories. However, we do not know an a priori definition of such a localizing invariant.
\end{remark}

\begin{remark}
Presumably, $\A^1$-colocalization should also recover existing logarithmic cohomology theories in mixed characteristic.
Let $K$ be a $p$-adic field and $\mathcal{O}_K$ the ring of integers in $K$.
Let us consider the $\A^1$-colocalization of the $\1_{\A^1}$-module $\HH\Z_p^\syn$ in $\MS_{\mathcal{O}_K}$ representing the syntomic cohomology of smooth $\mathcal{O}_K$-schemes (see Example~\ref{ex:syn}).
For a smooth proper $K$-scheme $X$, suppose that there is a semi-stable model $\mathfrak{X}$ of $X$ over $\mathcal{O}_K$.
Then we expect that there is an isomorphism
\[
	(\Sigma_{\P^1}^n\HH\Z_p^{\syn,\dagger})(X) \simeq \mathrm{R}\Gamma_\syn(\mathfrak{X},\Z_p(n))[2n] \rlap,
\]
where the right hand side is the weight $n$ logarithmic syntomic cohomology of $\mathfrak{X}$.
In particular, the motivic spectrum $j^*(\HH\Z_p^{\syn,\dagger})$, where $j$ denotes the open embedding $\Spec(K)\into\Spec(\mathcal{O}_K)$, should give an integral refinement of Neková\v{r}--Nizio\l's arithmetic syntomic cohomology over $K$ defined in \cite{NN}.
\end{remark}

\section{Algebraic cobordism via Grassmannians}
\label{sec:lisse}

We show that the motivic spectrum $\MGL\in\MS_S$ is a colimit of Thom spectra of Grassmannians as in $\A^1$-homotopy theory, supplying the details to \cite[Remark 7.7]{AHI}.

We denote by
\[
\SSeq(\scr C) = \Fun(\Fin^\simeq, \scr C)
\]
the $\infty$-category of symmetric sequences in a symmetric monoidal $\infty$-category $\scr C$, which is a symmetric monoidal $\infty$-category under Day convolution.

Recall that the symmetric monoidal $\infty$-category $\MS_S=\Sp_{\P^1}(\scr P_{\Nis,\ebu}(\Sm_S,\Sp))$ is a left Bousfield localization of the symmetric monoidal $\infty$-category $\Sp_{\P^1}^\mathrm{lax}(\scr P_{\Nis,\ebu}(\Sm_S,\Sp))$ of \emph{lax symmetric $\P^1$-spectra}, which are by definition modules in $\SSeq(\scr P_{\Nis,\ebu}(\Sm_S,\Sp))$ over the free commutative algebra generated by the symmetric sequence $(0,\P^1,0,0,\dotsc)$ \cite[Definition 1.3.4]{AnnalaIwasa2}.

\begin{construction}
	\label{ctr:MGr}
	We construct a functor
	\[
	\MGr\colon (\Vect^\epi(S)_{/\scr O})^\op \to \CAlg(\MS_S)_{/\MGL},\quad\scr E\mapsto \MGr(\scr E).
	\]
	Given a finite locally free sheaf $\scr E$ on $S$, there is a lax symmetric monoidal functor
	\[
	\Fin^\simeq\to(\Sm_S)_{/\Vect}, \quad I\mapsto \Gr_{\lvert I\rvert}(\scr E^I).
	\]
	Indeed, it is a subfunctor of $I\mapsto \Gr(\scr E^I)$, which is the composition of the lax symmetric monoidal functors $I\mapsto\scr E^I$ and $\scr E\mapsto\Gr(\scr E)$. As everything is natural in $\scr E$, this defines a functor
	\begin{equation*}\label{eqn:MGr}
	\Vect^\epi(S)^\op \to \CAlg(\SSeq((\Sm_S)_{/\Vect})),\quad \scr E\mapsto (I\mapsto \Gr_{\lvert I\rvert}(\scr E^I)).
	\end{equation*}
	This functor sends $\scr O$ to $I\mapsto (S,\scr O^I)$, which is the free commutative algebra generated by the symmetric sequence $(\emptyset,(S,\scr O),\emptyset,\emptyset,\dotsc)$. Passing to slice categories, we thus get a functor
	\[
	(\Vect^\epi(S)_{/\scr O})^\op \to \CAlg(\Sp^\mathrm{lax}_{(S,\scr O)}((\Sm_S)_{/\Vect})).
	\]
	Composing with the Thom space functor $\Th\colon (\Sm_S)_{/\Vect}\to \scr P_\ebu(\Sm_S)_*$, which is symmetric monoidal \cite[Section 3]{AHI} and sends $(S,\scr O)$ to $\P^1$, we obtain the functor
	\[
	\MGr\colon (\Vect^\epi(S)_{/\scr O})^\op \to \CAlg(\Sp^\mathrm{lax}_{\P^1}(\scr P_\ebu(\Sm_S)_*))\to \CAlg(\MS_S).
	\]
	
	To see that this construction comes with a natural $\E_\infty$-map to $\MGL$, we observe that the symmetric sequence $I\mapsto \Gr_{\lvert I\rvert}(\scr E^I)$ maps to the sequence $I\mapsto\Vect_{\lvert I\rvert}$ in $\CAlg(\SSeq(\scr P(\Sm_S)_{/\Vect}))$ (which factors through $\Fin^\simeq\to\N$), which maps further to $I\mapsto \K_{\rk=\lvert I\rvert}$ in $\CAlg(\SSeq(\scr P(\Sm_S)_{/\K}))$. We then use the extension of the Thom space functor to K-theory: there is a commutative square of symmetric monoidal $\infty$-categories
 \[
 \begin{tikzcd}
 \scr P(\Sm_S)_{/\Vect} \ar{r}{\Th} \ar{d} & \scr P_{\Nis,\ebu}(\Sm_S)_* \ar{d} \\
 \scr P(\Sm_S)_{/\K} \ar{r}{\Th}  &  \Sp_{\P^1}(\scr P_{\Nis,\ebu}(\Sm_S)_*)\rlap,
 \end{tikzcd}
 \]
 inducing a commutative square
 \[
 \begin{tikzcd}
 \CAlg(\Sp^\mathrm{lax}_{(S,\scr O)}(\scr P(\Sm_S)_{/\Vect})) \ar{r}{\Th} \ar{d} & \CAlg(\Sp_{\P^1}(\scr P_{\Nis,\ebu}(\Sm_S)_*)) \ar{d}[sloped]{\sim} \\
 \CAlg(\Sp^\mathrm{lax}_{(S,\scr O)}(\scr P(\Sm_S)_{/\K})) \ar{r}{\Th}  &  \CAlg(\Sp_{\P^1}(\Sp_{\P^1}(\scr P_{\Nis,\ebu}(\Sm_S)_*)))\rlap.
 \end{tikzcd}
 \]
 The right vertical map is now an isomorphism, since $\Sp_{\P^1}$ is idempotent, and its inverse sends a $\P^1$-spectrum (in $\P^1$-spectra) to its $\emptyset$th term. The sequence $I\mapsto \K_{\rk=\lvert I\rvert}$ in $\CAlg(\Sp^\mathrm{lax}_{(S,\scr O)}(\scr P(\Sm_S)_{/\K}))$ is therefore sent to $\Th(\K_{\rk=0}\to\K)=\MGL$ in the upper right corner.
\end{construction}

\begin{proposition}[Grassmannian model for $\MGL$]
	\label{prop:MGr}
	Let $\scr E$ be a finite locally free sheaf on $S$ with an epimorphism $\scr E\onto\scr O$. Then the $\E_\infty$-map 
	\[
	\colim_{n\to\infty}\MGr(\scr E^n)\to\MGL
	\] 
	is an isomorphism in $\MS_S$.
\end{proposition}

\begin{proof}
	Let $\phi\colon \MGr(\scr E^\infty)\to\MGL$ be this map. 
	Forgetting the $\E_\infty$-ring structure, $\MGr(\scr E^\infty)$ is the colimit of the Thom spectra $\Sigma_{\P^1}^{-n}\Th_{\Gr_n(\scr E^\infty)}(\scr T_n)$, where $\scr T_n$ is the tautological sheaf, and $\MGL$ is similarly the colimit of the Thom spectra $\Sigma_{\P^1}^{-n}\Th_{\Vect_n}(\scr U_n)$, where $\scr U_n$ is the universal sheaf \cite[Proposition 7.1]{AHI}.
	One then obtains $\phi$ by applying the Thom spectrum functor to the forgetful map
	\[
	\Gr_n(\scr E^\infty) \to \Vect_n.
	\]
	This map becomes an isomorphism in $\MS_S$ by \cite[Theorem 5.3]{AHI}. Using the Thom isomorphism, it follows that for any oriented ring spectrum $E\in\CAlg(\h\MS_S)$, the map $\phi$ induces an isomorphism
	\[
	\phi^*\colon \Map_{\MS_S}(\MGL,E) \simto \Map_{\MS_S}(\MGr(\scr E^\infty),E).
	\]
	Since the ring spectrum $\MGr(\scr E^\infty)$ is oriented by \cite[Remark 6.10]{AHI}, the Yoneda lemma implies that $\phi$ is an isomorphism in $\MS_S$.
\end{proof}

\begin{remark}
	\leavevmode
	\begin{enumerate}
		\item After $\A^1$-localization, the statement of Proposition~\ref{prop:MGr} follows more directly from the fact that the Thom spectrum functor inverts unstable motivic equivalences \cite[Remark 16.11]{norms}. We do not know an analogue of this fact in our non-$\A^1$-invariant theory, which is purely stable.
		\item Construction~\ref{ctr:MGr} shows that the $\E_\infty$-structure on $\L_{\A^1}\MGL$ constructed by Panin, Pimenov, and Röndigs in \cite[Section 2.1]{Panin:2008} is canonically isomorphic to the one constructed by Bachmann and Hoyois in \cite[Section 16]{norms}.
	\end{enumerate}
\end{remark}

\begin{definition}
	Let $\scr C$ be a presentably symmetric monoidal $\infty$-category. 
	We denote by $\scr C^\dual\subset\scr C$ the full subcategory of dualizable objects and by $\scr C^\lisse\subset\scr C$ the full subcategory generated under colimits by $\scr C^\dual$. Objects of $\scr C^\lisse$ are called \emph{lisse}, and we denote by
	\[
	\lisse\colon \scr C\to\scr C^\lisse
	\]
	the right adjoint to the inclusion (which exists since $\scr C^\dual$ is small).
\end{definition}

\begin{observation}\label{obs:lisse-proj}
	\leavevmode
	\begin{enumerate}
		\item Let $\scr C$ and $\scr D$ be presentably symmetric monoidal $\infty$-categories and let $F\colon\scr C\to\scr D$ be a symmetric monoidal functor with a colimit-preserving right adjoint $G$. For any $X\in\scr C^\lisse$ and $Y\in\scr D$, the canonical map
	\[
	X\otimes G(Y) \to G(F(X)\otimes Y)
	\]
	is an isomorphism.
	\item Let $\scr C$ be a stable presentably symmetric monoidal $\infty$-category with $\1\in\scr C^\omega$. Then the functor $\lisse(\ph)$ preserves colimits.
	As a special case of (i), we have a natural isomorphism
	 \[
	 X\otimes\lisse(Y)\simeq \lisse(X\otimes Y)
	 \]
	 for any $X\in\scr C^\lisse$ and $Y\in\scr C$.
	\end{enumerate}
\end{observation}

\begin{example}\label{ex:KGL-lisse}
	The motivic spectrum $\KGL\in\MS_S$ representing algebraic K-theory is lisse. This follows from the Snaith presentation $\KGL\simeq\Sigma^\infty_{\P^1}\Pic_+[\beta^{-1}]$ \cite[Theorem 5.3.3]{AnnalaIwasa2}, the isomorphism $\Sigma^\infty_{\P^1}\Pic_+\simeq \Sigma^\infty_{\P^1}\P^\infty_+$ \cite[Theorem 5.3]{AHI}, and the fact that $\Sigma^\infty_{\P^1}\P^n_+$ is dualizable.
\end{example}

\begin{corollary}\label{cor:MGL-lisse}
	The motivic spectrum $\MGL\in\MS_S$ is lisse.
\end{corollary}

\begin{proof}
	Proposition~\ref{prop:MGr} presents $\MGL$ as a colimit of Thom spectra of Grassmannians $\Gr_{m}(\scr O^{nm})$, which are dualizable by Corollary~\ref{cor:atiyah}.
\end{proof}

\begin{corollary}[Homological Conner–Floyd isomorphism]
	\label{cor:conner-floyd}
	Let $S$ be a qcqs derived scheme. The orientation map $\MGL\to\KGL$ induces an isomorphism of bigraded multiplicative homology theories
	\[
	\MGL_{**}(\ph)\otimes_\L\Z[\beta^{\pm 1}]\simeq \KGL_{**}(\ph)\colon \MS_S\to \Ab^{\Z\times\Z},
	\]
	where $\L\to\Z[\beta^{\pm 1}]$ classifies the graded formal group law $x+y-\beta xy$.
\end{corollary}

\begin{proof}
	By the cohomological Conner–Floyd isomorphism \cite[Theorem~8.11]{AHI}, the map
	\[
	\MGL^{**}(\ph)\otimes_\L\Z[\beta^{\pm 1}] \to \KGL^{**}(\ph)\colon \MS_S^\op\to \Ab^{\Z\times\Z}
	\]
	is an isomorphism on $\Sm_S$ and hence on $\MS_S^\omega$. By duality, the map
	\[
	\MGL_{**}(\ph)\otimes_\L\Z[\beta^{\pm 1}]\to \KGL_{**}(\ph) \colon \MS_S\to \Ab^{\Z\times\Z}.
	\]
	is an isomorphism on $\MS_S^\dual$ and hence on $\MS_S^\lisse$. As both $\MGL$ and $\KGL$ are lisse (Corollary~\ref{cor:MGL-lisse} and Example~\ref{ex:KGL-lisse}), it follows from Observation~\ref{obs:lisse-proj}(ii) that this map is an isomorphism on all of $\MS_S$.
\end{proof}

\section{Motivic Landweber exactness}
\label{sec:LEFT}

We now prove the Landweber exact functor theorem in our setting, which is entirely similar to its incarnation in $\A^1$-homotopy theory \cite{Naumann:2009} (which is in turn similar to the classical theorem in ordinary homotopy theory).
To guide the reader, we note that the theorem only uses the following three facts from the theory of non-$\A^1$-invariant motivic spectra:
\begin{enumerate}
	\item the fact that the $\MGL$-cohomology ring of a scheme carries a formal group law, together with the computation of the $\MGL$-module $\MGL\otimes\MGL$ \cite[Proposition 7.9]{AHI}, which implies that the $\MGL$-homology of any motivic spectrum defines a quasi-coherent sheaf on the stack $\scr M_\fg$ of formal groups;
	\item the fact that $\MGL$ is lisse (Corollary~\ref{cor:MGL-lisse}), which implies that $\MGL$-homology is determined on dualizable motivic spectra;
	\item the countability of the $\infty$-category $\MS_\Z^\mathrm{dual}$ of dualizable motivic spectra over $\Z$ (Lemma~\ref{lem:countable}), which allows us to use the representability theorem of Adams to produce motivic spectra from homology theories.
\end{enumerate}

We start with some categorical preliminaries.

\begin{definition}
	Let $\scr C$ be a stable compactly generated $\infty$-category. A morphism $f\colon A\to B$ in $\scr C$ is called \emph{phantom} if, for every compact object $K\in \scr C^\omega$ and every map $g\colon K\to A$, the composite $f\circ g$ is nullhomotopic. We denote by $\w\scr C$ the $1$-category whose morphisms are those of $\h\scr C$ modulo phantom maps.
\end{definition} 

In other words, $\w\scr C$ is uniquely determined by the factorization
\[
\h\scr C\to\w\scr C\to \Fun(\scr C^{\omega,\op},\Ab),\quad E\mapsto [\ph,E],
\]
where the first functor is essentially surjective and full, and the second functor is faithful.
Note that the functor $\h\scr C\to\w\scr C$ is also conservative, since the above composite is. 

\begin{lemma}\label{lem:phantom}
	Let $\scr C$ be a commutative algebra in stable compactly generated $\infty$-categories. For a morphism $f\colon A\to B$ in $\scr C^\lisse$, the following are equivalent:
	\begin{enumerate}
		\item $f$ is phantom in $\scr C^\lisse$, i.e., induces the zero map $[\ph,A]\to[\ph,B]\colon \scr C^{\mathrm{dual},\op}\to\Ab$.
		\item $f$ induces the zero map $[\1,A\otimes(\ph)]\to [\1,B\otimes(\ph)]\colon \scr C^\lisse\to \Ab$.
		\item $f$ induces the zero map $[\1,A\otimes(\ph)]\to [\1,B\otimes(\ph)]\colon \scr C\to \Ab$.
	\end{enumerate}
\end{lemma}

\begin{proof}
	The equivalence of (i) and (ii) follows by duality and the fact that $[\1,A\otimes(\ph)]$ preserves filtered colimits. 
	That (ii) implies (iii) follows from Observation~\ref{obs:lisse-proj}(ii).
\end{proof}

\begin{warning}
	A phantom map in $\scr C^\lisse$ need not be phantom in $\scr C$. In the sequel, we consider phantom maps in $\MS_S^\lisse$ but never in $\MS_S$.
\end{warning}

\begin{corollary}\label{cor:phantom}
	Let $\scr C$ be a commutative algebra in stable compactly generated $\infty$-categories. Then:
	\begin{enumerate}
		\item The symmetric monoidal structure on $\scr C^\lisse$ descends to $\w\scr C^\lisse$.
		\item The lax symmetric monoidal functor 
		\[
		\scr C^\lisse\to \Fun(\scr C,\Ab),\quad E\mapsto [\1,E\otimes(\ph)],
		\]
		factors through $\w\scr C^\lisse$.
	\end{enumerate}
\end{corollary}

\begin{proof}
	To prove (i), we must show that if $f$ is a phantom map and $g$ is any map in $\scr C^\lisse$, then $f\otimes g$ is phantom. This claim as well as Assertion (ii) follow immediately from the characterization of phantom maps in Lemma~\ref{lem:phantom}(iii).
\end{proof}

\begin{lemma}\label{lem:phantom-PB}
	Let $\scr C$ and $\scr D$ be commutative algebras in stable compactly generated $\infty$-categories and let $F\colon \scr C\to\scr D$ be a symmetric monoidal functor with a colimit-preserving right adjoint $G$. Then $F$ sends phantom maps in $\scr C^\lisse$ to phantom maps in $\scr D^\lisse$, and hence induces a functor $\w\scr C^\lisse\to\w\scr D^\lisse$.
\end{lemma}

\begin{proof}
	If $A\in\scr C^\lisse$, then $G(F(A)\otimes(\ph))\simeq A\otimes G(\ph)$ as functors on $\scr D$ by Observation~\ref{obs:lisse-proj}(i).
	The claim then follows immediately using the characterization of phantom maps from Lemma~\ref{lem:phantom}(iii).
\end{proof}

For a cocomplete stable $\infty$-category $\scr C$ and a cocomplete abelian category $\scr A$, we denote by
\[
\Fun^\mathrm{hom}(\scr C,\scr A)\subset\Fun(\scr C,\scr A)
\]
the full subcategory of \emph{homological functors}, i.e., functors that preserve filtered colimits, finite products, and send cofiber sequences to exact sequences. 
If $\scr C$ and $\scr A$ have symmetric monoidal structures, then $\Fun(\scr C,\scr A)$ is also symmetric monoidal via the Day convolution. While the Day convolution need not preserve homological functors, $\Fun^\mathrm{hom}(\scr C,\scr A)$ inherits a structure of $\infty$-operad as a full subcategory of a symmetric monoidal $\infty$-category.

If $\scr C$ is presentably symmetric monoidal with $\1\in\scr C^\omega$, then $[\1,E\otimes(\ph)]\colon\scr C\to\Ab$ is a homological functor for any $E\in\scr C$. The representability theorem of Adams states that in some cases all homological functors $\scr C\to\Ab$ are of this form:

\begin{proposition}[Adams representability for homology theories]
	\label{prop:adams}
	Let $\scr C$ be a commutative algebra in stable compactly generated $\infty$-categories with $\scr C^\dual=\scr C^\omega$, and suppose that the latter $\infty$-category is countable. Then the lax symmetric monoidal functor
	\[
	\scr C\to \Fun(\scr C,\Ab),\quad E\mapsto [\1,E\otimes(\ph)],
	\]
	induces an isomorphism of $\infty$-operads
	\[
	\w\scr C\simeq \Fun^\mathrm{hom}(\scr C,\Ab).
	\]
\end{proposition}

\begin{proof}
	The fact that it induces an isomorphism of categories is a special case of \cite[Corollary 19(2)]{AdamsRep} (which is a modern exposition of a classical theorem of Adams \cite{AdamsBrown}; see also \cite[Theorem 5.1 and Proposition 4.11]{NeemanAdams} for another proof). The claim that it is in fact an isomorphism of $\infty$-operads means the following: for any $E_1,\dotsc,E_n, E\in\scr C$, the canonical map
	\[
	\Map_{\w\scr C}(E_1\otimes\dotsb\otimes E_n,E) \to \Map((E_1)_0(\ph)\otimes\dotsb\otimes (E_n)_0(\ph),E_0(\ph))
	\]
	is a bijection, where $E_0(\ph)=[\1,E\otimes(\ph)]$.
	We define an inverse as follows. By duality, a natural transformation $\alpha$ in the target can be viewed as a morphism of cohomology theories
	\[
	\alpha\colon E_1^0(\ph)\otimes \dotsb \otimes E_n^0(\ph)\to E^0(\ph)\colon \scr C^{\dual,\op}\to \Ab,
	\]
	where $E^0(\ph)=[\ph,E]$.
	Denoting by $\widehat E^0\colon \scr C^\op\to\Ab$ the extension of $E^0$ that preserves cofiltered limits, we have an isomorphism $\Map_{\w\scr C}(\ph,E)\simeq \widehat E^0(\ph)$ by \cite[Lemma 17]{AdamsRep}. Plugging in the objects $E_1,\dotsc,E_n$ into $\alpha$, we obtain an element
	\[
	\alpha(\id_{E_1},\dotsc,\id_{E_n}) \in \widehat E^0(E_1\otimes\dotsb\otimes E_n)\simeq \Map_{\w\scr C}(E_1\otimes\dotsb\otimes E_n,E).
	\]
	It is straightforward to check that the map $\alpha\mapsto \alpha(\id_{E_1},\dotsc,\id_{E_n})$ is the desired inverse.
\end{proof}

Let now $S$ be a derived scheme. To formulate the Landweber exact functor theorem, we will need to consider \emph{graded} homology theories on $\MS_S$.
To that end, let $\bb S$ be the free $\E_\infty$-group on one element (i.e., the sphere spectrum). The invertible object $\P^1$ in $\MS_S$ determines a symmetric monoidal functor $\Sigma_{\P^1}^*\1\colon \bb S\to\MS_S$, which induces a lax symmetric monoidal functor
\[
\MS_S \to \Fun^\mathrm{hom}(\MS_S,\Ab^\bb S),\quad E\mapsto E_*(\ph)= [\Sigma_{\P^1}^*\1,E\otimes (\ph)],
\]
where $\Ab^\bb S$ and $\Fun(\MS_S,\Ab^\bb S)$ are equipped with the Day convolution. Of course, the $\bb S$-graded functor $E_*(\ph)$ is completely determined by $E_0(\ph)$, as $E_*=E_0\circ \Sigma^{-*}_{\P^1}$.

\begin{remark}\label{rmk:Z-grading}
	In general, the $\bb S$-grading on $E_*(\ph)$ does not descend to a $\Z$-grading, as the automorphism of $E_2(\ph)$ induced by the swap map on $\P^1\otimes\P^1$ is not necessarily the identity. It does however descend to a $\Z$-grading if $E$ is orientable, by the naturality of the Thom isomorphism.
\end{remark}

If $S$ is qcqs, then $\MS_S$ is compactly generated with $\1\in\MS_S^\omega$. Moreover, if $f\colon T\to S$ is any morphism of qcqs schemes, then $f_*\colon \MS_T\to\MS_S$ preserves colimits. Applying Observation~\ref{obs:lisse-proj}, we obtain the following commutative squares:
\begin{equation}
\label{eqn:LEFT-lisse}
\begin{tikzcd}[ampersand replacement=\&]
	\MS_S^\lisse \ar{r} \ar[hook]{d} \& \Fun^\mathrm{hom}(\MS_S^\lisse,\Ab^\bb S) \ar[hook]{d}{(\ph)\circ \lisse} \\
	\MS_S \ar{r} \& \Fun^\mathrm{hom}(\MS_S,\Ab^\bb S)\rlap,
\end{tikzcd}
\qquad
\begin{tikzcd}[ampersand replacement=\&]
	\MS_S^\lisse \ar{r} \ar{d}[swap]{f^*} \& \Fun^\mathrm{hom}(\MS_S,\Ab^\bb S) \ar{d}{(\ph)\circ f_*} \\
	\MS_T^\lisse \ar{r} \& \Fun^\mathrm{hom}(\MS_T,\Ab^\bb S)\rlap.
\end{tikzcd}
\end{equation}

Let $\scr M_\mathrm{fg}$ denote the stack of (smooth, $1$-dimensional, connected, and commutative) formal groups and $\scr M_\mathrm{fg}^s$ that of formal groups with trivialized Lie algebra. The stack $\scr M_\fg^s$ is represented (as a presheaf on classical affine schemes) by the Hopf algebroid $(\L,\LB)$, where $\L$ is the Lazard ring and $\LB=\L[b_0,b_1,\dotsc]/(b_0-1)$. The usual grading on $(\L,\LB)$ defines an action of $\G_m$ on $\scr M_\mathrm{fg}^s$ such that $\scr M_\mathrm{fg}=\scr M_\mathrm{fg}^s/\G_m$. Thus, we have a cartesian square of faithfully flat maps
\begin{equation}\label{eqn:Mfg}
\begin{tikzcd}
	\Spec(\L) \ar{r} \ar{d} & \scr M_\mathrm{fg}^s \ar{d} \\
	\Spec(\L)/\G_m \ar{r} & \scr M_\mathrm{fg}\rlap.
\end{tikzcd}
\end{equation}

\begin{remark}\label{rmk:equivariant-FGL}
	If $R$ is a $\Z$-graded commutative ring, then any $\G_m$-equivariant map $\Spec(R)\to\scr M_\fg^s$ factors $\G_m$-equivariantly through $\Spec(L)$.
	Indeed, such a map classifies a graded formal group $\G=\mathrm{Spf}(A)$ over $R$ with a trivialization $\omega_{\G}\simeq R(1)$. Choosing a lift of $1\in R$ in $A_{-1}$ defines an isomorphism of graded $R$-algebras $R[[t]]\simeq A$, hence a graded formal group law over $R$.
\end{remark}

\begin{construction}\label{ctr:Phi-star}
Let $S$ be a qcqs derived scheme and let $X\in\MS_S$. As explained following \cite[Lemma 8.7]{AHI}, the $\Z$-graded abelian group $\MGL_*(X)$ has a structure of comodule over the $\Z$-graded Hopf algebroid $(\L,\LB)$, i.e., it is a quasi-coherent sheaf on $\scr M_\mathrm{fg}$, which we shall denote by $\MGL_\mathrm{fg}(X)$. This defines a lax symmetric monoidal homological functor
\[
\MGL_\mathrm{fg}\colon \MS_S \to \QCoh(\scr M_\mathrm{fg})^\heart.
\]
Accordingly, there is a lax symmetric monoidal functor
\[
\Phi_*\colon \QCoh(\scr M_\mathrm{fg})^\heart \to \Fun^{\mathrm{filt},\times}(\MS_S,\Ab^\Z), \quad \scr F \mapsto \Gamma(\scr M_\mathrm{fg}^s,\MGL_\mathrm{fg}(\ph)\otimes \scr F).
\]
When $\scr F$ is the pushforward of a graded $\L$-module $M$, we have
\begin{equation*}\label{eqn:MGL-tensor}
\Phi_*(\scr F) \simeq \MGL_*(\ph)\otimes_{\L} M
\end{equation*}
by the projection formula and base change for the cartesian square~\eqref{eqn:Mfg}. Also, if $\omega\in \Pic(\scr M_\mathrm{fg})$ is the pullback of the universal invertible sheaf on $\rm B\G_m$, then
\[
\Phi_n(\scr F) = \Gamma(\scr M_\mathrm{fg},\MGL_\mathrm{fg}(\ph)\otimes \scr F\otimes\omega^{\otimes n}).
\]
The functor $\Phi_*$ is thus determined by $\Phi_0$, since $\MGL_\mathrm{fg}(\ph)\otimes\omega^{\otimes n}=\MGL_\mathrm{fg}\circ \Sigma_{\P^1}^{-n}$.
\end{construction}

\begin{remark}\label{rmk:topological-to-motivic}
	In the context of Construction~\ref{ctr:Phi-star}, the unit section $1\in \Gamma(\scr M_\fg^s,\MGL_\fg(\1_S))$ defines a symmetric monoidal natural transformation 
	\[
	\Gamma(\scr M_\fg^s,\ph)\to \Gamma(\scr M_\fg^s,\MGL_\fg(\1_S)\otimes (\ph)) = \Phi_*(\ph)(\1_S)\colon \QCoh(\scr M_\mathrm{fg})^\heart\to \Ab^\Z.
	\]
	For example, evaluating this natural transformation on the structure sheaf of $\Spec(\L)/\G_m$ gives the graded ring homomorphism $\L\to\MGL_*(\1_S)$ classifying the formal group law of $\MGL$.
\end{remark}

\begin{lemma}\label{lem:LEFT}
	\leavevmode
	\begin{enumerate}
		\item For any $\scr F\in \QCoh(\scr M_\mathrm{fg})^\heart$, the functor $\Phi_*(\scr F)\colon \MS_S\to \Ab^\Z$ factors through the localization $\lisse\colon\MS_S\to\MS_S^\lisse$.
		\item For any morphism of qcqs derived schemes $f\colon T\to S$, there is a commuting triangle
		\[
		\begin{tikzcd}[row sep=5pt,column sep=1.5em]
			 & \Fun(\MS_S,\Ab^\Z) \ar{dd}{(\ph)\circ f_*} \\ 
			\QCoh(\scr M_\mathrm{fg})^\heart \ar[start anchor={[shift={(-10pt,0)}]north east},end anchor={[shift={(0,-5pt)}]west}]{ur}[above=1pt]{\Phi_*} \ar[start anchor={[shift={(-10pt,0)}]south east},end anchor={[shift={(0,5pt)}]west}]{dr}[below=1pt]{\Phi_*} & \\
			& \Fun(\MS_T,\Ab^\Z)\rlap.
		\end{tikzcd}
		\]
	\end{enumerate}
\end{lemma}

\begin{proof}
	This follows from the fact that $\MGL$ is lisse (Corollary~\ref{cor:MGL-lisse}) and the commutative squares~\eqref{eqn:LEFT-lisse}.
\end{proof}

\begin{construction}\label{ctr:Mod-flat}
	Consider the $\infty$-category $\Mod_\fg$ of pairs $(\scr X,\scr F)$ where $\scr X$ is a small presheaf on classical affine schemes with a map $\pi\colon\scr X\to\scr M_\fg$ and $\scr F\in \QCoh(\scr X)^\heart$; morphisms are contravariant in $\scr X$ and covariant in $\scr F$. It has a symmetric monoidal structure with
	\[
	(\scr X,\scr F)\otimes(\scr Y,\scr G)=(\scr X\times_{\scr M_\fg}\scr Y,\scr F\boxtimes\scr G),
	\]
	and there is a lax symmetric monoidal functor
	\[
	\Mod_\fg\to\QCoh(\scr M_\fg)^\heart,\quad (\scr X,\scr F)\mapsto\pi_*(\scr F).
	\]	
	Let $\Mod_\fg^\flat\subset \Mod_\fg$ be the full subcategory consisting of pairs $(\scr X,\scr F)$ such that:
	\begin{enumerate}
		\item $\scr X^s=\scr X\times_{\scr M_\fg}\scr M_\fg^s$ is affine;
		\item $\scr F$ is flat over $\scr M_\fg$.
	\end{enumerate}
	Note that Condition (i) excludes the unit $(\scr M_\fg,\scr O)$ of $\Mod_\fg$, but both conditions are preserved by binary tensor products. Thus, $\Mod_\fg^\flat$ is a nonunital symmetric monoidal subcategory of $\Mod_\fg$. More generally, $\Mod_\fg^\flat$ inherits the structure of an $\infty$-operad as a full subcategory of $\Mod_\fg$ (so that we can still talk about unital algebraic structures).

	The subcategory of $\Mod_\fg$ satisfying (i) can alternatively be described as the $1$-category of triples $(R,\G,M)$, where $R$ is a $\Z$-graded commutative ring, $\G$ is a $\G_m$-equivariant map $\Spec(R)\to\scr M_\fg^s$, and $M$ is a $\Z$-graded $R$-module. By Remark~\ref{rmk:equivariant-FGL}, any such $\G$ comes from a graded formal group law over $R$.
\end{construction}

\begin{lemma}\label{lem:exact}
	If $(\scr X,\scr F)\in \Mod_\fg^\flat$, then the functor
	\[
	\Gamma(\scr X^s,\pi^*(\ph)\otimes \scr F)\colon \QCoh(\scr M_\fg)^\heart\to \Mod_{\scr O(\scr X^s)}(\Ab^\Z)
	\]
	is exact.
\end{lemma}

\begin{proof}
	Let $\scr X^s=\Spec(R)$ and let $M$ be the graded $R$-module corresponding to $\scr F$. Since $\scr M_\fg^s$ has affine diagonal, we have a cartesian square of the form
	\[
	\begin{tikzcd}
		\Spec(A) \ar{d}[swap]{q} \ar{r}{f} & \Spec(L) \ar{d}{p} \\
		\Spec(R) \ar{r}{\pi} & \scr M_\fg^s\rlap,
	\end{tikzcd}
	\]
	where the vertical maps are faithfully flat. Since $q$ is faithfully flat and $f$ is affine, it suffices to show that the functor $f_*q^*(\pi^*(\ph)\otimes M)$ is exact.
	We have natural isomorphisms
	\[
	 f_*q^*(\pi^*(\ph)\otimes M)\simeq  f_*(f^*p^*(\ph)\otimes q^*(M)) \simeq p^*(\ph)\otimes f_*q^*(M).
	\]
	The assumption that $\scr F$ is flat over $\scr M_\fg$ means that $f_*q^*(M)$ is a flat $\L$-module. As $p$ is also flat, the above functor is exact.
\end{proof}

\begin{remark}
	Landweber's theorem \cite[Lecture 16]{Lurie:2010} gives a necessary and sufficient condition for a pair $(\scr X,\scr F)\in\Mod_\fg$ to satisfy Condition (ii) of Construction~\ref{ctr:Mod-flat}. For completeness, we recall it here. For each prime $p$, there is a canonical sequence of sections
	\[
	v_n\in \Gamma(\scr M_\fg^{\geq n},\omega^{\otimes (p^n-1)}), \quad n\geq 0,
	\]
	starting with $v_0=p$, where $\scr M_\fg^{\geq n}\subset\scr M_\fg$ is the vanishing locus of the sections $v_0,\dotsc,v_{n-1}$ (which depends on $p$).
	Denote by $\scr F^{\geq n}$ the restriction of $\scr F$ to $\scr X^{\geq n}=\scr X\times_{\scr M_\fg}\scr M_\fg^{\geq n}$. Then $\scr F$ is flat over $\scr M_\fg$ if and only if, for each prime $p$ and each $n\geq 0$, the map $v_n\colon \scr F^{\geq n}\to \scr F^{\geq n}\otimes\omega^{\otimes (p^n-1)}$ is injective.
	
	For a formal group law $F$ over a commutative ring $R$, we can take $v_n\in R$ to be the coefficient of $x^{p^n}$ in the $p$-series of $F$. An $R$-module $M$ is then flat over $\scr M_\fg$ if and only if $(v_0,v_1,\dotsc)$ is a regular sequence for $M$ (for all primes $p$).
\end{remark}

\begin{lemma}\label{lem:countable}
	Let $S$ be a qcqs derived scheme. Suppose that $S$ is countable, i.e., admits an open covering by spectra of animated rings with countable homotopy groups. Then the $\infty$-category $\MS_S^\omega$ is countable, i.e., its anima of objects and all its mapping anima have countable homotopy groups.
\end{lemma}

\begin{proof}
	If $\scr C$ is a presentable $\infty$-category and $E$ is a small collection of maps in $\scr C$, the left Bousfield localization $L_E\colon\scr C\to\scr C$ can be obtained as follows. For $F\in\scr C$, choose a surjection from a well-ordered set to the anima of all pairs $(f\colon X\to Y,X \to F)$ with $f$ in the smallest collection containing $E$ and closed under codiagonals, and let $TF$ be the transfinite composition of the pushouts along $f$. Iterate such a construction to obtain an ordinal sequence
	\[
	F\to TF\to T^2F\to\dotsb.
	\]
	Then it is clear that $L_EF=T^\kappa F$ if $E$ is contained in $\scr C^\kappa$.
	Suppose now that $\scr C$ is stable, that $E$ is closed under shifts and contained in $\scr C^\omega$, and that $\scr C^\omega$ is countable. 
	Then, if $F\in\scr C$ is such that the sets $[X,F]$ are countable for all $X\in\scr C^\omega$, the above description of $L_EF$ immediately shows that $[X,L_EF]$ is also countable for all $X\in\scr C^\omega$.

	We now apply this observation with $\scr C=\Sp_{\P^1}^\mathrm{lax}(\scr P(\Sm_S^\fp,\Sp))$. 
	The countability of $S$ implies that $\Sm_S^\fp$ is countable. As also $\Fin^\simeq$ and $\Sp^\omega$ are countable, $\SSeq(\scr P(\Sm_S^\fp,\Sp))^\omega$ is countable. If $A$ is a commutative algebra in a symmetric monoidal compactly generated $\infty$-category, then the $\infty$-category of $A$-modules is compactly generated, with compact objects generated under finite colimits and retracts by the free $A$-modules $A\otimes X$ with $X$ compact.
	In our situation, $A$ is the commutative algebra $(\1,\P^1,(\P^1)^{\otimes 2},\dotsc)$ in $\SSeq(\scr P(\Sm_S^\fp,\Sp))$, which is a sequential colimit of compact objects. Hence, for any compact symmetric sequences $X$ and $Y$, the set $[X,A\otimes Y]$ is countable. This shows that $\scr C^\omega$ is countable. Finally, $\MS_S$ is a left Bousfield localization of $\scr C$ at a collection of maps between compact objects, so that $\MS_S^\omega$ is countable.
\end{proof}

\begin{theorem}[Motivic Landweber exact functor theorem]
	\label{thm:LEFT}
	For any qcqs derived scheme $S$, there is a morphism of $\infty$-operads $\Phi$ making the triangle
	\[
	\begin{tikzcd}
		& \w\MS_S^\lisse \ar{d}{E\mapsto E_*(\ph)} \\
		\Mod_\fg^\flat \ar{r}[swap]{\Phi_*} \ar[dashed]{ur}{\Phi} & \Fun(\MS_S,\Ab^\bb S)
	\end{tikzcd}
	\]
	commute, which is natural in $S$ and uniquely determined as such. Moreover, $\Phi$ is a strict nonunital symmetric monoidal functor.
\end{theorem}

\begin{proof}
	The vertical functor is well-defined and lax symmetric monoidal by Corollary~\ref{cor:phantom}, and it lands in the subcategory of lisse-extended homological functors $\Fun^\mathrm{hom}(\MS_S^\lisse,\Ab^\bb S)$ by the first square in~\eqref{eqn:LEFT-lisse}.
	By Lemmas~\ref{lem:LEFT}(i) and~\ref{lem:exact}, the restriction of $\Phi_*$ to $\Mod_\fg^\flat$ also lands in this subcategory.
	By Proposition~\ref{prop:adams}, the functor
	\[
	\w\MS_S^\lisse \to \Fun^\mathrm{hom}(\MS_S^\lisse,\Ab), \quad E\mapsto E_0(\ph),
	\]
	is an isomorphism of $\infty$-operads if $\MS_S^\mathrm{dual}$ is countable, and this holds if $S$ is countable by Lemma~\ref{lem:countable}. Hence, for countable $S$, there is a unique lift $\Phi$ of $\Phi_0$ (and hence of $\Phi_*$) as indicated, which is automatically a morphism of $\infty$-operads.
	We then obtain such a lift for general $S$ using the commutativity of the second square in~\eqref{eqn:LEFT-lisse}, Lemma~\ref{lem:phantom-PB}, and Lemma~\ref{lem:LEFT}(ii).
	
	It remains to check that the lax nonunital symmetric monoidal structure of $\Phi$ is actually strict.
	Let $(\Spec(R_1)/\G_m,M_1)$ and $(\Spec(R_2)/\G_m,M_2)$ be objects of $\Mod_\fg^\flat$. Their tensor product is 
	\[
	(\Spec(R)/\G_m,M_1\boxtimes M_2),\quad \text{where}\quad \Spec(R)=\Spec(R_1)\times_{\scr M_\fg^s}\Spec(R_2).
	\]
	Let $E_1$, $E_2$, and $E$ be the images of these pairs by $\Phi$.
	We must show that the induced map $E_1\otimes E_2\to E$ in $\w\MS_S^\lisse$ is an isomorphism, and we may assume $S$ countable.
	In this case, it is by definition the unique map making the triangle
	\[
	\begin{tikzcd}
		E_{1*}(\ph) \otimes E_{2*}(\ph) \ar{r} \ar{d} & E_*(\ph) \\
		(E_1\otimes E_2)_*(\ph) \ar[dashed]{ur}
	\end{tikzcd}
	\]
	commute, where the horizontal map is the lax monoidal structure of $\Phi_*$.
	It will thus suffice to construct a natural isomorphism $(E_1\otimes E_2)_*(\ph)\simeq E_*(\ph)$ making the triangle commute.
	
	By Remark~\ref{rmk:equivariant-FGL}, we can choose $\G_m$-equivariant factorizations of $\Spec(R_i)\to\scr M_\fg^s$ through $\Spec(\L)$. They induce isomorphisms $E_{i*}(X) \simeq \MGL_*(X)\otimes_\L M_i$ and
	\[
	E_*(X) \simeq (\MGL\otimes\MGL)_*(X)\otimes_{\LB} (M_1\boxtimes M_2),
	\]
	since $(\MGL\otimes\MGL)_*(X)$ is by definition the pullback of $\MGL_\fg(X)$ to $\Spec(\LB)$.
	We then have a sequence of natural isomorphisms:
	\begin{align*}
		(E_1\otimes E_2)_*(X) & \simeq E_{1*}(E_2\otimes X) \\
		&\simeq M_1\otimes_\L\MGL_*(E_2\otimes X)\\
		& \simeq M_1\otimes_\L E_{2*}(\MGL\otimes X)\\
		& \simeq M_1\otimes_\L\MGL_*(\MGL\otimes X)\otimes_\L M_2   \\
		& \simeq M_1\otimes_\L(\MGL\otimes\MGL)_*(X)\otimes_\L M_2  \\
		& \simeq (\MGL\otimes\MGL)_*(X)\otimes_\LB(M_1\boxtimes M_2) \\
		& \simeq E_*(X).
	\end{align*}
	The commutativity of the above triangle can be checked using the following claim twice: if $F=\Phi(M)$ for some $\Z$-graded $\L$-module $M$, then the following square commutes for all $X,Y\in\MS_S$:
	\[
	\begin{tikzcd}[column sep=4em]
		F_*(X)\otimes \1_*(Y) \ar{d}[swap,sloped]{\sim} \ar{r}{\eta} & F_*(X\otimes Y) \ar{d}[sloped]{\sim} \\
		M\otimes_\L\MGL_*(X)\otimes \1_*(Y) \ar{r}{\id_M\otimes \eta} & M\otimes_\L\MGL_*(X\otimes Y).
	\end{tikzcd}
	\]
	Here, the maps $\eta$ are the evident ``assembly maps''. The commutativity of the square on some element $a\in \1_{*}(Y)$ is exactly the naturality of the isomorphism $F_*(\ph)\simeq M\otimes_\L \MGL_*(\ph)$ with respect to the map $\id_X\otimes a\colon \Sigma_{\P^1}^*X\to X\otimes Y$.
\end{proof}

\begin{remark}
	The multiplicative properties of $\Phi$ can also be stated as follows. If $\Mod_{\smash[b]{\fg}}^{\flat,+}$ denotes the symmetric monoidal full subcategory of $\Mod_\fg$ consisting of $\Mod_{\smash[b]{\fg}}^\flat$ and the unit $(\scr M_\fg,\scr O)$, then $\Phi$ extends uniquely to a symmetric monoidal functor $\Phi\colon \Mod_\fg^{\flat,+}\to\w\MS_S^\lisse$.
\end{remark}

\begin{example}[Graded formal group laws]
	\label{ex:graded-FGL}
	\leavevmode
	\begin{enumerate}
		\item We recover $\MGL$ as a commutative monoid in $\w\MS_S^\lisse$ from $(\Spec(\L)/\G_m,\scr O)\in\Mod_\fg^\flat$.
		\item A graded formal group law $F$ over a $\Z$-graded commutative ring $R$ defines a map $\Spec(R)/\G_m\to\scr M_\fg$ over $\rm B\G_m$.
		Any $\Z$-graded $R$-module $M$ that is flat over $\scr M_\fg$ then defines an $\MGL$-module $\Phi(R,F,M)$ in $\w\MS_S^\lisse$. If moreover $R$ itself is flat over $\scr M_\fg$, then $\Phi(R,F,M)$ is a module over the commutative $\MGL$-algebra $\Phi(R,F)$ in $\w\MS_S^\lisse$.
	\end{enumerate}
\end{example}

\begin{example}[Weakly periodic Landweber exact motivic spectra]
	\label{ex:weakly-periodic}
	Any formal group $(R,\G)$ classified by a flat map $\Spec(R)\to\scr M_\fg$ gives a commutative monoid $\Phi(R,\G)$ in $\w\MS_S^\lisse$. 
	These motivic spectra are weakly periodic in the following sense: if $L\in\Pic(R)$ is the pullback of the invertible sheaf $\omega$, then $\Phi(R,\G)_*(\ph)=\Phi(R,\G)_0(\ph)\otimes_R\bigoplus_{n\in \Z}L^{\otimes n}$.
	For example:
	\begin{enumerate}
		\item The universal formal group law over $\L$ gives the commutative monoid $\PMGL$.
		\item The multiplicative formal group law over $\Z$ gives the commutative monoid $\KGL$ (Corollary~\ref{cor:conner-floyd}).
		\item Let $k$ be a perfect field of positive characteristic and let $\G$ be a formal group of finite height over $k$. Then $\G$ admits a universal deformation $\hat\G$ defined over the \emph{Lubin–Tate ring} $\mathrm{LT}(\G)$, which is classified by a flat map $\Spec\mathrm{LT}(\G)\to\scr M_\fg$. The resulting commutative monoid $\rm E(\G)$ in $\w\MS_S^\lisse$ is the motivic analogue of the \emph{Morava E-theory} associated with $\G$.
	\end{enumerate}
\end{example}

\begin{remark}\label{rmk:LEFT-oriented}
	Let $F$ be a graded formal group law over a $\Z$-graded commutative ring $R$ and let $M$ be a $\Z$-graded $R$-module that is flat over $\scr M_\fg$.
	Since the motivic spectrum $\Phi(R,F,M)$ is an $\MGL$-module in $\w\MS_S^\lisse$, it admits an orientation that is canonical modulo phantom maps.
\end{remark}

\begin{remark}
	In \cite[Proposition 8.9]{Naumann:2009}, it is claimed that Landweber exact motivic spectra can be refined to $\MGL$-modules (in the $\infty$-category of motivic spectra), but the proof is flawed.
	The mistake originates in \cite[Proposition 7.9]{Naumann:2009}, where it is claimed that the homotopy groups of any $\MGL$-module have a structure of $(\L,\LB)$-comodule, following the analogous claim for $\MU$-modules made in \cite[Lemma 11]{MayWrong}. These claims are wrong, since the $(\L,\LB)$-comodule structure encodes in essence the descent data to the sphere spectrum.
	This invalidates \cite[Theorem 9.7]{Naumann:2009} (unless the motivic spectrum $F$ in \emph{loc.\ cit.}\ is a priori an $\MGL$-module), as well as the main result of \cite{Spitzweck:2012}, which uses this $\MGL$-module structure in an essential way.
\end{remark}

\section{Operations in algebraic K-theory and rational motivic cohomology}
\label{sec:HQ}

As applications of the motivic Landweber exact functor theorem, we compute the algebra of $\P^1$-stable operations in algebraic K-theory and we show that rational motivic cohomology is an idempotent algebra in $\MS_S$. These are non-$\A^1$-invariant enhancements of some of the results from \cite[Section~5.3]{RiouK}, \cite[Sections 9 and~10]{Naumann:2009}, and \cite[Section~14.1]{CD}, and our proofs are essentially the same.

Our first goal is to explicitly describe the endomorphism ring of $\KGL$.
We make some preliminary observations about Hopf algebroids. 

\begin{digression}[Dualizing and extending Hopf algebroids]
	Let $(A,\Gamma)$ be a cocategory object in commutative rings with left unit $\eta_L$, right unit $\eta_R$, counit $\epsilon$, and comultiplication $\Delta$. We denote by $\Gamma^\vee$ the $A$-linear of $\Gamma$, viewed as an $A$-module via $\eta_L$.
	Then $\Gamma^\vee$ is an associative algebra in $A$-bimodules, with unit $\epsilon^\vee\colon A\to \Gamma^\vee$ and multiplication $\circ$ given by
	\[
	f\circ g\colon \Gamma \xrightarrow{\Delta} \Gamma\otimes_A\Gamma \xrightarrow{\id\otimes g} \Gamma\otimes_AA\simeq \Gamma \xrightarrow{f} A.
	\]
	Assuming that $\Gamma$ is free (or more generally Mittag-Leffler) as a left $A$-module, we can regard $\Gamma^\vee$ as a topological left $A$-module with the inverse limit topology, and its right $A$-module and ring structures are then also continuous. Dualizing the commutative algebra structure of $\Gamma$ (in left $A$-modules), we obtain a structure of cocommutative coalgebra on $\Gamma^\vee$ (in topological left $A$-modules), whose comultiplication $\Gamma^\vee \to (\Gamma^\vee)^{\hatotimes_A n}$ is moreover right $A^{\otimes n}$-linear.
	There is a further compatibility between the algebra and coalgebra structures of $\Gamma^\vee$, but we do not spell it out.

	Let now $R$ be a commutative $A$-algebra. The data of a cocartesian morphism of cocategory objects \[(A,\Gamma)\to (R,R\otimes_A\Gamma)\] is equivalent to the data of a ring map $\rho\colon R\to R\otimes_A\Gamma$ that is a right coaction of $(A,\Gamma)$ on $R$: the left unit, counit, and comultiplication of $(R,R\otimes_A\Gamma)$ are extended from those of $(A,\Gamma)$, and the right unit is the coaction $\rho$. 
	Dualizing $\rho$, we find a left action of $\Gamma^\vee$ on $R$ given by the left $A$-linear map
	\[
	\lambda\colon\Gamma^\vee\otimes_A R\to R,\quad f\otimes r\mapsto f(r):=(\id_R\otimes f)(\rho(r)),
	\]
	which extends $\eta_L^\vee\colon\Gamma^\vee\to A$.
	Assuming $\Gamma$ free over $A$, this action is continuous (and it intertwines the algebra structure of $R$ and the coalgebra structure of $\Gamma^\vee$, but we do not spell this out). Furthermore, there is then an isomorphism of rings
	\begin{equation}\label{eqn:twisted-ring}
	R\hatotimes_A\Gamma^\vee \simto \Hom_A(\Gamma, R) = (R\otimes_A\Gamma)^\vee
	\end{equation}
	with the following multiplication on $R\hatotimes_A\Gamma^\vee$:
	\[
	(R\hatotimes_A\Gamma^\vee)\hatotimes_A (R\hatotimes_A\Gamma^\vee) \to  R\hatotimes_A\Gamma^\vee,\quad (r\cdot f)\otimes (s\cdot g) \mapsto r\cdot \Delta(f)(s)\circ g.
	\]
	In other words, this multiplication is left $R$-linear and right $\Gamma^\vee$-linear, and in the middle it is given by the composition
	\[
	\Gamma^\vee\hatotimes_A R\xrightarrow{\Delta\otimes\id} \Gamma^\vee_\ell\hatotimes_A\Gamma^\vee \hatotimes_A R \xrightarrow{\id\otimes \lambda} \Gamma^\vee_\ell\hatotimes_AR = R\hatotimes_A\Gamma^\vee,
	\]
	where the subscript $\ell$ means that the tensor product uses the left $A$-module structure of $\Gamma^\vee$.
	To make the isomorphism~\eqref{eqn:twisted-ring} more useful, we note that if $R_0\subset R$ is the equalizer of the left and right units, then $\Delta(f)(s)=s\cdot f$ for any $s\in R_0$. This means that the isomorphism~\eqref{eqn:twisted-ring} restricts to a ring map
	\[
	R_0\hatotimes_{A_0}\Gamma^\vee \to (R\otimes_A\Gamma)^\vee,
	\]
	where the multiplication on the left-hand side is computed componentwise. In particular, the map 
	\[
	\Gamma^\vee\to (R\otimes_A\Gamma)^\vee, \quad f\mapsto\id_R\otimes f,
	\]
	is a morphism of associative rings.
\end{digression}

We recall the structure of the $\G_m$-stack of multiplicative formal groups with trivialized Lie algebra.
It is presented by the graded Hopf algebroid $(\Z[\beta^{\pm 1}],\Gamma_m)$ with
\[
\Gamma_m=\Z[\beta^{\pm 1}]\otimes_\L\LB\otimes_\L\Z[\beta^{\pm 1}],
\]
where the ring map $\L\to\Z[\beta^{\pm 1}]$ classifies the graded formal group law $x+y-\beta xy$.
It is a standard fact that $\Gamma_m$ is a free $\Z[\beta^{\pm 1}]$-module on countably many generators. 
The associative ring $\Gamma_m^\vee$ can be described explicitly as follows: it is a twisted Laurent polynomial algebra over a sequential limit of rings
\begin{equation}\label{eqn:Gamma_m}
\Gamma_m^\vee=\lim\left(\dotsb \xrightarrow{\omega} \Z[[x]] \xrightarrow{\omega} \Z[[x]]\right)[\beta^{\pm 1}], \quad \omega(f)=(1-x)\frac{df}{dx},\quad \beta^{-1}a\beta = \omega(a),
\end{equation}
where the ring structure $\circ$ on the topological abelian group $\Z[[x]]$ is uniquely determined by the formula
\[
(1-x)^{-k}\circ (1-x)^{-l}=(1-x)^{-kl}
\]
for all $k,l\in\Z$. Note that $\Gamma_m^\vee$ is \emph{not} a $\Z[\beta^{\pm 1}]$-algebra, since $\beta$ is not central.\footnote{For this reason, the statement of \cite[Theorem 9.3]{Naumann:2009} is incorrect. We give a corrected statement in Proposition~\ref{prop:End(KGL)}(ii) below.}

\begin{remark}
	From the perspective of stable homotopy theory, these claims can be understood as follows. We have $\Gamma_m=\KU_{2*}\KU$, which is a free $\KU_{2*}$-module by \cite[Theorem 2.1]{AdamsClarke}, and its dual is $\Gamma_m^\vee=\KU^{2*}\KU$ with ring structure given by composition of endomorphisms. To obtain the above description of $\Gamma_m^\vee$, let $x=\beta c_1\in \KU^0(\rm B\C^\times)$, so that $\KU^{0}(\rm B\C^\times)\simeq \Z[[x]]$.\footnote{Following our algebro-geometric conventions, the element $1-x\in \KU^0(\rm B\C^\times)$ is the dual of the universal line bundle, and the Bott element $\beta\in \KU^{-2}(*)\subset \KU^0(\C\P^1)$ is the restriction of $x$ along the map $\C\P^1\to \rm B\C^\times$ classifying the tautological bundle.} The latter group can be identified with the subgroup of $\KU^{0}(\Omega^\infty\KU)$ consisting of additive maps, which is a ring under composition.
	Under this identification, the power series $(1-x)^{-k}$ corresponds to the (unstable) Adams operation $\psi^k$, and the formula for $\circ$ corresponds to the equation $\psi^k\circ\psi^l=\psi^{kl}$. Finally, the limit along $\omega$ corresponds to the Snaith presentation $\KU=\Sigma^\infty_+ \rm B\C^\times[\beta^{-1}]$ (and the $\lim^1$ obstruction to this limit computing $\KU^{0}(\KU)$ vanishes since $\KU^{-1}=0$).
\end{remark}

\begin{proposition}
	\label{prop:End(KGL)}
	Let $S$ be a qcqs derived scheme and $\Lambda\subset\Q$ a subring. 
	\begin{enumerate}[beginpenalty=10000]
		\item \textnormal{($\KGL$-homology cooperations)} There is a right coaction of the Hopf algebroid $(\Z[\beta^{\pm}],\Gamma_m)$ on the $\Z[\beta^{\pm 1}]$-algebra $\KGL_{\star}(S)$ and an isomorphism of $\Pic(\MS_S)$-graded Hopf algebroids
		\[
		(\KGL_{\Lambda\star},\KGL_{\Lambda\star}\KGL)\simeq \KGL_{\Lambda\star}(S) \otimes_{\Z[\beta^{\pm 1}]} (\Z[\beta^{\pm 1}],\Gamma_m),
		\]
		where on the right-hand side the right unit is the coaction while the left unit, counit, and comultiplication are extended.
		\item \textnormal{($\KGL$-cohomology operations)} There is a continuous 
		left action of the $\Z[\beta^{\pm 1}]$-bimodule algebra $\Gamma_m^\vee$ on the left $\Z[\beta^{\pm 1}]$-module $\KGL^{\star}(S)$ 
		extending $\eta_L^\vee\colon \Gamma_m^\vee\to \Z[\beta^{\pm 1}]$ 
		and an isomorphism of $\Pic(\MS_S)$-graded rings
		\[
		\KGL_{\Lambda}^{\star}\KGL\simeq \KGL_\Lambda^{\star}(S)\hatotimes_{\Z[\beta^{\pm 1}]} \Gamma_m^\vee,
		\]
		where the ring structure on the right-hand side is given by: $(u\cdot\phi)\circ(v\cdot\psi)=u\cdot\Delta(\phi)(v)\circ \psi$.
		In particular, there is a ring homomorphism $\Lambda\hatotimes\Gamma_m^\vee\to \KGL_{\Lambda}^{*}\KGL$, which is an isomorphism when $\K_0(S)\otimes\Lambda=\Lambda$.
	\end{enumerate}
\end{proposition}

\begin{proof}
	We use the symmetric monoidal functor $\Phi$ from Theorem~\ref{thm:LEFT}. Recall from Example~\ref{ex:weakly-periodic}(ii) that $\KGL$ is the image by $\Phi$ of the graded $\L$-algebra $\Z[\beta^{\pm 1}]$. The $\KGL$-module $\KGL\otimes\KGL$ is then the image by $\Phi$ of the $\Z[\beta^{\pm 1}]$-module $\Gamma_m$, so it is a sum of copies of $\KGL$ (indexed by a basis of $(\Gamma_m)_0$). 
	
	(i) As observed above, this statement is equivalent to the existence of a cocartesian morphism of Hopf algebroids $(\Z[\beta^{\pm 1}],\Gamma_m)\to(\KGL_{\star},\KGL_{\star}\KGL)$. 
	First of all, the Čech conerve of $\1\to\KGL$ in $\CAlg(\h\MS_S)$ does induce a Picard-graded Hopf algebroid $(\KGL_\star,\KGL_\star \KGL)$ by \cite[Lemma 8.7]{AHI},
	and since the decomposition of $\KGL\otimes\KGL$ does not involve any grading shifts, the inclusion of the $\bb S$-graded subalgebroid $(\KGL_*,\KGL_* \KGL)$ is cocartesian.
	Thus, it will suffice to define a cocartesian morphism $(\Z[\beta^{\pm 1}],\Gamma_m)\to(\KGL_{*},\KGL_{*}\KGL)$.
	The graded Hopf algebroid $(\Z[\beta^{\pm 1}],\Gamma_m)$ is the Čech conerve in $\CAlg(\Mod_\fg)$ of the map $\Spec\Z\to\scr M_\fg$ classifying $\hat\G_m$.
	Since $\Phi$ is symmetric monoidal, it sends this Čech conerve to that of $\1\to\KGL$ in $\CAlg(\w\MS_S^\lisse)$.
	Given the commutative triangle of Theorem~\ref{thm:LEFT}, we now obtain the desired morphism $(\Z[\beta^{\pm 1}],\Gamma_m)\to (\KGL_*,\KGL_*\KGL)$ by evaluating the symmetric monoidal natural transformation of Remark~\ref{rmk:topological-to-motivic} on the Hopf algebroid $(\Z[\beta^{\pm 1}],\Gamma_m)$.
		
	(ii) Since $\KGL\otimes\KGL$ is a free $\KGL$-module and since the unit in $\MS_S$ is compact, the canonical ring homomorphism
	\[
	\KGL^{\star}_\Lambda\KGL \to \Hom_{\KGL_{\star}}(\KGL_{\star}\KGL,\KGL_{\star}\otimes\Lambda)
	\]
	is an isomorphism. By inspection, the ring structure on the right-hand side is obtained by dualizing the Hopf algebroid from (i).
	The desired isomorphism of rings is thus an instance of~\eqref{eqn:twisted-ring}.
	The last statement follows from the equation $\Delta(\phi)(v)=v\cdot\phi$ for $v\in\Lambda$, which holds because $\Lambda$ is equalized by the left and right units.
\end{proof}

\begin{example}[Adams operations]
	\label{ex:Adams}
	Let $k\in\Z$. The $\E_\infty$-map $\Pic\to \Pic$, $\scr L\mapsto \scr L^{\otimes k}$, induces a morphism of $\E_\infty$-rings
	\[
	\psi^k\colon \Sigma^\infty_{\P^1}\Pic_+\to \Sigma^\infty_{\P^1}\Pic_+
	\]
	such that $\psi^k(\beta)=k\beta$ in $\KGL^{-1}(S)$. Thus, after inverting $\beta$ and $k$, we obtain an $\E_\infty$-map
	\[
	\psi^k\colon \KGL[\tfrac 1k]\to \KGL[\tfrac 1k],
	\]
	called the $k$th \emph{Adams operation}.
	In fact, since $k\mapsto (\ph)^{\otimes k}$ is an action of the multiplicative monoid $(\Z,\cdot)$ on the $\E_\infty$-monoid $\Pic$, we obtain for any submonoid $M\subset \Z$ an induced action of $M$ on $\KGL[M^{-1}]\in \CAlg(\MS_S)$ via Adams operations.
	
	Under the isomorphism of Proposition~\ref{prop:End(KGL)}(ii), the $k$th Adams operation $\psi^k\in\End(\KGL[\frac 1k])$ corresponds to the unique element $\psi^k\in (\Gamma_m^\vee)_0\hatotimes\Z[\frac 1k]$ such that $\omega^\infty(\psi^k)=(1-x)^{-k}$. Indeed, the latter power series in $\KGL^{0}[[x]]\simeq \KGL^{0}(\Pic)$ corresponds to the map $\Pic\to\K$, $\scr L\mapsto\scr L^{\otimes k}$. Explicitly, $\psi^k$ is given by the sequence
	\[
	\psi^k=\left(k^{-n}(1-x)^{-k}\right)_{n\in\N}
	\]
	in the limit~\eqref{eqn:Gamma_m} (which is how Riou defines the Adams operations in \cite[Definition 5.3.2]{RiouK}).
\end{example}

Our next goal is to investigate the case $\Lambda=\Q$ of Proposition~\ref{prop:End(KGL)} in more details, using the fact that the multiplicative formal group becomes isomorphic over $\Q$ to the additive one.
Analogously to the multiplicative case, the $\G_m$-stack of additive formal groups with trivialized Lie algebra is presented by the graded Hopf algebroid $(\Z[\beta^{\pm 1}],\Gamma_a)$ with
\[
\Gamma_a=\Z[\beta^{\pm 1}]\otimes_\L\LB\otimes_\L\Z[\beta^{\pm 1}],
\]
where the ring map $\L\to\Z[\beta^{\pm 1}]$ now classifies the graded formal group law $x+y$.
Unlike $\Gamma_m$, $\Gamma_a$ is not torsionfree and is more difficult to describe explicitly. For each prime $p$, $\Gamma_a\otimes\F_p$ is a double Laurent polynomial algebra over the even part of the dual Steenrod algebra at $p$. Rationally, however, we simply have
\[
\Gamma_a\otimes \Q=\Q[\beta^{\pm 1}]\otimes \Q[\beta^{\pm 1}].
\]
Viewing $\Gamma_a$ as a $\Z[\beta^{\pm 1}]$-module via the left unit and dualizing, we find
\begin{equation}\label{eqn:Gamma_a}
	(\Gamma_a\otimes \Q)^\vee=\Q^\Z[\beta^{\pm 1}], \quad \beta^{-1}a\beta = \sigma(a),
\end{equation}
as topological rings, where the ring structure on $\Q^\Z$ is given by pointwise multiplication and $\sigma\colon \Q^\Z\to\Q^\Z$ is the shift operator $(a_n)_{n\in \Z}\mapsto (a_{n+1})_{n\in \Z}$.

The graded formal group laws $x+y-\beta xy$ and $x+y$ over $\Q[\beta^{\pm 1}]$ are isomorphic via the power series
\[
-\beta^{-1}\log(1-\beta t) = \sum_{n\geq 1}\frac 1n \beta^{n-1}t^n\in \Q[\beta^{\pm 1}][[t]].
\]
Conjugation with this isomorphism defines an isomorphism of graded Hopf algebroids 
\[
(\Q[\beta^{\pm 1}],\Gamma_m\otimes \Q)\simeq (\Q[\beta^{\pm 1}],\Gamma_a\otimes \Q),
\]
hence an isomorphism of graded topological rings 
\begin{equation}\label{eqn:mult=add}
	(\Gamma_a\otimes \Q)^\vee\simeq (\Gamma_m\otimes \Q)^\vee.
\end{equation}

\begin{remark}\label{rmk:mult=add}
	In terms of the explicit descriptions \eqref{eqn:Gamma_m} and \eqref{eqn:Gamma_a} of the algebras $(\Gamma_m\otimes\Q)^\vee$ and $(\Gamma_a\otimes\Q)^\vee$, we have the following formula for the isomorphism~\eqref{eqn:mult=add}. First, there is an isomorphism of topological rings
		\[
		(\Q^\N,{\cdot})\simto (\Q[[x]],{\circ}),\quad (a_n)_{n\in \N}\mapsto \sum_{n=0}^\infty \frac{a_n}{n!}(-\log(1-x))^n.
		\]
	The inverse sends $f\in \Q[[x]]$ to $(a_n)_{n\in\N}$, where $a_n/n!$ is the $n$th coefficient of $f(1-\exp(x)^{-1})$. The power series $(1-x)^{-k}$ thus corresponds to $(k^n)_{n\in\N}$, which shows that the isomorphism is compatible with the ring structures.
	The shift operator $(a_n)_{n\in\N}\mapsto (a_{n+1})_{n\in \N}$ on $\Q^\N$ corresponds to the endomorphism $\omega$ of $\Q[[x]]$, which yields in the limit the desired isomorphism
		\[
		(\Gamma_a\otimes \Q)^\vee_0= \Q^\Z\simto (\Gamma_m\otimes \Q)^\vee_0,\quad (a_n)_{n\in \Z}\mapsto \left(\sum_{n=0}^\infty \frac{a_{n-k}}{n!}(-\log(1-x))^n\right)_{k\in \N}
		\]
	(cf.\ \cite[Proposition 5.3.7]{RiouK}).
\end{remark}

	Combining Proposition~\ref{prop:End(KGL)}(ii) with~\eqref{eqn:mult=add} and~\eqref{eqn:Gamma_a}, we obtain an isomorphism of rings 
	\[\Q^\Z\simto \End_{\h\MS_\Z}(\KGL_\Q).\] 
	For each $n\in \Z$, the characteristic function of $n$ thus defines an idempotent endomorphism $e_n$ of $\KGL_\Q$ over $\Spec\Z$, hence over any derived scheme $S$.
	We denote by $\KGL_\Q^\adams{n}$ the image of $e_n$, called the $n$th \emph{Adams eigenspace} of $\KGL_\Q$. 
	By Proposition~\ref{prop:Adams-dec}(i,iv) below, the motivic spectrum $\KGL_\Q^\adams{n}$ represents the $k^n$-eigenspace of the Adams operation $\psi^k$ on the rational $\K$-groups $\K_*(-)_\Q$ for any $k\in\Z-\{0,\pm 1\}$. Note that $\KGL_\Q^\adams{n}$ is by definition stable under arbitrary base change.

\begin{definition}\label{def:HQ}
	Let $S$ be a derived scheme.
	The \emph{rational motivic cohomology spectrum} $\HH\Q\in\MS_S$ is the motivic spectrum $\KGL_\Q^\adams{0}$ equipped with the $\E_0$-algebra structure $\1\to\KGL_\Q\to \KGL_\Q^\adams{0}$.
\end{definition}

\begin{remark}\label{rmk:E_0-retract}
	Under the isomorphism~\eqref{eqn:mult=add}, the map $\eta_L^\vee\colon (\Gamma_m\otimes\Q)^\vee_0\to \Q$ corresponds to $\eta_L^\vee\colon \Q^\Z=(\Gamma_a\otimes\Q)^\vee_0\to \Q$, which is evaluation at $0\in \Z$. In particular, it sends the idempotent $e_0$ to $1$. This implies that the endomorphism $e_0$ of $\KGL_\Q$ with image $\HH\Q$ is a morphism of $\E_0$-algebras in $\h\MS_S$ (this also follows from Lemma~\ref{lem:HQ=LQ} below).
\end{remark}

\begin{remark}
	If $S$ is regular noetherian, so that $\KGL\in\MS_S$ is $\A^1$-invariant, the spectrum $\KGL_\Q^\adams{n}$ coincides by construction with the one defined by Riou \cite[Definition 5.3.9]{RiouK} (the slight variations in our formulas are explained by a different choice of coordinate in our description of $\Gamma_m^\vee$: Riou uses the coordinate $u=(1-x)^{-1}-1$).
	In particular, $\HH\Q$ coincides in this case with the $\A^1$-invariant rational motivic cohomology spectrum defined in \cite[Definition 5.3.17]{RiouK} and studied further by Cisinski and Déglise in \cite[Section 14]{CD}.
\end{remark}

\begin{proposition}[Adams decomposition]
	\label{prop:Adams-dec}
	Let $S$ be a derived scheme.
	\begin{enumerate}
	\item The canonical map
	\[
	\bigoplus_{n\in \Z}\KGL_\Q^\adams{n} \to \KGL_\Q
	\]
	is an isomorphism in $\MS_S$.
	\item If $S$ is locally of finite Krull dimension, then the canonical map
	\[
	\KGL_\Q\to \prod_{n\in \Z}\KGL_\Q^\adams{n}
	\]
	is an isomorphism in $\MS_S$.
	\item For every $n\in \Z$, the Bott element $\beta$ induces an isomorphism $\Sigma_{\P^1}\KGL_\Q^\adams{n}\simeq\KGL_\Q^\adams{n+1}$. In particular, $\Sigma_{\P^1}^n\HH\Q\simeq\KGL_\Q^\adams{n}$.
	\item For every $n\in \Z$ and $k\in\Z-\{0\}$, we have $\psi^k=k^n\cdot\id$ on $\KGL_\Q^\adams{n}$.
	\end{enumerate}
\end{proposition}

\begin{proof}
	(i) This follows from the case $S=\Spec\Z$ proved in \cite[Theorem 5.3.10]{RiouK}. We give a proof for completeness. We may assume $S$ qcqs, so that $\MS_S$ is compactly generated. It then suffices to show that for every $X\in\MS_S^\omega$, the map
	\[
	\colim_n\left[X,\bigoplus_{i=-n}^n\KGL_\Q^\adams{i}\right] \to [X,\KGL_\Q]
	\]
	is an isomorphism. Injectivity is obvious. Let $e_{[-n,n]}$ be the idempotent endomorphism $\sum_{i=-n}^ne_i$ of $\KGL_\Q$, whose image is $\bigoplus_{i=-n}^n\KGL_\Q^\adams{i}$. To prove surjectivity, we must show that for every $f\colon X\to \KGL_\Q$, there exists an $n$ such that $f\simeq e_{[-n,n]}\circ f$. Recall that $\KGL\otimes\KGL_\Q$ is a countable sum of copies of $\KGL_\Q$. Since $\KGL\otimes X$ is a compact $\KGL$-module, the map $\id_{\KGL}\otimes f$ factors through a finite sum of copies of $\KGL_\Q$. This shows that the map
	\[
	\KGL^0_\Q(\KGL) \to \KGL^0_\Q(X),\quad \alpha\mapsto \alpha\circ f,
	\]
	is continuous, where the target has the discrete topology. Since the sequence $(e_{[-n,n]})_n$ converges to $1$ in $\KGL^0_\Q(\KGL)$, its image $(e_{[-n,n]}\circ f)_n$ in $\KGL^0_\Q(X)$ is eventually equal to $f$, as desired.
	
	(iii) By Proposition~\ref{prop:End(KGL)}(ii), there is a morphism of graded rings $\Q^\Z[\beta^{\pm 1}]\to \KGL^*_\Q\KGL$, where $\Q^\Z[\beta^{\pm 1}]$ is as described in~\eqref{eqn:Gamma_a}. The claim thus follows from the equation $e_{n+1}\beta=\beta e_n$ in the ring $\Q^\Z[\beta^{\pm 1}]$.
	
	(iv) By Example~\ref{ex:Adams} and Remark~\ref{rmk:mult=add}, $\psi^k$ is given by the sequence $(k^n)_{n\in\Z}$ in $\Q^\Z$.
	
	(ii) In light of (i), (iii), and (iv), the statement is equivalent to the following: for any animated ring $A$ of finite Krull dimension and any $n\in \Z$, there are only finitely many $i\in\Z$ such that $\K_n(A)_\Q^{\smash[t]{(i)}}\neq 0$. Here, the groups $\K_n(A)_\Q^{\smash[t]{(i)}}$ are the eigenspaces of the Adams operations on $\K_n(A)_\Q$. By the Bass delooping theorem \cite[Corollary 4.13]{AHI}, we can in fact ignore the negative K-groups and assume $n\geq 0$.
	In this case, we have $\K_n(A)_\Q^\adams{i}=0$ for any $i<0$, so it remains to prove that $\K_n(A)_\Q^\adams{i}=0$ for large enough $i$.
	Since $\K_{\geq 0}(\ph)_\Q^\adams{i}$ is finitary and the Zariski $\infty$-topos of $A$ has finite homotopy dimension \cite[Theorem 3.12]{ClausenMathew}, 
	this follows from Lemma~\ref{lem:relativeK}(ii) below.
\end{proof}	

\begin{lemma}\label{lem:relativeK}
	Let $i\in\Z$.
	\begin{enumerate}
		\item Let $R\to S$ be a morphism of animated commutative rings such that $\pi_0(R)\to\pi_0(S)$ is surjective with nilpotent kernel. Then the spectrum $\K(R,S)_\Q^\adams{i}$ is $i$-connective.
		\item Let $R$ be a local animated commutative ring. Then the spectrum $\K_{\geq 0}(R)_\Q^\adams{i}$ is $i$-connective.
	\end{enumerate}
\end{lemma}

\begin{proof}
	If $R$ is local and static, (ii) was proved by Soulé \cite[Section 2, Théorème 1]{Soule}. It thus remains to prove (i).
	Consider the zigzag of natural transformations on animated commutative rings
	\[
	(\K_{\geq 0})_\Q\xrightarrow{\mathrm{Tr}}\HC^-_\Q\xleftarrow{\mathrm{Nm}}\HC_\Q^+[1].
	\]
	We equip $(\K_{\geq 0})_\Q$ with the Adams filtration $(\K_{\geq 0})_\Q^{(\geq *)}$ and $\HC^-$ and $\HC^+$ with the HKR filtrations as defined in \cite[Section 6]{Raksit} or \cite[Section 6.3]{BhattLurie}. By definition of the latter, $\mathrm{Nm}$ is canonically a filtered map (where $[1]$ also shifts the filtration by $1$). We claim that $\mathrm{Tr}$ can also be refined (uniquely) to a filtered map.
	This follows from connectivity considerations as in \cite[Proposition 4.6]{ElmantoMorrow}. Indeed, since $(\K_{\geq 0})_\Q$ is left Kan extended from smooth $\Z$-algebras \cite[Example A.0.6]{EHKSY3}, it suffices to consider $\mathrm{Tr}$ on smooth $\Z$-algebras.
	We already know that $\K_{\geq 0}(\ph)_\Q^\adams{\geq i}$ is $i$-connective as a Zariski sheaf on classical schemes, by Soulé's theorem. On the other hand, for a smooth $\Z$-algebra $R$, the HKR filtration on $\HC^-(R)$ is exhaustive with graded pieces 
	\[\gr^i_\mathrm{HKR}\HC^-(R)=\Omega^{\geq i}_R[2i]\] 
	(by \cite[Example 6.3.8]{Raksit}), so that $\HC^-(R)/\Fil_\mathrm{HKR}^{i}\HC^-(R)$ is $(i-1)$-truncated. 
	This shows that $\mathrm{Tr}$ refines uniquely to a filtered map.
	
	By Goodwillie's theorem \cite[Theorem II.3.4 and Lemma I.3.3]{Goodwillie}, both transformations $\mathrm{Tr}$ and $\mathrm{Nm}$ are isomorphisms on $(R,S)$. In particular, since $\HC^+$ commutes with rationalization, we have
	\[
	\K(R,S)_{\Q}^\adams{i} \simto \K(R_\Q,S_\Q)_{\Q}^\adams{i}.
	\]
	We may thus assume that $R$ is a $\Q$-algebra. 
	In this case, we claim that $\mathrm{Tr}$ and $\mathrm{Nm}$ induce isomorphisms on associated graded
	\[
	\K(R,S)_{\Q}^\adams{i} \simeq \gr_\mathrm{HKR}^{i}\mathrm{HC}^-(R,S) \simeq \gr_\mathrm{HKR}^{i-1}\mathrm{HC}^+(R,S)[1].
	\]
	The first isomorphism follows from the compatibility of the filtered map $\mathrm{Tr}$ with the Adams operations $\psi^k$ \cite[Corollary B.19]{ElmantoMorrow} (see also the proof of \cite[Lemma 4.5]{Bou}), and the fact that the $i$th graded pieces are $k^i$-eigenspaces of $\psi^k$ (this forces $\mathrm{Tr}$ to be a diagonal matrix with respect to Adams decompositions).
	The second isomorphism follows from the nilinvariance of $\HP=\cofib(\mathrm{Nm})$ and the splitting of the HKR filtration thereon \cite[Theorem~3.4]{Bals}.
	The $i$th graded piece $\gr_\mathrm{HKR}^{i}\mathrm{HC}^+(R)$ is the Hodge-truncated derived de Rham cohomology $\rm L\Omega^{\leq i}_{R}[2i]$ (again by \cite[Example 6.3.8]{Raksit}), which is $i$-connective. Moreover, since $R\to S$ is surjective on $\pi_0$, the induced map on $\gr_\mathrm{HKR}^{i-1}\mathrm{HC}^+$ is surjective on $\pi_{i-1}$, so that the fiber $\gr_\mathrm{HKR}^{i-1}\mathrm{HC}^+(R,S)$ is still $(i-1)$-connective.
	By the above isomorphisms, we deduce that $\K(R,S)_\Q^\adams{i}$ is $i$-connective.
\end{proof}

Next, we observe that the Adams decomposition of $\KGL_\Q$ can also be obtained using the functor $\Phi$ from Theorem~\ref{thm:LEFT}.
Let $\Phi(\Q)\in \CAlg(\w\MS_S^\lisse)$ be the commutative monoid induced by the graded ring homomorphism $\L\to\Q$ classifying the additive formal group law over $\Q$, as in Example~\ref{ex:graded-FGL}(ii).

\begin{lemma}\label{lem:HQ=LQ}
	Let $S$ be a qcqs derived scheme.
	The summand inclusion $\KGL_\Q^\adams{n}\to \KGL_\Q$ is the image by $\Phi\colon \Mod_\fg^\flat\to \w\MS_S^\lisse$ of the morphism
	\[
	(\Spec(\Q)/\G_m,\hat\G_a,\Q(n)) \to (\Spec(\Q),\hat\G_a,\Q) \simeq (\Spec(\Q),\hat\G_m,\Q),
	\]
	In particular, there is an isomorphism of $\E_0$-algebras $\HH\Q\simeq\Phi(\Q)$.
\end{lemma}

\begin{proof}
	The idempotent endomorphism $e_n$ of $\KGL_{\Q*}(\ph)$ is by construction the image by $\Phi_*$ of the idempotent endomorphism of $(\Spec(\Q)/\G_m,\hat\G_a,\bigoplus_{n\in \Z}\Q(n))$ projecting onto $\Q(n)$.
\end{proof}

\begin{theorem}\label{thm:HQ}
	For any derived scheme $S$, the motivic $\E_0$-ring spectrum $\HH\Q\in\MS_S$ is idempotent. 
	Hence, it admits a unique $\E_\infty$-ring structure, and the forgetful functor $\Mod_{\HH\Q}(\MS_S)\to \MS_S$ is fully faithful.
\end{theorem}

\begin{proof}
	Since $\HH\Q$ is stable under base change, we may assume $S$ qcqs. We must show that the map \[\id_{\HH\Q}\otimes\eta\colon \HH\Q\to\HH\Q\otimes \HH\Q\] is an isomorphism, where $\eta\colon\1\to\HH\Q$ is the unit.
	By Lemma~\ref{lem:HQ=LQ}, it can be identified with the map \[\id_{\Phi(\Q)}\otimes\eta\colon \Phi(\Q)\to\Phi(\Q)\otimes\Phi(\Q).\]
	The latter is a section in $\w\MS_S^\lisse$ of the multiplication $\mu\colon \Phi(\Q)\otimes\Phi(\Q)\to\Phi(\Q)$. 
	Since the functor $\MS_S^\lisse\to\w\MS_S^\lisse$ is conservative, it suffices to show that $\mu$ is an isomorphism.
	Since $\Phi$ is a (nonunital) symmetric monoidal functor by Theorem~\ref{thm:LEFT}, $\mu$ is the image by $\Phi$ of the map
	\[
	\Q\otimes_\L\LB\otimes_\L\Q\to \Q
	\]
	classifying the identity automorphism of the additive formal group law over $\Q$. This map is an isomorphism, since the Hurewicz map $\eta_R\colon \L\to\Z\otimes_\L\LB$ is rationally an isomorphism by Lazard's theorem (see for example \cite[Lecture 2, Lemma 10]{Lurie:2010}).
\end{proof}

\begin{remark}
	The use of the motivic Landweber exact functor theorem in Proposition~\ref{prop:End(KGL)} is merely a convenience, as it would be possible to make that computation using only the Snaith presentation of $\KGL$. However, its use in the proof of Theorem~\ref{thm:HQ} seems more essential: we do not know an alternative argument.
\end{remark}

\begin{remark}[The $\E_\infty$-orientation of $\HH\Q$]
	\label{rmk:HQ-orientation}
	Being a retract of $\KGL_\Q$ as an $\E_0$-ring (Remark~\ref{rmk:E_0-retract}), the $\E_\infty$-ring $\HH\Q$ is canonically oriented.
	By \cite[Corollary 7.10]{AHI}, there is an isomorphism of $\HH\Q$-algebras $\HH\Q\otimes\MGL\simeq \HH\Q[b_1,b_2,\dotsc]$ in $\h\MS_S$.
	Since the action of the symmetric group $\Sigma_n$ on $\Sigma_{\P^1}^{n}\HH\Q$ is trivialized by the Thom isomorphism and $\Q\otimes\Sigma^\infty_+\rm B\Sigma_n\simeq \Q$, $\HH\Q[b_1,b_2,\dotsc]$ is actually the free $\E_\infty$-$\HH\Q$-algebra on the generators $b_i$. We therefore obtain a morphism of $\E_\infty$-rings
	\[
	\MGL \to \HH\Q\otimes\MGL \xleftarrow{\sim} \HH\Q[b_1,b_2,\dotsc] \to \HH\Q
	\]
	by sending each $b_i$ to $0$. 
\end{remark}

\begin{proposition}\label{prop:HQ-modules}
	The following are equivalent for a $\Q$-linear motivic spectrum $E\in \MS_S$:
	\begin{enumerate}
		\item $E$ lies in the full subcategory $\Mod_{\HH\Q}(\MS_S)\subset \MS_S$.
		\item $E$ admits a structure of $\MGL$-module in $\MS_S$.
		\item $E$ admits a structure of $\MGL$-module in $\h\MS_S$.
	\end{enumerate}
	In particular, every $\Q$-linear orientable object in $\CAlg(\h\MS_S)$ is an $\HH\Q$-module.
\end{proposition}

\begin{proof}
	(i) $\Rightarrow$ (ii) follows from Remark~\ref{rmk:HQ-orientation}, and (ii) $\Rightarrow$ (iii) is obvious.
	The subcategory $\Mod_{\HH\Q}(\MS_S)\subset \MS_S$ is a tensor ideal and is closed under retracts. To prove (iii) $\Rightarrow$ (i), it therefore suffices to show that the canonical map $\MGL_\Q\to \MGL\otimes\HH\Q$ is an isomorphism, and we may assume $S$ qcqs.
	The spectrum $\MGL_\Q$ is the image by $\Phi$ of the universal rational formal group law on the graded ring $\L\otimes\Q$.
On the other hand, by Lemma~\ref{lem:HQ=LQ} and Theorem~\ref{thm:LEFT}, $\MGL\otimes \HH\Q$ is the image by $\Phi$ of the graded $\L$-algebra $\LB\otimes_\L\Q$. The claim follows since the map $\L\otimes\Q\to \LB\otimes_\L\Q$ is an isomorphism.
\end{proof}

\begin{remark}[Periodic rational motivic cohomology and the Chern character]
	\label{rmk:chern}
	If $E\in\CAlg(\MS_S)$ is a $\Q$-linear orientable $\E_\infty$-ring, there is a canonical isomorphism
	\[
	\bigoplus_{n\geq 0} \Sigma_{\P^1}^n E \simto \Sym_E(\Sigma_{\P^1}E)
	\] 
	(see Remark~\ref{rmk:HQ-orientation}).
	Inverting the tautological element $u\colon \P^1\to \Sym_E(\Sigma_{\P^1}E)$, we obtain a further isomorphism
	\[
	\rm PE=\bigoplus_{n\in \Z} \Sigma_{\P^1}^n E \simto \Sym_E(\Sigma_{\P^1}E)[u^{-1}].
	\]
	In other words, the spectrum $\rm PE$ has a structure of $\E_\infty$-$E$-algebra, which is initial among $\E_\infty$-$E$-algebras $R$ with a unit $u\colon \P^1\to R$.
	By Proposition~\ref{prop:HQ-modules}, $\KGL_\Q$ is an $\E_\infty$-$\HH\Q$-algebra. Hence, there is a unique map of $\E_\infty$-$\HH\Q$-algebras
	\[
	\mathrm{ch}^{-1}\colon \rm P\HH\Q\to \KGL_\Q
	\]
	sending $u$ to $\beta$, which is an isomorphism by Proposition~\ref{prop:Adams-dec}(iii).
\end{remark}

\begin{example}[Integral étale motivic cohomology]
	\label{ex:HZetale}
	Let $\HH\Z^\et\in \MS_\Z$ be the étale localization of the motivic cohomology spectrum $\HH\Z$. For any derived scheme $S$, let $\HH\Z^\et\in \MS_S$ be its base change to $S$. Then:
	\begin{enumerate}
		\item For any prime $p$, the $p$-completion of $\HH\Z^\et$ is the motivic spectrum $\HH\Z_p^\et$ from Example~\ref{ex:syn} representing Bhatt–Lurie syntomic cohomology.
		\item The rationalization of $\HH\Z^\et$ is the rational motivic cohomology spectrum $\HH\Q$ from Definition~\ref{def:HQ} representing the eigenspaces of the Adams operations on the rational K-groups.
		\item $\HH\Z^\et$ satisfies étale descent, i.e., belongs to the full subcategory $\MS_S^\et\subset\MS_S$.
		\item Over a Dedekind domain $D$, $\HH\Z^\et$ is the étale localization of $\HH\Z$ in $\MS_D$, and $\Omega^{\infty-n}_{\P^1}\HH\Z^\et$ is the étale sheafification of the Bloch–Levine motivic complex $\Z(n)[2n]$ on $\Sm_D$.
	\end{enumerate}
	Indeed, (i) and (ii) are true over $\Spec(\Z)$ and hence over arbitrary $S$ as both $\HH\Z_p^\et$ and $\HH\Q$ are stable under base change. 
	We therefore have the following cartesian fracture square in $\CAlg(\MS_S)$:
	\[
	\begin{tikzcd}
		\HH\Z^\et \ar{r} \ar{d} & \prod_p \HH\Z_p^\et \ar{d} \\
		\HH\Q \ar{r} & \left(\prod_p\HH\Z_p^\et\right)_\Q\rlap.
	\end{tikzcd}
	\]
	To prove (iii), it remains to show that $(\prod_p\HH\Z_p^\et)_\Q$ satisfies finite étale descent.
	Since $\mathrm{R}\Gamma_\syn(\ph,\Z_p(*))$ is an étale sheaf of $\E_\infty$-rings, it has a unique structure of finite étale transfers satisfying the projection formula \cite[Corollary C.13]{norms}. Moreover, it is the constant sheaf $\Z_p$ in weight $0$. Hence, for a finite étale map $f\colon Y\to X$, the endomorphism $f_*f^*$ of $\mathrm{R}\Gamma_\syn(X,\Z_p(*))$ is multiplication by the degree of $f$. This remains true if we take the product over all primes, so that we obtain a finite étale sheaf after rationalizing.
	Over a Dedekind domain $D$, we see using the idempotence of $\HH\Q$ (Theorem~\ref{thm:HQ}) that the fracture cospan for $\HH\Z$ maps to the one for $\HH\Z^\et$, which yields a map $\HH\Z\to \HH\Z^\et$ in $\CAlg(\MS_D)$. The induced map $\Z(n)[2n]=\Omega_{\P^1}^{\infty-n}\HH\Z\to \Omega_{\P^1}^{\infty-n}\HH\Z^\et$ exhibits its target as the étale sheafification of its source, since this holds after $p$-completion (by Example~\ref{ex:syn}) and rationally (using that $(\rm L_\et\Z(*))_\Q$ is an étale sheaf on $\Sm_D^\fp$, by a transfer argument as above). This proves (iv).
\end{example}


\bibliographystyle{alphamod}
\bibliography{references}
    
\end{document}